\documentclass[10pt,a4paper]{article}

\usepackage[utf8]{inputenc}
\usepackage[T1]{fontenc}
\usepackage[a4paper,margin=2.5cm]{geometry}
\usepackage{amsmath}
\usepackage{amssymb}
\usepackage{amsthm}
\usepackage{amsfonts}
\usepackage{multirow}
\usepackage{array}
\usepackage{appendix}
\usepackage[svgnames]{xcolor}
\usepackage{hyperref}
\usepackage{ulem}
\usepackage{bbm}
\usepackage{xspace}
\hypersetup{
    colorlinks,
    citecolor=SeaGreen,
    filecolor=black,
    linkcolor=Indigo,
    urlcolor=black
}

\usepackage{mathtools}
\usepackage{dsfont}
\usepackage{graphicx}
\usepackage{fancyhdr}
\usepackage{geometry}
\usepackage{stmaryrd}
\usepackage{MnSymbol}
\usepackage{natbib}
\usepackage{empheq}

\theoremstyle{plain}
\newtheorem{theorem}{{Theorem}}[section] 
\newtheorem*{theorem*}{{Theorem}}

\newtheorem*{conjecture*}{Conjecture}
\newtheorem{proposition}[theorem]{Proposition}
\newtheorem*{proposition*}{Proposition}
\newtheorem{corollary}[theorem]{Corollary}
\newtheorem*{corollary*}{Corollary}
\newtheorem{lemma}[theorem]{Lemma}
\newtheorem*{lemma*}{Lemma}

\newtheorem*{assumption*}{Assumption}
\newtheorem{definition}[theorem]{Definition}
\newtheorem*{definition*}{Definition}

\theoremstyle{remark}
\newtheorem{example}{Example}
\newtheorem*{example*}{Example}
\newtheorem*{notation*}{Notation}
\newtheorem*{remark*}{Remark}
\newtheorem{remark}{Remark}

\newtheorem{innercustomgeneric}{\customgenericname}
\providecommand{\customgenericname}{}
\newcommand{\newcustomtheorem}[2]{%
  \newenvironment{#1}[1]
  {%
   \renewcommand\customgenericname{#2}%
   \renewcommand\theinnercustomgeneric{##1}%
   \innercustomgeneric
  }
  {\endinnercustomgeneric}
}
\newcustomtheorem{customexa}{Example}

\newcommand{\Cov}{\mathrm{Cov}}

\newcommand{\Card}{\mathrm{Card}}
\newcommand{\Aut}{\mathrm{Aut}}

\def\XS{\xspace}
\DeclareMathAlphabet{\mathb}{OML}{cmm}{b}{it}
\def\sbm#1{\ensuremath{\mathb{#1}}}                
  \def\ib{{\sbm{i}}\XS}
  \def\jb{{\sbm{j}}\XS}

\def\PM{\kern0pt^{\textrm{{\scriptsize PM}}}\kern0pt}
\def\MMAP{\kern1pt^{\textrm{{\tiny MMAP}}}\kern-1pt}

\def\projecttitle{
    
}
  
\RequirePackage{fancyhdr}
\title{Characterization of $U$-statistics on RCE matrices}

\author{Tâm \textsc{Le Minh}}
  
\graphicspath{ {./figures/} }

\begin{document}

\begin{center}
{\Large
	\textsc {Characterization of the asymptotic behavior of $U$-statistics on row-column exchangeable matrices}
}
\bigskip

 T\^am Le Minh $^{1,2}$
\bigskip

\textit{\noindent
$^{1}$ Univ. Grenoble Alpes, Inria, CNRS, Grenoble INP, LJK, 38000 Grenoble, France \\
$^{2}$ Universit\'e Paris-Saclay, AgroParisTech, INRAE, UMR MIA Paris-Saclay, 91120, Palaiseau, France 
}
\end{center}

\paragraph{Abstract.} 
 We consider $U$-statistics on row-column exchangeable matrices, arrays invariant to separate permutations of rows and columns and common in bipartite data. Under the standard dissociation assumption, we develop a graph-indexed analogue of the Hoeffding decomposition tailored to RCE dependence. We present a new decomposition based on orthogonal projections onto probability spaces generated by sets of Aldous-Hoover-Kallenberg variables. These sets are indexed by bipartite graphs, enabling the application of graph-theoretic concepts to describe the decomposition. This framework provides new insights into the characterization of $U$-statistics on row-column exchangeable matrices, particularly their asymptotic behavior, including in degenerate cases. Notably, the limit distribution depends only on specific terms in the decomposition, corresponding to non-zero components indexed by the smallest graphs, namely the principal support graphs. We show that the asymptotic behavior of a $U$-statistic is characterized by the properties of its principal support graphs. The number of nodes in these graphs (the principal degree) dictates the convergence rate to the limit distribution, with degeneracy occurring if and only if this number is strictly greater than 1. Furthermore, when the principal support graphs are connected, the limit distribution is Gaussian, even in degenerate cases. Applications to network analysis illustrate these findings.

\paragraph{Keywords:} degenerate $U$-statistics, row-column exchangeability, Hoeffding decomposition, asymptotic distribution, network statistics.

\section*{Introduction}

$U$-statistics generalize the empirical mean to functions of many variables. Given a sample of $n$ observations $(X_1, ..., X_n)$, a $U$-statistic is defined as
\begin{equation*}
    U_n = \binom{n}{k}^{-1} \sum_{1 \le i_1 < ... < i_k \le n} h(X_{i_1}, ..., X_{i_k}),
\end{equation*}
where the kernel $h : \mathbb{R}^k \rightarrow \mathbb{R}$ is a measurable symmetric function. $U$-statistics are a broad class of statistics encompassing many well-known examples such as the empirical variance, the Wilcoxon one-sample statistic, or Kendall's $\tau$. When the observations $(X_1, ..., X_n)$ are independent and identically distributed (i.i.d.), the properties of $U$-statistics are well understood. For general kernels, the central limit theorem (CLT) for $U$-statistics \citep{hoeffding1948class} ensures that the distribution of $\sqrt{n} (U_n - \theta)$ converges to a Gaussian distribution with variance $V$. This result becomes trivial in degenerate cases where $V = 0$. However, \cite{rubin1980asymptotic} showed that there exists an integer $2 \le d \le k$ such that the distribution of $n^{d/2} (U_n - \theta)$ converges to a non-trivial limit, which can be explicitly identified. A key tool in this analysis is the \textit{Hoeffding decomposition} \citep{hoeffding1961strong}, an orthogonal decomposition of $U_n$.

\paragraph{Motivation and applications.} The study of \textit{row-column exchangeable} (RCE) matrices is motivated by their relevance in network analysis. Many real-world datasets consist of relational data that are naturally represented as networks. In such representations, entities correspond to nodes, while their connections are depicted as links. A common structure is a bipartite network, where two distinct sets of nodes exist, and links exclusively connect nodes from different sets. Typical examples of bipartite network-structured data include recommender systems \citep{zhou2007bipartite}, scientific authorship networks \citep{newman2001structure} or ecological pollination networks \citep{dormann2009indices}. The adjacency matrix of a bipartite network is typically rectangular, with rows and columns representing the two distinct node sets. In an adjacency matrix $Y$, each matrix entry $Y_{ij}$ encodes the relationship between the $i$-th row node and the $j$-th column node. For binary networks, $Y_{ij} = 1$ if there is a link between nodes $i$ and $j$, and $Y_{ij} = 0$ otherwise. For weighted networks, $Y_{ij}$ denotes the weight of the link.

Exchangeability of nodes is a standard assumption in probabilistic network analysis. Many random network models satisfy node exchangeability, including the stochastic block models \citep{snijders1997estimation}, the expected degree distribution models \citep{picard2008assessing}, the graphon model \citep{lovasz2006limits}, and their bipartite extensions \citep{govaert2003clustering, ouadah2022motif, diaconis2008graph}. For bipartite networks, exchangeability implies that the distribution remains invariant under independent permutations of the two node sets. Therefore, the adjacency matrix of an exchangeable bipartite network corresponds to the leading rows and columns of an infinite RCE matrix. $U$-statistics on RCE matrices provide a general framework for defining network statistics. Among these, motif (or subgraph) counts are particularly well-studied and serve as key descriptors of network topology  \citep{stark2001compound, picard2008assessing, reinert2010random, bickel2011method, bhattacharyya2015subsampling, coulson2016poisson, gao2017testing, maugis2020testing, naulet2021bootstrap, ouadah2022motif, levin2025bootstrapping}. Motif counts have been widely used to analyze networks across diverse domains, including biology \citep{shen2002network, prvzulj2004modeling, przytycka2006important}, ecology \citep{bascompte2005simple, stouffer2007evidence, baker2015species, lanuza2023non} and sociology \citep{bearman2004chains, faust2006comparing, duma2014network, choobdar2012comparison}. 

\paragraph{$U$-statistics on RCE matrices.} In this paper, we address the asymptotic behavior of $U$-statistics on dissociated, row-column exchangeable matrices. An infinite matrix $Y$ is said to be \textit{row-column exchangeable} (RCE) if its probability distribution remains invariant under separate permutations of its rows and columns  \citep{aldous1981representations}, that is, for any pair of permutations $(\sigma_1, \sigma_2)$ of $\mathbb{N}$,
\begin{equation*}
    (Y_{\sigma_1(i)\sigma_2(j)})_{i,j} \overset{\mathcal{D}}{=} Y.
\end{equation*}
An RCE matrix is \textit{dissociated} if, for any $(m,n)$, $(Y_{ij})_{i \le m, j \le n}$ is independent of $(Y_{ij})_{i > m, j > n}$.

For $n > 0$, let $[ n ] := \{ 1, \dots, n\}$ and $\mathbb{S}_n$ denote the group of permutations of $[ n ]$. The kernels $h : \mathcal{M}_{p,q}(\mathbb{R}) \rightarrow \mathbb{R}$ considered are functions of a $p \times q$ submatrix of $Y$ that satisfy the following symmetry property: for all $(\sigma_1, \sigma_2) \in \mathbb{S}_p \times \mathbb{S}_q$, 
\begin{equation*}
h(Y_{(i_{\sigma_1(1)},...,i_{\sigma_1(p)};j_{\sigma_2(1)},...,j_{\sigma_2(q)})}) = h(Y_{(i_1,i_2,...,i_p;j_1,j_2,...,j_q)}),
\end{equation*} 
where $Y_{(i_1,...,i_p;j_1,...,j_q)}$ denotes the $p \times q$ submatrix of $Y$ corresponding to rows $i_1,...,i_p$ and columns $j_1,...,j_q$. For such symmetric kernels, the order of the indices in the submatrix is irrelevant. Hence, we use the simplified notation
\begin{equation*}
h(Y_{\{i_{1},...,i_{p}\};\{j_{1},...,j_{q}\}}) :=  h(Y_{(i_1,i_2,...,i_p;j_1,j_2,...,j_q)}).
\end{equation*} 

The associated $U$-statistic $U_{m,n}$ computed on the first $m$ rows and $n$ columns of an infinite matrix $Y$, is defined as
\begin{equation}
    U_{m,n} = \binom{m}{p}^{-1} \binom{n}{q}^{-1} \sum_{\substack{\ib \in \mathcal{P}_p([ m ])\\ \jb \in \mathcal{P}_q([ n ])}} h(Y_{\ib,\jb}),
    \label{eq:ustat}
\end{equation}
where for a set $A$ and an integer $k$, $\mathcal{P}_{k}(A)$ denotes the set of all the subsets of $A$ with cardinality $k$, and $Y_{\ib,\jb}$ represents the submatrix of $Y$ defined by the row indices in $\ib$ and column indices in $\jb$. A Hoeffding-type decomposition for these $U$-statistics has been introduced in \cite{leminh2025hoeffding}. This decomposition has been applied to derive a CLT and to construct a general estimator for the asymptotic variance of such $U$-statistics. However, this decomposition cannot identify the limiting distribution in degenerate cases.

\begin{customexa}{A}[Same-row co-engagement ($K_{1,2}$)]
Fix $p=1$ and $q=2$ and consider the kernel $h(y_{11},y_{12})=y_{11}y_{12}$, symmetric in its two column arguments. Our general definition~\eqref{eq:ustat} specializes to
\begin{equation*}
    U_{m,n} = \frac{2}{m n (n-1)} \sum_{i=1}^m \sum_{1\le j_1<j_2\le n} Y_{i j_1} Y_{i j_2},
\end{equation*}
the same-row co-engagement count: in recommender systems (users $\times$ items), it weights pairs of items co-engaged by the same user, in pollination networks (pollinators $\times$ plants), pairs of plants visited by the same pollinator. Under a dissociated RCE null with standardized residuals ($Y_{ij}\stackrel{\text{i.i.d.}}{\sim}\mathcal N(0,1)$), the smallest configuration that contributes to the variance is one row connected to two distinct columns (the 2-star motif $K_{1,2}$). Consequently, $U_{m,n}^h$ is degenerate of order $2$, and because this motif is connected, the centered, scaled statistic has a Gaussian limit with rate $N^{3/2}$ (we set $N:=m+n$ and take $m,n\to\infty$ with $m/N\to\rho\in(0,1)$). This foreshadows two general rules we establish later: the number of nodes in the smallest contributing configurations sets the rate, connected configurations yield Gaussian limits. We reuse Running example A throughout the paper to instantiate the definitions (projections, graph-indexed components) and the variance decomposition.
\end{customexa} 

\paragraph{Preview of running examples.} We revisit three running examples throughout the paper. \textbf{A: Same-row co-engagement $(K_{1,2})$,} introduces the key objects (orthogonal projections and graph-indexed components) and a connected, degenerate case in a simple Erd\H{o}s--Rényi-type network model. \textbf{B: Row heterogeneity,} applies the same machinery with row/column effects and tests whether the row-effect function is constant. \textbf{C: Overdispersion beyond row-column effects,} another degenerate case that tests for residual overdispersion not captured by row-column effects.

\paragraph{Contribution.}
We develop a new orthogonal decomposition for $U$-statistics on RCE matrices. The decomposition is built on the Aldous--Hoover--Kallenberg (AHK) representation \citep{hoover1979relations, aldous1981representations, kallenberg2005probabilistic}, linking our approach to bipartite exchangeable networks. Each term corresponds to a bipartite ``support graph'' (defined in Sec.~\ref{sec:degenerate:principal_part}), which lets us: (i) isolate the pieces that determine the limit, \textbf{principal support graphs}; (ii) read off the convergence rate from their common size, the \textbf{principal degree} (total node count); and (iii) obtain asymptotic normality whenever all principal support graphs are \textbf{connected}, even in Hoeffding-degenerate cases.  

This refines the projection-based approach of \cite{leminh2025hoeffding}, which aggregates graph topologies and therefore cannot separate degenerate higher-order fluctuations. The framework is closed under standard graph operations (intersection, inclusion, connectedness, isomorphism) and interfaces directly with bipartite network tools. We focus on regimes with connected principal support graphs. When they are disconnected, non-Gaussian limits can arise. A full taxonomy is left for future work.

\paragraph{Relation to prior work.} This new framework shares similarities with the generalized $U$-statistics studied by~\cite{janson1991asymptotic}, recently used in the works of \cite{kaur2021higher} and \cite{bhattacharya2023fluctuations}. However, these works primarily address motif counts in unipartite binary exchangeable networks. More precisely, \cite{kaur2021higher} investigated the asymptotic distribution of ``centered'' motif counts, which are not proper $U$-statistics directly computed as defined in~\eqref{eq:ustat}. They obtained a normal approximation theorem through Stein's method, but they did not consider degenerate cases. In contrast, \cite{bhattacharya2023fluctuations} examined the limit distribution of traditional motif counts in degenerate cases, depending on the properties of the network model. However, many other interesting statistics can also be expressed as network $U$-statistics, notably when the networks are weighted \citep{leminh2023ustatistics, leminh2025hoeffding}. Furthermore, the dependence structure induced by the AHK representation for bipartite networks differs from the unipartite setting, making the extension of results from unipartite to bipartite networks nontrivial, as shown in this paper.

\paragraph{Summary and outline.} The principal contribution of this paper is the development of a graph-based framework to characterize network $U$-statistics on bipartite exchangeable models. This framework is associated with an orthogonal decomposition of these $U$-statistics, offering a comprehensive characterization of their asymptotic properties. These properties are illustrated with some examples from statistical network analysis. Section~\ref{sec:degenerate:ahk} introduces the AHK representation of RCE matrices and Section~\ref{sec:degenerate:graph_subsets} presents the concept of graph sets of AHK variables, a key tool in this framework. These graph sets are used in Section~\ref{sec:degenerate:decomposition_probability_space}, which establishes the probability spaces that form the basis for the orthogonal decomposition of $U$-statistics on RCE matrices. Section~\ref{sec:degenerate:decomposition_u} formally defines the orthogonal decomposition. Section~\ref{sec:degenerate:variance} derives a variance decomposition for $U$-statistics on RCE matrices. Sections~\ref{sec:degenerate:principal_part} and~\ref{sec:degenerate:limit_distribution} link the decomposition to the limit distribution of $U$-statistics via the concept of \textit{principal support graphs}, which will be defined there. These sections demonstrate that the limit distribution is determined by the leading terms of the decomposition associated with these graphs. Section~\ref{sec:degenerate:gaussian_case} provides a sufficient condition on principal support graphs for obtaining a Gaussian limit. Section~\ref{sec:degenerate:running_examples} applies the previous result to identify the limit distribution of some statistics, used for hypothesis testing in the running examples. Finally, Section~\ref{sec:degenerate:other_asymptotic_frameworks} explores alternative asymptotic regimes and their implications for principal support graphs and Section~\ref{sec:asymmetric} and extends the results to asymmetric kernels. Theory readers may read end to end. Applied readers can skim Sections~\ref{sec:degenerate:decomposition_probability_space},~\ref{sec:degenerate:principal_part} and~\ref{sec:degenerate:gaussian_case} and rely on the three running examples, using the cheat sheet in Appendix~\ref{app:cheat_sheet} when needed.

\begin{table}[]
\begin{tabular}{|l|l|}
 \hline
$Y=(Y_{ij})$ & Infinite dissociated row-column exchangeable (RCE) matrix \\
$m,n$ & Numbers of rows and columns used to compute $U_{m,n}$ \\
$p,q$ & Kernel dimensions, size of the submatrix given to $h$ \\
$r,c$ & Graph dimensions, numbers of row and column nodes \\
$\ib, \jb$ & Sets of row and column node indices \\
$h:\mathcal M_{p,q}\to\mathbb R$ & Symmetric kernel in $p\times q$ entries \\
$U_{m,n}$ & $U$-statistic on the leading $m \times n$ submatrix of $Y$ \\
$U_{m,n}^h$ & Same as above, emphasizing the kernel $h$ \\
$N$ & Total size $m+n$ (used for rates) \\
$(m_N,n_N)$ & Sequence of matrix sizes in the asymptotic framework \\
$U_N^h$ & Shorthand for $U_{m_N,n_N}^h$ along the sequence $(m_N,n_N)$ \\
$\rho$ & $\lim_{N\to\infty} m_N/N \in (0,1)$ with $N=m_N+n_N$ \\

 \hline
    \end{tabular}
    \caption{Notations defined in the introduction.}
    \label{tab:notations_intro}
\end{table}

\section{Sets of Aldous--Hoover--Kallenberg variables}
\label{sec:degenerate:ahk_sets}

In this section, we introduce the Aldous--Hoover--Kallenberg representation of dissociated RCE matrices in a space generated by independent, identically distributed random variables. Then, we define graph-indexed sets of these variables, referred to as \textit{graph sets}. 

\subsection{Aldous--Hoover--Kallenberg representation of RCE matrices}
\label{sec:degenerate:ahk}

We use the Aldous--Hoover--Kallenberg (AHK) representation for dissociated RCE matrices \citep{hoover1979relations, aldous1981representations, kallenberg2005probabilistic}. If $Y$ is a dissociated RCE matrix, then there exist arrays of i.i.d. variables $(\xi_i)_{i \ge 1}$, $(\eta_j)_{j \ge 1}$ and $(\zeta_{ij})_{i,j \ge 1}$ with uniform distribution over $[0,1]$, and a real measurable function $\phi$ such that for all $1 \le i,j < \infty$, 
\begin{equation*}
    Y_{ij} \overset{a.s.}{=} \phi(\xi_i, \eta_j, \zeta_{ij}).
\end{equation*}
A function of entries of $Y$ can then be written using these AHK variables. In particular, the kernel $h(Y_{\ib, \jb})$, where $\ib \in \mathcal{P}_{p}(\mathbb{N})$ and $\jb \in \mathcal{P}_{q}(\mathbb{N})$, can be written as 
\begin{equation*}
    h(Y_{\ib, \jb}) = h \Big(\big(\phi(\xi_i, \eta_j, \zeta_{ij})\big)_{i \in \ib, j \in \jb}\Big) =: h_\phi\big((\xi_i)_{i \in \ib}, (\eta_j)_{j \in \jb}, (\zeta_{ij})_{i \in \ib, j \in \jb}\big),
\end{equation*} 
and $h_\phi : [0,1]^{p+q+pq} \rightarrow \mathbb{R}$ is a function, with the symmetry property 
\begin{equation*}
    h_\phi\big((\xi_{\sigma_1}(i))_{i \in \ib}, (\eta_{\sigma_2}(j))_{j \in \jb}, (\zeta_{\sigma_1(i)\sigma_2(j)})_{i \in \ib, j \in \jb}\big) = h_\phi\big((\xi_i)_{i \in \ib}, (\eta_j)_{j \in \jb}, (\zeta_{ij})_{i \in \ib, j \in \jb}\big),
\end{equation*}
for all permutations $\sigma_1$ and $\sigma_2$ of $\ib$ and $\jb$ respectively. The $U$-statistic with kernel $h$ defined by~\eqref{eq:ustat} can be rewritten with $h_\phi$ as follows
\begin{equation*}
    U_{m,n} = \left[ \binom{m}{p} \binom{n}{q} \right]^{-1} \sum_{\substack{\ib \in \mathcal{P}_p([ m ])\\ \jb \in \mathcal{P}_q([ n ])}} h_\phi\big((\xi_i)_{i \in \ib}, (\eta_j)_{j \in \jb}, (\zeta_{ij})_{i \in \ib, j \in \jb}\big).
\end{equation*}

With this formula, it becomes apparent that $U_{m,n}$ shares some similarities with the generalized $U$-statistics defined by~\cite{janson1991asymptotic}. Their generalized $U$-statistics are averages of random variables of the form $f((\xi_i)_{i \in \ib}; (\zeta_{ij})_{(i,j) \in \ib^2, i \neq j})$. Thus, although generalized $U$-statistics are adapted to unipartite random graphs, our bipartite setup changes the structure of the variables averaged in the $U$-statistics, which will lead to a different characterization.

For simplification, we will now write $h_{\ib, \jb} := h_\phi\big((\xi_i)_{i \in \ib}, (\eta_j)_{j \in \jb}, (\zeta_{ij})_{i \in \ib, j \in \jb}\big)$, so that
\begin{equation*}
    U_{m,n} = \left[ \binom{m}{p} \binom{n}{q} \right]^{-1} \sum_{\substack{\ib \in \mathcal{P}_p([ m ])\\ \jb \in \mathcal{P}_q([ n ])}} h_{\ib, \jb}.
\end{equation*}

\begin{remark}
Assume $Y_{ij}=\psi(\xi_i,\eta_j)$ (no edge-specific effect $\zeta_{ij}$). 
For any $(p,q)$ and any kernel $h:\mathcal{M}_{p,q}(\mathbb{R})\to\mathbb{R}$ that is symmetric under separate permutations of rows and columns, define
\[
h_\psi(x_1,\ldots,x_p;\,y_1,\ldots,y_q)
:= 
h\!\left(\big[\psi(x_r,y_s)\big]_{1\le r\le p,\;1\le s\le q}\right).
\]
Then for any index sets $\ib \in \mathcal{P}_{p}(\mathbb{N})$ and $\jb \in \mathcal{P}_{q}(\mathbb{N})$,
\[
h\!\left(Y_{\ib,\jb}\right)
=
h_\psi\!\big((\xi_i)_{i\in\ib};\,(\eta_j)_{j\in\jb}\big),
\]
so the associated statistic is a two-sample $U$-statistic with kernel size $(p,q)$ based on the independent samples $(\xi_i)_{i\ge1}$ and $(\eta_j)_{j\ge1}$. Classical two-sample results therefore apply \citep{lehmann1951consistency,lee1990u,korolyuk2013theory}.

In contrast, in the general case $Y_{ij}=\phi(\xi_i,\eta_j,\zeta_{ij})$, the edge effects $\{\zeta_{ij}\}$ are indexed by pairs $(i,j)$. Once $\ib$ and $\jb$ are chosen, the $p \times q$ edge noises are fixed to $(\zeta_{ij})_{(i,j)\in\ib\times\jb}$ and there is no independent averaging over a third sample. This coupling violates the independent-sample structure required by classical multi-sample $U$-statistics, so our object is not a three- (nor two-sample) $U$-statistic.
\end{remark}

\subsection{Graph sets of Aldous--Hoover--Kallenberg variables}
\label{sec:degenerate:graph_subsets}

The idea behind the new decomposition of a $U$-statistic is to find orthogonal projections first for $h_{\ib, \jb}$, for all $\ib$ and $\jb$, and then use the previous expression to derive the decomposition for $U_{m,n}$. To define the projections for $h_{\ib, \jb}$, we have to define the relevant subspaces for these projections. These subspaces, defined in the next section, are generated by subsets of AHK variables. To denote these subsets, we will be using a notation involving bipartite graphs. These graphs have no direct link with the network data, they are just a formalism to define subsets of AHK variables.

\subsubsection{Notations for bipartite graphs}

A bipartite graph $G$ is denoted $G=\big(V_1(G), V_2(G), E(G)\big)$, where $V_1(G)$ and $V_2(G)$ are the two sets of vertices and $E(G) \subseteq V_1(G) \times V_2(G)$ is the set of edges of $G$. We denote $v_1(G) = \Card\big(V_1(G)\big)$ and $v_2(G) = \Card\big(V_2(G)\big)$. A subgraph $F \subseteq G$ is such that $V_1(F) \subseteq V_1(G)$, $V_2(F) \subseteq V_2(G)$ and $E(F) \subseteq \big(V_1(F) \times V_2(F)\big) \cap E(G)$. We write $F \subset G$ if we have both $F \subseteq G$ and $F \neq G$.

Let $A = \{a_i : i \in I \}$ be a countable set indexed by $I$ and $\sigma$ some mapping $\sigma : I \rightarrow I$. We denote the action of $\sigma$ on $A$ by $\sigma A = \{ a_{\sigma(i)} : i \in I\}$. Let $G$ be a bipartite graph. Suppose that $V_1(G)$ is indexed by the set $I$ and $V_2(G)$ by the set $J$. The action of a pair of mappings $\Phi = (\sigma_1, \sigma_2)$ on $G$, where $\sigma_1 : I \rightarrow I$ and $\sigma_1 : J \rightarrow J$, is denoted 
\begin{equation}
    \Phi G := \big(\sigma_1 V_1(G), \sigma_2 V_2(G), \Phi E(G)\big),
\end{equation}
where $\Phi E(G) = \big\{ (x_{\sigma_1(i)}, y_{\sigma_2(j)}) : (x_i, y_j) \in E(G), (i,j) \in I \times J \big\}$. Among these mappings, the bijective ones are called permutations.

For two bipartite graphs $G_1$ and $G_2$ with same number of row nodes $r = v_1(G_1) = v_1(G_2)$ and column nodes $c = v_2(G_1) = v_2(G_2)$, we say that they are isomorphic if and only if there exists a pair of permutations $\Phi = (\sigma_1, \sigma_2) \in \mathbb{S}_r \times \mathbb{S}_c$ such that $\Phi G_1 = G_2$. In this case, we write $G_1 \sim G_2$. The number of elements $\Phi$ of $\mathbb{S}_r \times \mathbb{S}_c$ such that $\Phi G = G$ is the number of automorphisms of $G$, denoted $\lvert \Aut(G) \rvert$.

We define $K_{\ib, \jb} = (\ib, \jb, \ib \times \jb)$ the fully connected bipartite graph with row node set $\ib$ and column node set $\jb$. For $p \ge 0$ and $q \ge 0$, we denote $K_{p,q} = K_{[ p ], [ q ]}$. 

For $r \ge 0$ and $c \ge 0$, we can define a minimal set $\Gamma_{r,c}$ of all subgraphs of $K_{r,c}$ with $r$ row nodes and $c$ column nodes, such that every graph $G$ with the same numbers of nodes is isomorphic to exactly one element of $\Gamma_{r,c}$. Denote $\Gamma^-_{p,q} = \bigcup_{(0,0) < (r,c) \le (p,q)} \Gamma_{r,c}$. The notation $(0,0) < (r,c) \le (p,q)$ means $0 \le r \le p$, $0 \le c \le q$ and $(r,c) \neq (0,0)$. Every non-empty graph $G$ with $v_1(G) \le p$ and $v_2(G) \le q$ is isomorphic to exactly one element of $\Gamma^-_{p,q}$.

\subsubsection{Definition of graph sets}

Let $G$ be a bipartite graph. We can define the set $H(G)$ of AHK variables associated to $G$ as 
\begin{equation*}
    H(G) = \big((\xi_i)_{i \in V_1(G)}, (\eta_j)_{j \in V_2(G)}, (\zeta_{ij})_{(i,j)\in E(G)}\big).
\end{equation*}
$H(G)$ is composed of the variables $(\xi_i)$ indexed by the vertices in $V_1(G)$, $(\eta_j)$ indexed by the vertices in $V_2(G)$, and $(\zeta_{ij})$ indexed by the edges in $E(G)$. Figure~\ref{fig:degenerate:subgraph} right illustrates this construction. We see that $h_{\ib, \jb} = h_\phi\big((\xi_i)_{i \in \ib}, (\eta_j)_{j \in \jb}, (\zeta_{ij})_{i \in \ib, j \in \jb}\big) = h_\phi\big(H(K_{\ib, \jb})\big)$. In other words, $h_{\ib, \jb}$ belongs to some functional probability space generated by the AHK variables $H(K_{\ib, \jb})$. The subspaces on which  $h_{\ib, \jb}$ will be decomposed are generated by subsets of $H(K_{\ib, \jb})$, which are of the form $H(G)$, where $G \subset K_{\ib, \jb}$, as shown in Figure~\ref{fig:degenerate:subgraph}.

\begin{figure}[!tb]
\centering
\includegraphics[width=0.7\linewidth]{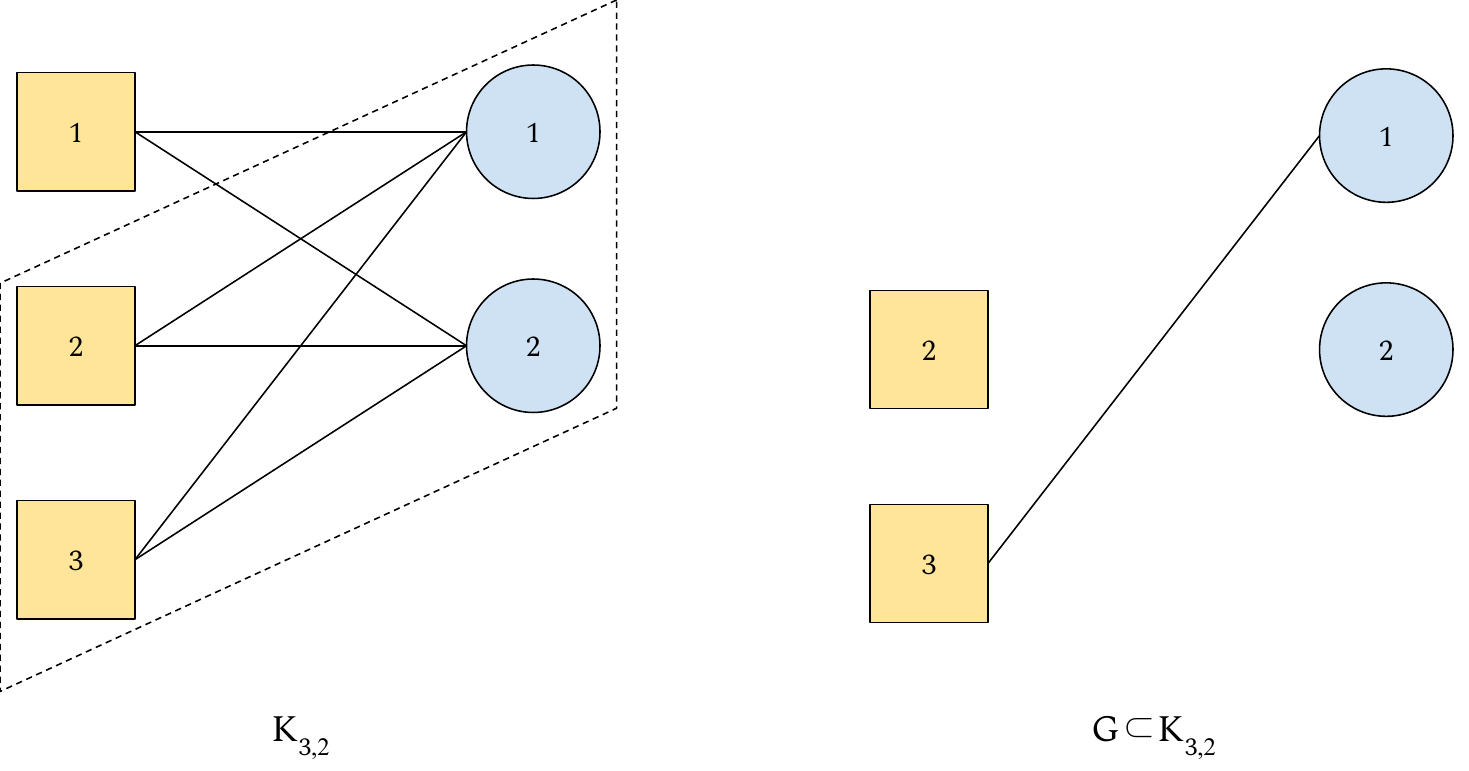}
\caption[A bipartite graph and one subgraph.]{A bipartite graph and one subgraph. For each graph, the row nodes are on the right and the column nodes are on the left. Left: the graph $K_{3,2}$. Right: a subgraph $G$ extracted from the row nodes $\{2, 3\}$ and the column nodes $\{1,2\}$ of $K_{3,2}$. Here, $G$ only keeps one edge among the four allowed between the row nodes $\{2, 3\}$ and the column nodes $\{1,2\}$. $G$ defines the subset $H(G) = \big(\{\xi_2, \xi_3\}, \{\eta_1, \eta_2\}, \{\zeta_{13}\}\big)$.}
\label{fig:degenerate:subgraph}
\end{figure}

In the following section, we define rigorously these subspaces and we exhibit some of their properties. This enables us to define a decomposition for $U$-statistics on RCE matrices.

\begin{table}[]
\normalsize
\begin{tabular}{|l|l|}
\hline
AHK variables & Latents variables $(\xi_i),(\eta_j),(\zeta_{ij})_{i,j}$ in the Aldous--Hoover--Kallenberg representation \\
$K_{\ib,\jb}$ & Complete bipartite graph with row set $\ib$ and column set $\jb$, edge set $\ib\times\jb$ \\
$K_{r,c}$ & Shorthand for $K_{[r],[c]}$ with $[r]=\{1,\dots,r\}$ and $[c]=\{1,\dots,c\}$ \\
Minimal set & Only one representative per isomorphism class in the set \\
$\Gamma_{r,c}$ & Minimal set of subgraphs of $K_{r,c}$ with $r$ row and $c$ column nodes \\
$\Gamma_{p,q}^{-}$ & Minimal set of subgraphs of $K_{p,q}$ with at most $p$ row and $q$ column nodes \\
$|V(G)|$ & Total nodes of $G$, equals $|V_1(G)|+|V_2(G)|$ ($r+c$ for $G\in\Gamma_{r,c}$) \\
$G_1 \sim G_2$ & $G_1$ and $G_2$ are isomorphic \\
$\mathrm{Aut}(G)$ & Automorphism group of $G$ \\
$H(G)$ & AHK variables attached to graph $G$: \\
& \qquad $H(G) = ((\xi_i)_{i\in V_1(G)},(\eta_j)_{j\in V_2(G)},(\zeta_{ij})_{(i,j)\in E(G)})$ \\
\hline
    \end{tabular}
    \caption{Principal notions and notations defined in Section~\ref{sec:degenerate:ahk_sets}.}
    \label{tab:notations_1}
\end{table}

\section{Orthogonal decomposition of $U$-statistics on RCE matrices}
\label{sec:degenerate:orthogonal_decomposition}

In this section, we propose a new orthogonal decomposition of the probability space generated by the AHK variables and compare it to the coarser decomposition of \cite{leminh2025hoeffding}. Then, we deduce the associated decomposition for $U$-statistics on RCE matrices and their variances.

\subsection{Decomposition of the probability space}
\label{sec:degenerate:decomposition_probability_space}

Let $G$ be a bipartite graph and denote $L_2(G)$ the space of all square-integrable random variables measurable with respect to $\sigma(H(G))$. $L_2(G)$ is an Hilbert space with inner product $\left<X, Y\right> = \mathbb{E}[ X Y]$. We investigate the following decomposition for $X \in L_2(G)$
\begin{equation}
    X = \sum_{F \subseteq G} p^F(X),
    \label{eq:degenerate:decomposition_l2}
\end{equation}
where the $p^F(X)$ are defined by recursion with $p^\emptyset(X) = \mathbb{E}[X]$ and for all $F$, 
\begin{equation*}
    p^F(X) = \mathbb{E}[X \mid H(F) ] - \sum_{F' \subset F} p^{F'}(X).
\end{equation*}

Now, we define $L_2^*(G) \subset L_2(G)$ as follows
\begin{equation}
    L_2^*(G) = \left\{ X \in L_2(G) : \mathbb{E}[X \mid H(F)] = 0, \forall F \subset G \right\}.
    \label{eq:degenerate:subspace_definition}
\end{equation}
These subspaces are linked to the decomposition~\eqref{eq:degenerate:decomposition_l2}. First, we show that each term of the decomposition belongs to one of these spaces, which shows that the decomposition is a decomposition on these subspaces. The following proposition can be shown by induction, as indicated in Appendix~\ref{app:degenerate:proof_decomposition_probability_space}.

\begin{proposition}
    For two bipartite graphs $F \subseteq G$ and $X \in L_2(G)$, $p^F(X) \in L_2^*(F)$.
    \label{prop:degenerate:l2_direct_sum}
\end{proposition}

Now, we prove the most important property of this decomposition. An Hoeffding-type decomposition is an orthogonal decomposition. The following theorem shows that this is the case and characterizes our decomposition.
\begin{theorem}
    For all bipartite graphs $G$, $L_2(G)$ is the orthogonal direct sum of the spaces $L_2^*(F)$, for all $F \subseteq G$. This is denoted by $$L_2(G) \bigoplus^\perp_{F \subseteq G} L_2^*(F).$$
    \label{th:degenerate:l2_ortho}
\end{theorem}

\begin{proof}
    Equation~\eqref{eq:degenerate:decomposition_l2} and Proposition~\ref{prop:degenerate:l2_direct_sum} already show that $L_2(G) \bigoplus_{F \subseteq G} L_2^*(F)$.
    We only have to show that for any two distinct bipartite graphs $G_1$ and $G_2$, we have $L_2^*(G_1) \perp L_2^*(G_2)$. Let $X_1 \in L_2^*(G_1)$ and $X_2 \in L_2^*(G_2)$. Let $\overline{G} = G_1 \cap G_2$. Since $G_1$ and $G_2$ are distinct, then at least one of the affirmations $\overline{G} \subset G_1$ and $\overline{G} \subset G_2$ is true. Assume that $\overline{G} \subset G_1$, then $\mathbb{E}[X_1 X_2] = \mathbb{E}[\mathbb{E}[X_1 X_2 \mid H(G_1)]] = \mathbb{E}[X_1 \mathbb{E}[X_2 \mid H(\overline{G})]] = 0$, so $L_2^*(G_1) \perp L_2^*(G_2)$.
\end{proof}

    \begin{remark}
        From this proof, we can see that $L_2^*(G)$ can also be characterized by the expression $L_2^*(G) = L_2(G) \cap (\cup_{F \subset G} L_2(F)^\perp)$.
    \end{remark}
    
\begin{example}[Projection for $K_{1,1}$]\label{ex:proj-K11}
$L_2(K_{1,1})$ is generated by $(\xi_1,\eta_1,\zeta_{11})$. Let
\[
F_1=\varnothing,\quad
F_2=K_{1,0},\quad
F_3=K_{0,1},\quad
F_4=\big(\{1\},\{1\},\varnothing\big)\ \text{(two nodes, no edge)},\quad
F_5=K_{1,1}.
\]
For each $F_i$, we write $H(F_i)$ for the associated AHK variables. For any $X\in L_2(K_{1,1})$, the recursion in \eqref{eq:degenerate:decomposition_l2} gives
\[
X=\sum_{i=1}^5 p^{F_i}(X),
\]
with
\begin{itemize}
  \item $p^{F_1}(X)=\mathbb E[X]$,
  \item $p^{F_2}(X)=\mathbb E[X\mid H(F_2)]-\mathbb E[X]=\mathbb E[X\mid \xi_1]-\mathbb E[X]$,
  \item $p^{F_3}(X)=\mathbb E[X\mid H(F_3)]-\mathbb E[X]=\mathbb E[X\mid \eta_1]-\mathbb E[X]$,
  \item $p^{F_4}(X)=\mathbb E[X\mid H(F_4)]-\mathbb E[X\mid H(F_2)]-\mathbb E[X\mid H(F_3)]+\mathbb E[X] =\mathbb E[X\mid \xi_1,\eta_1]-\mathbb E[X\mid \xi_1]-\mathbb E[X\mid \eta_1]+\mathbb E[X]$,
  \item $p^{F_5}(X)=\mathbb E[X\mid H(F_5)]-\mathbb E[X\mid H(F_4)]
        =\mathbb E[X\mid \xi_1,\eta_1,\zeta_{11}]-\mathbb E[X\mid \xi_1,\eta_1]
        = X-\mathbb E[X\mid \xi_1,\eta_1]$.
\end{itemize}
Each component lies in $L_2^*(F_i)$ by Proposition~\ref{prop:degenerate:l2_direct_sum}, and the family is pairwise orthogonal by Theorem~\ref{th:degenerate:l2_ortho}.
\end{example}

\subsection{Decomposition of $U$-statistics}
\label{sec:degenerate:decomposition_u}

For all $(0,0) \le (p,q) \le (m,n)$, $(\ib,\jb) \in \mathcal{P}_p([ m ]) \times \mathcal{P}_q([ n ])$, $G \subseteq K_{\ib, \jb}$, we can apply the decomposition~\eqref{eq:degenerate:decomposition_l2} on $h_{\ib, \jb} \in L_2(K_{\ib, \jb})$. 
\begin{equation*}
    p^{G}(h_{\ib, \jb}) = \mathbb{E}[h_{\ib, \jb} \mid H(G)] - \sum_{F \subset G} p^{F}(h_{\ib, \jb}),
\end{equation*}
where $p^{\emptyset}(h_{\ib, \jb}) = \mathbb{E}[h_{\ib, \jb}] = \mathbb{E}[h_{[ p ], [ q ]}]$. 

For all $G \subseteq K_{\ib, \jb}$, we remind that $V_1(G) \subseteq \ib$ and $V_2(G) \subseteq \jb$. Define $\overline{V_1(G)}$ and $\overline{V_2(G)}$ the complements of respectively $V_1(G)$ and $V_2(G)$ in respectively $\ib$ and $\jb$. In fact, the term $p^{G}(h_{\ib, \jb})$ does not depend on the elements of $\overline{V_1(G)}$ and $\overline{V_2(G)}$, i.e. even if $(\ib_1, \jb_1) \neq (\ib_2, \jb_2)$, as long as $G \subset K_{\ib_1, \jb_1} \cap K_{\ib_2, \jb_2}$, we have $p^{G}(h_{\ib_1, \jb_1}) = p^{G}(h_{\ib_2, \jb_2})$. Therefore, we use the notation $p^{G} := p^{G}(h_{\ib, \jb})$, for all $G \in K_{\ib, \jb}$. From Equation~\eqref{eq:degenerate:decomposition_l2}, we can write
\begin{equation*}
    h_{\ib, \jb} = \sum_{G \subseteq K_{\ib, \jb}} p^{G},
\end{equation*}
and the $U$-statistic $U_{m,n}$ can be rewritten
\begin{equation*}
\begin{split}
    U_{m,n} &= \binom{m}{p}^{-1} \binom{n}{q}^{-1} \sum_{\substack{\ib \in \mathcal{P}_p([ m ])\\ \jb \in \mathcal{P}_q([ n ])}} \sum_{G \subseteq K_{\ib, \jb}} p^{G} \\
    &= \binom{m}{p}^{-1} \binom{n}{q}^{-1} \sum_{\substack{\ib \in \mathcal{P}_p([ m ])\\ \jb \in \mathcal{P}_q([ n ])}} \sum_{(0,0) \le (r,c) \le (p,q)} \sum_{\substack{G \subseteq K_{\ib, \jb} \\ (v_1(G), v_2(G)) = (r,c)}} p^{G} \\
    &= \sum_{(0,0) \le (r,c) \le (p,q)} P^{r,c}_{m,n},
\end{split}
\end{equation*}
where 
\begin{equation*}
    P^{r,c}_{m,n} = \binom{m}{p}^{-1} \binom{n}{q}^{-1} \sum_{\substack{\ib \in \mathcal{P}_p([ m ])\\ \jb \in \mathcal{P}_q([ n ])}} \sum_{\substack{G \subseteq K_{\ib, \jb} \\ (v_1(G), v_2(G)) = (r,c)}} p^{G}.
\end{equation*}

Note that in general, for $G \subseteq K_{\ib, \jb}$, $p^{G}$ is not symmetric, that means $p^{G}(h_{\sigma_1 \ib, \sigma_2 \jb}) \neq p^G(h_{\ib, \jb})$ for a pair of permutations $(\sigma_1, \sigma_2) \in \mathbb{S}_p \times \mathbb{S}_q$. We define $\bar{p}^{G}$ the symmetrized version of $p^{G}$ as
\begin{equation*}
    \bar{p}^{G} = \sum_{(\sigma_1, \sigma_2) \in \mathbb{S}_p \times \mathbb{S}_q} p^{G}(h_{\sigma_1 \ib, \sigma_2 \jb}) = \sum_{\Phi \in \mathbb{S}_p \times \mathbb{S}_q} p^{\Phi G} = \sum_{\substack{G' \subseteq K_{\ib,\jb} \\ G' \sim G}} p^{G'}.
\end{equation*}

For two isomorphic subgraphs $G_1$ and $G_2$ of $K_{\ib, \jb}$, we have $\bar{p}^{G_1} = \bar{p}^{G_2}$ by symmetry. There is exactly one element $G \in \Gamma_{r,c}$, where $r = v_1(G_1) = v_1(G_2)$ and $c = v_2(G_1) = v_2(G_2)$, which is isomorphic to both $G_1$ and $G_2$. Therefore, for all $(\ib, \jb) \in \mathcal{P}_p([ m ]) \times\mathcal{P}_q([ n ])$, we can index these quantities with the graph $G \in \Gamma_{r,c}$ instead of $G \in K_{\ib, \jb}$. Then, we denote
\begin{equation*}
    \widetilde{p}^G_{\ib, \jb} := \bar{p}^{G'},
\end{equation*} 
where $G \in \Gamma_{r,c}$ and $G'$ is any subgraph of $K_{\ib, \jb}$ which is isomorphic to $G$. We can also denote $\widetilde{p}^G$ the function $\widetilde{p}^G : (\ib, \jb) \longmapsto \widetilde{p}^G_{\ib, \jb}$.

Because there are $r! \binom{p}{r} c! \binom{q}{c} \lvert \Aut(G) \rvert^{-1}$ distinct subgraphs of $K_{\ib, \jb}$ that are isomorphic to $G \in \Gamma_{r,c}$, we obtain the following alternative decomposition
\begin{equation*}
    h_{\ib, \jb} = (p!q!)^{-1} \sum_{G \subseteq K_{\ib, \jb}} \bar{p}^{G} = \sum_{0 \le (r,c) \le (p,q)} \sum_{G \in \Gamma_{r,c}} \frac{1}{(p-r)! (q-c)! \lvert \Aut(G) \rvert} \widetilde{p}^G_{\ib, \jb}
\end{equation*}
and 
\begin{equation*}
     P^{r,c}_{m,n} = \sum_{G \in \Gamma_{r,c}} \frac{1}{(p-r)! (q-c)! \lvert \Aut(G) \rvert} \widetilde{P}^{G}_{m,n},
\end{equation*}
where for all $G \in \Gamma_{r,c}$, 
\begin{equation*}
    \widetilde{P}^{G}_{m,n} = \binom{m}{p}^{-1} \binom{n}{q}^{-1} \sum_{\substack{\ib \in \mathcal{P}_p([ m ])\\ \jb \in \mathcal{P}_q([ n ])}} \widetilde{p}^{G}_{\ib, \jb}
\end{equation*}
is the $U$-statistic of kernel $\widetilde{p}^G$. Finally, the $U_{m,n}$ can be rewritten as
\begin{equation}
    U_{m,n} = \sum_{0 \le (r,c) \le (p,q)} \sum_{G \in \Gamma_{r,c}} \frac{1}{(p-r)! (q-c)! \lvert \Aut(G) \rvert} \widetilde{P}^{G}_{m,n}.
    \label{eq:decomposition_ustat}
\end{equation}

\begin{remark}
    This decomposition is related to the one defined by \cite{leminh2025hoeffding}. The latter consists of an orthogonal projection of $h_{\ib, \jb} \in L_2(K_{\ib, \jb})$ on the subspaces $(\underline{L}_2(K_{\ib', \jb'}))_{\ib' \subseteq \ib, \jb' \subseteq \jb}$, where
\begin{equation}
    \underline{L}_2(K_{\ib, \jb}) = \{ X \in L_2(K_{\ib, \jb}) : \mathbb{E}[X \mid H(K_{\ib', \jb'})] = 0, \forall \ib' \subseteq \ib, \jb' \subseteq \jb \}.
    \label{eq:degenerate:old_decomposition}
\end{equation}
Comparing this with the subspaces~\eqref{eq:degenerate:subspace_definition}, we see that the decomposition on the subspaces of the form~\eqref{eq:degenerate:old_decomposition} is coarser, as they only consist in subspaces generated by graphs of the form $K_{\ib, \jb}$. For this reason, it does not capture the subtleties determining the limit distribution of degenerate $U$-statistics. We will see that the decomposition given by equation~\eqref{eq:decomposition_ustat} can fill this gap, at the cost of being more complex.
\end{remark}

\begin{table}[]
\normalsize

\raggedright{\textbf{Probability space decomposition}} 
\begin{equation*}
    L_2(G) \bigoplus^\perp_{F \subseteq G} L_2^*(F)
\end{equation*}

\begin{tabular}{|l|l|} 
\hline
$L_2(G)$ & Space of square-integrable variables measurable with respect to $\sigma(H(G))$ \\
$L_2^\ast(G)$ & Subspace of $L_2(G)$ with $\mathbb E[X\mid H(F)]=0$ for all $F\subset G$ \\

$p^F$ & Orthogonal projection on $L_2^\ast(F)$: for $X\in L_2(G)$, $F \subseteq G$, \\
& \qquad $p^\varnothing(X)=\mathbb E[X]$ and
$p^F(X)=\mathbb E[X\mid H(F)]-\sum_{F'\subset F} p^{F'}(X)$ \\
\hline
    \end{tabular}
    \vspace{1em} \\
\textbf{$U$-statistic decomposition} 
\begin{equation*}
    U_{m,n} = \sum_{(0,0) \le (r,c) \le (p,q)} P^{r,c}_{m,n} = \sum_{G \in \Gamma^{-}_{r,c}} \left[(p-r)! (q-c)! \lvert \Aut(G) \rvert \right]^{-1} \widetilde{P}^{G}_{m,n}
\end{equation*}
\begin{tabular}{|l|l|} 
\hline
$\widetilde p^{G}_{\ib,\jb}$ & Symmetrized, shape-indexed increment for $G\in\Gamma_{r,c}$:
$\sum_{F\subseteq K_{\ib,\jb},\, F\sim G} p^{F}$ \\
$\widetilde{P}^{G}_{m,n}$ & Contribution of the isomorphism class of $G$: \\
& \qquad $\widetilde{P}^{G}_{m,n} = \binom{m}{p}^{-1} \binom{n}{q}^{-1} \sum_{\ib\in\mathcal P_p([m])} \sum_{\jb\in\mathcal P_q([n])}  \widetilde{p}^{G}_{\ib, \jb}$ \\
$P^{r,c}_{m,n}$ & Contribution of size $(r,c)$ in $U_{m,n}$: \\
& \qquad $P^{r,c}_{m,n} = \sum_{G \in \Gamma_{r,c}} \left[(p-r)! (q-c)! \lvert \Aut(G) \rvert \right]^{-1} \widetilde{P}^{G}_{m,n}$ \\
\hline
\end{tabular}
    \caption{Principal notions and notations defined in Sections~\ref{sec:degenerate:decomposition_probability_space} and~\ref{sec:degenerate:decomposition_u}.}
    \label{tab:notations_2}
\end{table}

\subsection{Decomposition of the variance of $U$-statistics}
\label{sec:degenerate:variance}

Just like the classic Hoeffding decomposition of $U$-statistics of i.i.d. observations~\citep{hoeffding1961strong}, the decomposition~\eqref{eq:decomposition_ustat} is convenient to decompose the variance of $U$-statistics on row-column exchangeable matrices. The following two results come from the orthogonality of the projections. For a random variable $X$, $\mathbb{V}[X]$ denotes its variance. 

The first expression links $\mathbb{V}[U_{m,n}]$ to the variance of the projections $\mathbb{V}[p^G] = \mathbb{E}[(p^G)^2]$. It is obtained by direct calculation, as shown in Appendix~\ref{app:degenerate:proof_variance}.
\begin{proposition}
\begin{equation*}
    \mathbb{V}[U_{m,n}] = \sum_{(0,0) < (r,c) \le (p,q)} \frac{(m-r)!}{m!} \frac{(n-c)!}{n!} V^{(r,c)},
\end{equation*}
where for all $(0,0) < (r,c) \le (p,q)$,
\begin{equation*}
    V^{(r,c)} = \frac{p!^2 q!^2}{(p-r)!^2 (q-c)!^2} \sum_{G \in \Gamma_{r,c}} \lvert \Aut(G) \rvert^{-1} \mathbb{E}[(p^G)^2].
\end{equation*}
    \label{prop:degenerate:variance_ustat_ep}
\end{proposition}

The second expression links $\mathbb{V}[U_{m,n}]$ to the variance of the $U$-statistics $\widetilde{P}^G_{m,n}$ associated to the symmetrized projections $\widetilde{p}^G$.
\begin{corollary}
\begin{equation*}
\begin{split}
    \mathbb{V}[U_{m,n}] &= \sum_{0 < (r,c) \le (p,q)} \sum_{G \in \Gamma_{r,c}} \left(\frac{1}{(p-r)! (q-c)! \lvert \Aut(G) \rvert}\right)^2 \mathbb{V}[\widetilde{P}^G_{m,n}]. \\
\end{split}
\end{equation*}
\end{corollary}

It can be naturally obtained from Proposition~\ref{prop:degenerate:variance_ustat_ep} using the following lemma. 
\begin{lemma}
    \begin{equation*}
        \mathbb{V}[\widetilde{P}^{G}_{m,n}] = \frac{(m-r)!}{m!} \frac{(n-c)!}{n!} p!^2 q!^2 \rvert \Aut(G) \lvert \mathbb{E}[(p^{G})^2].
    \end{equation*}
    \label{lem:degenerate:var_pg}
\end{lemma}

The proof of this lemma requires to handle the symmetrized projections, which can be tricky. In this regard, the next lemma is particularly helpful. For this reason, it will also be used several times later. The proofs of both lemmas are given in Appendix~\ref{app:degenerate:proof_variance}.

\begin{lemma}
    Let $G$ subgraph of $K_{p,q}$. Let $(G^1_{\ib, \jb})_{\substack{\ib \in \mathcal{P}_p([ m ]) \\ \jb \in \mathcal{P}_q([ n ])}}$ and $(G^2_{\ib, \jb})_{\substack{\ib \in \mathcal{P}_p([ m ]) \\ \jb \in \mathcal{P}_q([ n ])}}$ two families of graphs such that for all $(\ib, \jb) \in \mathcal{P}_p([ m ]) \times \mathcal{P}_q([ n ])$, both $G^1_{\ib, \jb}, G^2_{\ib, \jb} \subseteq K_{\ib, \jb}$ and are isomorphic to $G$. We have
\begin{equation*}
    \sum_{\substack{\ib_1, \ib_2 \in \mathcal{P}_p([ m ]) \\ \jb_1, \jb_2 \in \mathcal{P}_q([ n ])}} \sum_{\Phi_1, \Phi_2 \in \mathbb{S}_p \times \mathbb{S}_q} \mathds{1}(\Phi_1 G^1_{\ib_1, \jb_1} = \Phi_2 G^2_{\ib_2, \jb_2}) = \frac{m! (m-r)!}{(m-p)!^2} \frac{n! (n-c)!}{(n-q)!^2} \lvert \Aut(G) \rvert.
\end{equation*}    
    \label{lem:degenerate:count_equal_graphs}
\end{lemma}

\section{Asymptotic behavior}

In this section, we first define the notion of principal part of a $U$-statistic on RCE matrices, and the associated notions of principal degree and support graphs. Then, we use these notions to study and characterize the asymptotic behavior of $U$-statistics. In particular, we find that connected principal support graphs lead to Gaussian limits. Finally, we investigate the asymptotic setting where the row and column numbers have different growth rates and generalize our results to asymmetric kernels.

\subsection{Principal part and support graphs}
\label{sec:degenerate:principal_part}

Let us define a sequence for network sizes $(m_N, n_N)$ such that $m_N + n_N = N$ and $m_N/N \xrightarrow[N \rightarrow \infty]{} \rho$, for some $\rho \in ]0,1[$. We denote $U_N := U_{m_N,n_N}$, $P^{r,c}_N := P^{r,c}_{m_N,n_N}$ and $\widetilde{P}^G_{N} := \widetilde{P}^G_{m_N, n_N}$. The kernel $h$ is still a symmetric function of a matrix of size $p \times q$. Other regimes for $m_N$ and $n_N$ are considered in Section~\ref{sec:degenerate:other_asymptotic_frameworks}. In this asymptotic framework, we give the following definitions. 

\begin{definition}
Let 
\begin{equation*}
    p^{(k)} := \sum_{\substack{G \subseteq K_{p,q} \\ v_1(G) + v_2(G) = k}} p^G,
\end{equation*}
for $1 \le k \le p+q$. Let $d$ be the smallest integer such that $p^{(d)} \neq 0$. Then, we have $P^{r,c}_{N} = 0$ for all $(r,c)$ such that $r+c < d$. 
\begin{itemize}
    \item We call $d-1$ the \textnormal{order of degeneracy} of $U_N$. 
    \item By analogy with the theory of generalized $U$-statistics~\citep{janson1991asymptotic}, we call 
    \begin{equation*}
        \sum_{\substack{(0,0) \le (r,c) \le (p,q) \\ r+c = d}} P^{r,c}_{N}
    \end{equation*}
    the \textnormal{principal part} of $U_N$ and the pairs $(r,c)$ such that $r+c=d$ are the \textnormal{principal degrees} of $U_N$.
    \item We call the \textnormal{principal support graphs} of $U_N$ the graphs $G \subseteq K_{m,n}$ such that both
\begin{itemize}
    \item $v_1(G) + v_2(G) = d$,
    \item $p^G \neq 0$.
\end{itemize}
\end{itemize} 
\end{definition}

In the following examples, we identify the order of degeneracy and the principal support graphs of several $U$-statistics on row-column exchangeable matrices. 

\begin{figure}[!t]
\centering
\includegraphics[width=0.15\linewidth]{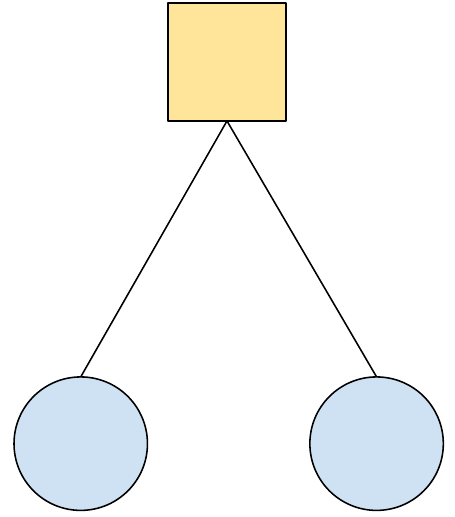}
\hspace{8em}
\includegraphics[width=0.15\linewidth]{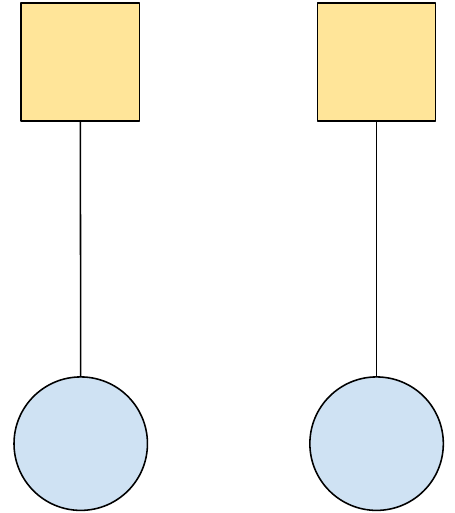}
\caption[Examples of principal support graphs for $U^{h_1}_N$ and $U^{h_2}_N$.]{Examples of principal support graphs for $U^{h_1}_N$ (left) and $U^{h_2}_N$ (right) when $Y_{ij} \overset{i.i.d.}{\sim} \mathcal{N}(0,1)$, as depicted in Running example~\ref{ex:1}. The principal support graphs of $U^{h_1}_N$ are the $1 \times 2$ graphs that are isomorphic to the left one. The principal support graphs of $U^{h_2}_N$ are the $2 \times 2$ graphs containing graphs that are isomorphic to the right one.
}
\label{fig:degenerate:principal_support}
\end{figure}

\begin{customexa}{A}[Same-row co-engagement ($K_{1,2}$)]
    The Erd\H{o}s--Rényi graph model is a binary random graph model where each edge has a fixed probability of being present, independently of the other edges \citep{erdos1959random, gilbert1959random}. It can easily be generalized to weighted graphs by defining a distribution from which the edge weights are sampled at i.i.d. 
    
    Let $Y$ be a random matrix such that $Y_{ij} \overset{i.i.d.}{\sim} \mathcal{N}(0,1)$. Let $h_1$ and $h_2$ be the kernel functions defined by $h_1(Y_{\{1\}, \{1,2\}}) = Y_{11} Y_{12}$ and $h_2(Y_{\{1,2\}, \{1,2\}}) = (Y_{11} Y_{22} + Y_{12}Y_{21})/2$, and $U^{h_1}_N$ and $U^{h_2}_N$ are the $U$-statistics associated to these kernels. \\
    $Y$ admits a natural AHK representation, which is $Y_{ij} \overset{a.s.}{=} \phi(\xi_i, \eta_j, \zeta_{ij}) = \Phi^{-1}(\zeta_{ij})$, where $\Phi^{-1}$ is the inverse c.d.f. of the standard Gaussian distribution. Remarkably, $Y_{ij}$ does not depend on the AHK variables $\xi_i$ and $\eta_j$. We have $\mathbb{E}[ Y_{ij}] = \mathbb{E}[ Y_{ij} \mid \xi_i] = \mathbb{E}[ Y_{ij} \mid \eta_j] = \mathbb{E}[ Y_{ij} \mid \xi_i, \eta_j] = 0$ and $\mathbb{E}[ Y_{ij} \mid \xi_i, \eta_j, \zeta_{ij}] = Y_{ij}$.  
    \begin{itemize}
        \item For $U^{h_1}_N$, $\mathbb{E}[h_1(Y_{\{1\}, \{1,2\}}) \mid H(G) ] \neq 0$ if and only if $$H(K_{1,2}) = (\xi_1, \eta_1, \eta_2, \zeta_{11}, \zeta_{12}) \subseteq H(G).$$ 
        Indeed, we have for all $G \subset K_{1,2}$, $\mathbb{E}[h_1(Y_{\{1\}, \{1,2\}}) \mid H(G) ] = 0$ and $\mathbb{E}[h_1(Y_{\{1\}, \{1,2\}}) \mid H(K_{1,2}) ] = Y_{11} Y_{12}$. Therefore, the only graph $G \subseteq K_{1,2}$ such that $p^G \neq 0$ is $G = K_{1,2}$. Thus, $U^{h_1}_N$ is degenerate of order $2$ and the family of principal support graphs of $U^{h_1}_N$ is $(K_{\ib, \jb})_{\ib \in \mathcal{P}_1([ m_N ]), \jb \in \mathcal{P}_2([ n_N ])}$ (Fig.~\ref{fig:degenerate:principal_support}).
        \item For $U^{h_2}_N$, $\mathbb{E}[h_2(Y_{\{1,2\}, \{1,2\}}) \mid H(G) ] \neq 0$ if and only if $$(\xi_1, \xi_2, \eta_1, \eta_2, \zeta_{11}, \zeta_{22}) \subseteq H(G) \text{\quad or \quad} (\xi_1, \xi_2, \eta_1, \eta_2, \zeta_{12}, \zeta_{21}) \subseteq H(G).$$ 
        Therefore, if $\mathbb{E}[h_2(Y_{\{1,2\}, \{1,2\}}) \mid H(G) ] \neq 0$, then $v_1(G) = 2$ or $v_2(G) = 2$, so $U^{h_2}_N$ is degenerate of order 3. The principal support graphs are the graphs that are isomorphic to one graph $G \subseteq \Gamma_{2,2}$ such that $\mathbb{E}[h_2(Y_{\{1,2\}, \{1,2\}}) \mid H(G) ] \neq 0$ (Fig.~\ref{fig:degenerate:principal_support}).
    \end{itemize}
    \label{ex:1}
\end{customexa}

\begin{customexa}{B}[Row heterogeneity]
    Let $Y$ be a random matrix sampled from the following RCE dissociated model: for $\lambda > 0$,
\begin{align}
        &\xi_i \overset{i.i.d.}{\sim} \mathcal{U}[0,1], &\forall 1 \le i \le m, \nonumber \\
        &\eta_j \overset{i.i.d.}{\sim} \mathcal{U}[0,1], &\forall 1 \le j \le n, \label{eq:poisson_bedd} \\ 
        &Y_{ij} \mid \xi_i, \eta_j \sim \mathcal{P}(\lambda f(\xi_i) g(\eta_j)), &\forall 1 \le i \le m, 1 \le j \le n. \nonumber
\end{align}
This describes the Poisson Bipartite Expected Degree Distribution (Poisson-BEDD) model \citep{ouadah2022motif, leminh2023ustatistics}. This model is a type of weighted bipartite graphon model~\citep{diaconis2008graph}, where the graphon function has a product form. It is defined by a density parameter $\lambda$ and functions $f : [0,1] \rightarrow \mathbb{R}$ and $g : [0,1] \rightarrow \mathbb{R}$ representing the expected degree distributions of the rows and the columns respectively. We assume that $\int f = \int g = 1$, so that the mean intensity of the network is $\mathbb{E}[Y_{ij}] = \lambda \int f \int g = \lambda$. The expected degree of the $i$-th row node is $\mathbb{E}[\sum_{j=1}^n Y_{ij} \mid \xi_i] = n \lambda f(\xi_i)$ and the expected degree of the $j$-th column node is $\mathbb{E}[\sum_{i=1}^m Y_{ij} \mid \eta_j] = m \lambda g(\eta_j)$. 

Suppose that we are interested in testing if the row degrees are homogeneous, i.e. $f \equiv 1$. For that, let us define the null hypothesis $\mathcal{H}_0 : f \equiv 1$ and confront it to $\mathcal{H}_1 : f \not\equiv 1$. 
The quantity $F_2 := \int f^2$ is related to the variance of the row expected degree distribution. We may use its estimate to test this hypothesis. Indeed, under $\mathcal{H}_0$, we have $F_2 = 1$ and otherwise, $F_2 > 1$. Consider the kernels $h_1$ and $h_2$ defined in Running example~\ref{ex:1}. Now, in the Poisson-BEDD model, they have expectations $\mathbb{E}[h_1(Y_{\{1\}, \{1,2\}}) ] = \lambda^2 F_2$ and $\mathbb{E}[h_2(Y_{\{1, 2\}, \{1,2\}})] = \lambda^2$. Therefore, 
\begin{equation*}
    U^{h_3}_N := U^{h_1}_N - U^{h_2}_N
\end{equation*} 
is also a $U$-statistic, associated to the kernel $h$ defined by 
\begin{equation*}
    h_3(Y_{\{1,2\}, \{1,2\}}) = \frac{1}{2}\left[h_1(Y_{\{1\}, \{1,2\}}) + h_1(Y_{\{2\}, \{1,2\}})\right] - h_2(Y_{\{1,2\}, \{1,2\}}),
\end{equation*}
centered around 
\begin{equation*}
    \mathbb{E}[U^{h_3}_N] = \lambda^2 (F_2 - 1)
\end{equation*}
which is equal to $0$ under $\mathcal{H}_0$ only.

We remark that 
\begin{equation*}
\begin{split}
    \mathbb{E}[h_3(Y_{\{1,2\},\{1,2\}}) \mid \xi_1] &= \frac{1}{2}\mathbb{E}[Y_{11}Y_{12} + Y_{21}Y_{22} - Y_{11}Y_{22} - Y_{21}Y_{12} \mid \xi_1] \\
    &= \frac{\lambda^2}{2} (f(\xi_1)^2 + F_2 - 2 f(\xi_1)),
\end{split}
\end{equation*}
and
\begin{equation*}
\begin{split}
    \mathbb{E}[h_3(Y_{\{1,2\},\{1,2\}}) \mid \eta_1] &= \frac{1}{2}\mathbb{E}[Y_{11}Y_{12} + Y_{21}Y_{22} - Y_{11}Y_{22} - Y_{21}Y_{12} \mid \eta_1] \\
    &= \lambda^2 (F_2 - 1) g(\eta_1).
\end{split}
\end{equation*}
Since $\mathbb{E}[h_3(Y_{\{1,2\},\{1,2\}}) \mid \xi_1] = \mathbb{E}[h_3(Y_{\{1,2\},\{1,2\}}) \mid \eta_1] = 0$ when $f \equiv 1$, this means that $U^{h_3}_N$ is degenerate of order at least 1 under $\mathcal{H}_0$. 

In order to find the principal support graphs of $U^{h_3}_N$, we can check if $\mathbb{E}[h_3(Y_{\{1,2\},\{1,2\}}) \mid H(G)] \neq 0$, first for graphs $G \in \cup_{r+c=2}\Gamma_{r,c}$. In fact, there are only four graphs in $\cup_{r+c=2}\Gamma_{r,c}$. Their corresponding conditional expectations $\mathbb{E}[h_3(Y_{\{1,2\},\{1,2\}}) \mid H(G)]$ are calculated in Lemmas~\ref{lem:degenerate:condexp_1} to~\ref{lem:degenerate:condexp_2}. Under $\mathcal{H}_0$, they become 
\begin{itemize}
    \item $\mathbb{E}[h_3(Y_{\{1,2\},\{1,2\}}) \mid \xi_1, \xi_2] = 0$,
    \item $\mathbb{E}[h_3(Y_{\{1,2\},\{1,2\}}) \mid \eta_1, \eta_2] = 0$,
    \item $\mathbb{E}[h_3(Y_{\{1,2\},\{1,2\}}) \mid \xi_1, \eta_1] = 0$,
    \item $\mathbb{E}[h_3(Y_{\{1,2\},\{1,2\}}) \mid \xi_1, \eta_1, \zeta_{11}] = 0$.
\end{itemize}
Since there are no graph of $\cup_{r+c=2}\Gamma_{r,c}$ such that $\mathbb{E}[h_3(Y_{\{1,2\},\{1,2\}}) \mid H(G)] \neq 0$, that means that $U^{h_3}_N$ is degenerate of order at least $2$.

Next, we check if $\mathbb{E}[h_3(Y_{\{1,2\},\{1,2\}}) \mid H(G)] \neq 0$, for graphs $G \in \cup_{r+c=3} \Gamma_{r,c}$. There are six graphs in $\cup_{r+c=3} \Gamma_{r,c}$. According to Lemmas~\ref{lem:degenerate:condexp_3} to~\ref{lem:degenerate:condexp_4}, we have under $\mathcal{H}_0$,
\begin{itemize}
    \item $\mathbb{E}[h_3(Y_{\{1,2\},\{1,2\}}) \mid \xi_1, \xi_2, \eta_1] = 0$,
    \item $\mathbb{E}[h_3(Y_{\{1,2\},\{1,2\}}) \mid \xi_1, \xi_2, \eta_1, \zeta_{11}] = 0$,
    \item $\mathbb{E}[h_3(Y_{\{1,2\},\{1,2\}}) \mid \xi_1, \xi_2, \eta_1, \zeta_{11}, \zeta_{21}] = 0$,
    \item $\mathbb{E}[h_3(Y_{\{1,2\},\{1,2\}}) \mid \xi_1, \eta_1, \eta_2] = 0$,
    \item $\mathbb{E}[h_3(Y_{\{1,2\},\{1,2\}}) \mid \xi_1, \eta_1, \eta_2, \zeta_{11}] = 0$,
    \item $\mathbb{E}[h_3(Y_{\{1,2\},\{1,2\}}) \mid \xi_1, \eta_1, \eta_2, \zeta_{11}, \zeta_{12}] = (Y_{11} Y_{12} + \lambda^2 g(\eta_1) g(\eta_2) - \lambda g(V_2) Y_{11} - \lambda g(V_1) Y_{12})/2 \neq 0$.
\end{itemize}
Therefore, there is one (and only one) graph $G$ satisfying this condition, so we can conclude that the order of degeneracy of $U_N$ is $2$. This graph is the one such that $H(G) = (\xi_1, \eta_1, \eta_2, \zeta_{11}, \zeta_{12})$, which means that $G = K_{1,2}$. Thus, the principal support graphs of $U^{h_3}_N$ are the graphs $(K_{\ib, \jb})_{\ib \in \mathcal{P}_1([ m_N ]), \jb \in \mathcal{P}_2([ n_N ])}$.
\label{ex:2}
\end{customexa}

\begin{customexa}{C}[Overdispersion beyond row-column effects]
    Although the Poisson distribution is often used to model count data, it comes with the limitation that the mean and variance of the counts are equal. For this reason, the Poisson distribution is not an appropriate choice to represent count data when it shows an overly large dispersion. Therefore, for some given dataset, it is interesting to assess whether the Poisson distribution fits the count data, i.e. whether the counts exhibit too much dispersion. However, for network count data, apparent overdispersion can arise from the row and column heterogeneity. A Poisson-graphon model such as the Poisson-BEDD model depicted in the previous example by~\eqref{eq:poisson_bedd} would be able to capture this effect. Indeed, in the Poisson-BEDD model, we have 
    \begin{equation}
        \mathbb{V}[Y_{ij}] = \mathbb{E}[Y_{ij}^2] - \mathbb{E}[Y_{ij}]^2 = \mathbb{E}[\lambda^2 f(\xi_i)^2 g(\eta_j)^2] + \lambda - \lambda^2  = \lambda^2 (F_2 G_2 - 1) + \lambda > \mathbb{E}[Y_{ij}],
        \label{eq:var_bedd}
    \end{equation}
    where $F_2 = \int f^2$ and $G_2 = \int g^2$. Now, in the following example, we build a test for this type of model to decide whether there is an overdispersion effect beyond the effect of the rows and columns, e.g. whether the conditional counts $Y_{ij} \mid \xi_i, \eta_j$ show overdispersion.
    
    Typically, overdispersed count data $(Z_i)_{1 \le i \le n}$ are better represented by families of distributions including a parameter controlling the variance separately from the mean, such as a negative binomial distribution, rather than a classic Poisson distribution. For example, the negative binomial representation $Z_i \overset{i.i.d.}{\sim} \mathcal{NB}(\lambda, \alpha)$ can be parameterized such that the mean is $\lambda$ and the variance is $\lambda + \lambda^2 \alpha$. The parameter $\alpha$ is then called the \textit{dispersion parameter}. In fact, the negative binomial distribution can be viewed as a Gamma-Poisson mixture, that is for $1 \le i \le n$, $Z_i \mid W_i \sim \mathcal{P}(\lambda W_i)$, where $W_1, \dots, W_n$ are independently drawn random variables from the Gamma distribution with mean $1$ and variance $\alpha$. 
    
    Here, we use a more generic approach to represent overdispersed data with Poisson mixtures where, contrasting with the negative binomial case, the distribution of the latent variables $(W_i)_{i \ge 1}$ is not necessarily Gamma, but some distribution $\mathcal{L}(\alpha)$ on the nonnegative real numbers with mean 1 and variance $\alpha$. This is sufficient to obtain $\mathbb{E}[Z_{i}] = \lambda$ and $\mathbb{V}[Z_{i}] = \lambda + \lambda^2 \alpha$.
    
    As for network count data, the Overdispersed-Poisson-BEDD model can be defined to capture a dispersion effect, separately from the row and column effects, for $\lambda > 0$, $\alpha \ge 0$,
    \begin{align*}
        &\xi_i \overset{i.i.d.}{\sim} \mathcal{U}[0,1], &\forall 1 \le i \le m, \\
        &\eta_j \overset{i.i.d.}{\sim} \mathcal{U}[0,1], &\forall 1 \le j \le n, \\
        &W_{ij} \overset{i.i.d.}{\sim} \mathcal{L}(\alpha), &\forall 1 \le i \le m, 1 \le j \le n, \\
        &Y_{ij} \mid \xi_i, \eta_j, W_{ij} \sim \mathcal{P}(\lambda f(\xi_i) g(\eta_j) W_{ij}), &\forall 1 \le i \le m, 1 \le j \le n,
    \end{align*}
    where $(\mathcal{L}(\alpha))_{\alpha \ge 0}$ is a family of probability distributions on the nonnegative real numbers with mean 1 and parameterized by their variances, i.e. for all $1 \le i \le m, 1 \le j \le n$, $\mathbb{E}[W_{ij}] = 1$ and $\mathbb{V}[W_{ij}] = \alpha$. Like the Poisson-BEDD model, the Overdispersed-Poisson-BEDD model is row-column exchangeable and admits a similar AHK representation, the only difference being the addition of the latent variables $(W_{ij})_{1 \le i \le m, 1 \le j \le n}$, each of which is measurable by the $\sigma$-field generated by the corresponding AHK variable $\zeta_{ij}$. We see that
    \begin{equation*}
        \mathbb{V}[Y_{ij}] = \mathbb{E}[Y_{ij}^2] - \mathbb{E}[Y_{ij}]^2 = \mathbb{E}[\lambda^2 f(\xi_i)^2 g(\eta_j)^2 W_{ij}^2] + \lambda - \lambda^2  = \lambda^2 (F_2 G_2 (\alpha + 1) - 1) + \lambda,
    \end{equation*} 
    which is larger than the variance under the Poisson-BEDD model derived in~\eqref{eq:var_bedd} by an additive term $\lambda^2 F_2 G_2 \alpha$. The Overdispersed-Poisson-BEDD model is simplified into the Poisson-BEDD model~\eqref{eq:poisson_bedd} when the dispersion parameter $\alpha = 0$. Therefore, we can use this network model to test the hypothesis $\mathcal{H}_0 : \alpha = 0$, which allows us to assess whether the Poisson-BEDD model is appropriate to represent the data concerning its dispersion.

    Let $h_4$ and $h_5$ be kernels functions defined by
    \begin{equation*}
        h_4(Y_{\{1,2\}, \{1,2\}}) = \frac{1}{4}\left( Y_{11} (Y_{11} - 1) Y_{22} + Y_{12} (Y_{12} - 1) Y_{21} + Y_{21} (Y_{21} - 1) Y_{12} + Y_{22} (Y_{22} - 1) Y_{11}\right)
    \end{equation*}
    and
    \begin{equation*}
        h_5(Y_{\{1,2\}, \{1,2\}}) = \frac{1}{4}\left( Y_{11} Y_{12} Y_{22} + Y_{12} Y_{22} Y_{21} + Y_{22} Y_{21} Y_{11} + Y_{21} Y_{11} Y_{12} \right).
    \end{equation*}
    We have $\mathbb{E}[h_4(Y_{\{1,2\}, \{1,2\}})] = \lambda^3 F_2 G_2 (1 + \alpha)$ and $\mathbb{E}[h_5(Y_{\{1,2\}, \{1,2\}})] = \lambda^3 F_2 G_2$, therefore, the kernel function $h_6 := h_4 - h_5$ has expectation 
    \begin{equation*}
        \mathbb{E}[h_6(Y_{\{1,2\}, \{1,2\}})] = \lambda^3 F_2 G_2 \alpha,
    \end{equation*}
    which is zero under $\mathcal{H}_0$ only. Thus, we can use the associated $U$-statistic $U^{h_6}_N$ to test $\mathcal{H}_0$.

    The first-order projection terms are given by Lemmas~\ref{lem:degenerate:condexp_5} and~\ref{lem:degenerate:condexp_6}. Under $\mathcal{H}_0$, we notice that
    \begin{itemize}
        \item $\mathbb{E}[h_6(Y_{\{1,2\}, \{1,2\}}) \mid \xi_1] = 0$,
        \item $\mathbb{E}[h_6(Y_{\{1,2\}, \{1,2\}}) \mid \eta_1] = 0$.
    \end{itemize}
    Therefore, we can deduce that $U^{h_6}_N$ is a degenerate $U$-statistic. Like in the previous example, we determine its order of degeneracy by finding the principal support graphs. This is done by searching for the graphs of $\cup_{r+c=d}\Gamma_{r,c}$ such that $\mathbb{E}[h_6(Y_{\{1,2\},\{1,2\}}) \mid H(G)] \neq 0$ under $\mathcal{H}_0$, for increasing values of $d$ until some are found.

    The projection terms corresponding to $d=2$ are given by Lemmas~\ref{lem:degenerate:condexp_7} to~\ref{lem:degenerate:condexp_8}. Here, we see that under $\mathcal{H}_0$,
    \begin{itemize}
        \item $\mathbb{E}[h_6(Y_{\{1,2\}, \{1,2\}}) \mid \xi_1, \xi_2] = 0$,
        \item $\mathbb{E}[h_6(Y_{\{1,2\}, \{1,2\}}) \mid \eta_1, \eta_2] = 0$,
        \item $\mathbb{E}[h_6(Y_{\{1,2\}, \{1,2\}}) \mid \xi_1, \eta_1] = 0$,
        \item $\mathbb{E}[h_6(Y_{\{1,2\}, \{1,2\}}) \mid \xi_1, \eta_1, \zeta_{11}] \neq 0$.
    \end{itemize}
    This means that we have found one (and only one) graph $G$ of $\cup_{r+c=2}\Gamma_{r,c}$ such that $\mathbb{E}[h_6(Y_{\{1,2\},\{1,2\}}) \mid H(G)] \neq 0$. This graph is the one such that $H(G) = \sigma(\xi_1, \eta_1, \zeta_{12})$. Thus, the order of degeneracy of $U^{h_6}_N$ is $1$ and its principal support graphs are the graphs $(K_{\{i\}, \{j\}})_{1 \le i \le n, 1 \le j \le m}$.
    \label{ex:3}
\end{customexa}

\subsection{Convergence of degenerate $U$-statistics}
\label{sec:degenerate:limit_distribution}

Now, having defined the principal part of a $U$-statistic, we see how it is directly related to the limit distribution of the $U$-statistic. From Proposition~\ref{prop:degenerate:variance_ustat_ep}, we have
\begin{equation*}
\begin{split}
    \mathbb{V}[U_{N}] &= \sum_{(0,0) < (r,c) \le (p,q)} \frac{(m_N-r)!}{m_N!} \frac{(n_N-c)!}{n_N!} V^{(r,c)} 
\end{split}
\end{equation*}
We see that $\mathbb{V}[U_N]$ is the sum of the $p \times q$ terms of the form $\frac{(m_N-r)!}{m_N!} \frac{(n_N-c)!}{n_N!} V^{(r,c)}$. Each term behaves like $\frac{(m_N-r)!}{m_N!} \frac{(n_N-c)!}{n_N!} V^{(r,c)} \asymp N^{-r-c}$. If for some $(r,c)$, $\sum_{\substack{G \in K_{p,q} \\ (v_1(G),v_2(G)) = (r,c)}} p^G = 0$, then $V^{(r,c)} = 0$. Therefore,
\begin{equation*}
\begin{split}
    \mathbb{V}[U_{N}] &= N^{-d} \sum_{\substack{(0,0) < (r,c) \le (p,q) \\ r+c=d}}  \rho^{-r} (1- \rho)^{-c} V^{(r,c)}  + o(N^{-d}) \\
    &= N^{-d} \sum_{r=0}^d  \rho^{-r} (1- \rho)^{-d-r} V^{(r,d-r)}  + o(N^{-d}) \\
\end{split}
\end{equation*}
This is a hint that the right normalization for the convergence in distribution of $U_N$ is given by its principal degrees. The following theorem, proven in Appendix~\ref{app:degenerate:proof_principal_part}, confirms it.

\begin{theorem}
    There is a random variable $W$ such that $N^{d/2} (U_N - p^\emptyset) \xrightarrow[]{\mathcal{D}} W$ if and only if 
    \begin{equation*}
        N^{d/2} \sum_{\substack{(0,0) < (r,c) \le (p,q) \\ r+c=d}} P^{r,c}_{N} \xrightarrow[]{\mathcal{D}} W.
    \end{equation*}
    \label{th:degenerate:principal_part}
\end{theorem}

This theorem says that the limit distribution of $U_N - p^\emptyset$ renormalized by $N^{d/2}$ is the same as that of its principal part $\sum_{\substack{(0,0) \le (r,c) \le (p,q) \\ r+c = d}} P^{r,c}_{N}$, renormalized by the same quantity. Therefore, the principal support graphs of $U_N$ characterize the limit distribution of $U_N$. More specifically, the limit distribution depends on the form of the principal support graphs of $U_N$. 

\subsection{Asymptotic Gaussian distribution}
\label{sec:degenerate:gaussian_case}

Now, we identify a sufficient condition for the principal support graphs to have a Gaussian limit distribution for $N^{d/2}(U_N - p^\emptyset)$, using the properties of the principal part of $U_N$. 

\begin{theorem}
    If all principal support graphs of $U_N$ are connected, then
    \begin{equation*}
        N^{d/2} (U_N - p^\emptyset) \xrightarrow[N \rightarrow \infty]{\mathcal{D}} \mathcal{N}(0, \sigma^2),
    \end{equation*}
    where 
    \begin{equation*}
        \sigma^2 = \sum_{\substack{(0,0) < (r,c) \le (p,q) \\ r+c=d}} \rho^{-r} (1-\rho)^{-c} V^{(r, c)}.
    \end{equation*}
    \label{th:degenerate:gaussian_theorem}
\end{theorem}

The proof of this theorem uses the fact that from Theorem~\ref{th:degenerate:principal_part}, $N^{d/2} (U_N - p^\emptyset)$ has the same limit as $$N^{d/2} \sum_{\substack{(0,0) < (r,c) \le (p,q) \\ r+c=d}} P^{r,c}_{N},$$ 
where
\begin{equation*}
     P^{r,c}_{N} = \sum_{G \in \Gamma_{r,c}} \frac{1}{(p-r)! (q-c)! \lvert \Aut(G) \rvert} \widetilde{P}^{G}_{N}.
\end{equation*}
Two lemmas are further needed. The convergence of the terms $N^{d/2} \widetilde{P}^{G}_{N}$ is proved by
the methods of moments (Lem.~\ref{lem:degenerate:convergence_gaussian_projection}). The calculation of the moments involves sums of terms of the form $\mathbb{E}[\prod_{k = 1}^K p^{G_k}]$, the values of which depend on the configuration of the sequence of graphs $G_1, ..., G_k$ (Lem.~\ref{lem:degenerate:graph_pairs}). Therefore, the moments are obtained by counting the frequency of the relevant configurations in these sums. 

Below, Lemmas~\ref{lem:degenerate:graph_pairs} and~\ref{lem:degenerate:convergence_gaussian_projection} are given before the full proof of Theorem~\ref{th:degenerate:gaussian_theorem}. The proofs for these lemmas can be found in Appendix~\ref{app:degenerate:proof_gaussian_case}.

\begin{lemma}
    Let $G_1, ..., G_K$ be subgraphs of $K_{m_N, n_N}$. If $\mathbb{E}[\prod_{k = 1}^K p^{G_k}] \neq 0$, then for all $G_k$, $1 \le k \le K$, each vertex of $V_1(G_k)$ or $V_2(G_k)$ or edge of $E(G_k)$ must also appear in another $G_\ell$, $\ell \neq k$.
    
    Furthermore, if $G_1, ..., G_K$ are connected and non-empty, then one and only one of the following propositions is true:
    \begin{itemize}
        \item $K$ is even and $G_1, ..., G_K$ coincide in $K/2$ pairs (i.e. the indices $[ K ]$ can be grouped into $K/2$ pairs and within each pair, the corresponding bipartite graphs are equal),
        \item some vertex belongs to at least three of these graphs.
    \end{itemize} 
    \label{lem:degenerate:graph_pairs}
\end{lemma}

\begin{lemma}
    Let $(G_k)_{1 \le k \le K}$ be a sequence of distinct connected graphs of $\Gamma^-_{p,q}$, with $v_1(G_k) = r_k$ and $v_2(G_k) = c_k$ for $1 \le k \le K$. We have that
\begin{equation}
    (m_N^{r_k/2} n_N^{c_k/2} \widetilde{P}^{G_k}_{N})_{1 \le k \le K} \xrightarrow[]{\mathcal{D}} (W_k)_{1 \le k \le K},
\end{equation}     
where $W_k$ are independent variables with respective distribution $\mathcal{N}(0, p!^2 q!^2 \lvert \Aut(G_k) \rvert \mathbb{E}[(p^{G_k})^2])$.
\label{lem:degenerate:convergence_gaussian_projection}
\end{lemma}

\begin{proof}[Proof of Theorem~\ref{th:degenerate:gaussian_theorem}]
    
Theorem~\ref{th:degenerate:principal_part} states that $N^{d/2} (U_N - p^\emptyset)$ has the same limit as $$N^{d/2} \sum_{\substack{(0,0) < (r,c) \le (p,q) \\ r+c=d}} P^{r,c}_{N}.$$
    
For all $(0,0) < (r,c) \le (p,q)$,
\begin{equation*}
     P^{r,c}_{N} = \sum_{G \in \Gamma_{r,c}} \frac{1}{(p-r)! (q-c)! \lvert \Aut(G) \rvert} \widetilde{P}^{G}_{N}.
\end{equation*}

So 
\begin{equation*}
     N^{d/2} \sum_{\substack{(0,0) < (r,c) \le (p,q) \\ r+c=d}} P^{r,c}_{N} = \sum_{\substack{(0,0) < (r,c) \le (p,q) \\ r+c=d}} N^{d/2} m_N^{-r/2} n_N^{-c/2} \sum_{G \in \Gamma_{r,c}}  \frac{m_N^{r/2} n_N^{c/2} \widetilde{P}^{G}_{N}}{(p-r)! (q-c)! \lvert \Aut(G) \rvert}.
\end{equation*}

By construction, $N^{d/2} m_N^{-r/2} n_N^{-c/2} \xrightarrow[N \rightarrow \infty]{} \rho^{-r/2} (1-\rho)^{-c/2}$. Therefore, by Lemma~\ref{lem:degenerate:convergence_gaussian_projection}, $$N^{d/2} \sum_{\substack{(0,0) < (r,c) \le (p,q) \\ r+c=d}} P^{r,c}_{N}$$ converges in distribution to 
\begin{equation*}
    Z = \sum_{\substack{(0,0) < (r,c) \le (p,q) \\ r+c=d}} \rho^{-r/2} (1-\rho)^{-c/2} \sum_{G \in \Gamma_{r,c}} W_G,
\end{equation*}
where for all $(r,c)$, $G \in \Gamma_{r,c}$, $W_G$ are independent Gaussian variables with mean $0$ and variance 
\begin{equation*}
    \frac{p!^2 q!^2}{(p-r)!^2 (q-c)!^2 \lvert \Aut(G) \rvert} \mathbb{E}[(p^{G})^2].
\end{equation*}

Finally, it follows that $Z$ is a Gaussian variable with mean $0$ and variance 
\begin{equation*}
    \sum_{\substack{(0,0) < (r,c) \le (p,q) \\ r+c=d}} \rho^{-r} (1-\rho)^{-c} V^{(r, c)},
\end{equation*} 
where 
\begin{equation*}
    V^{(r, c)} = \frac{p!^2 q!^2}{(p-r)!^2 (q-c)!^2 } \sum_{G \in \Gamma_{r, c} } \lvert \Aut(G) \rvert^{-1} \mathbb{E}[(p^{G})^2].
\end{equation*}

\end{proof}

\begin{remark}
    If $Y$ and $h$ are such that the principal support graphs of $U_N$ include $K_{1,0}$ or $K_{0,1}$, then the principal degree of $U_N$ is $1$, and the limit distribution is Gaussian. In this case, Theorem~\ref{th:degenerate:gaussian_theorem} directly yields the non-degenerate Central Limit Theorem for $U$-statistics on RCE matrices, as established by~\cite{leminh2023ustatistics} and~\cite{leminh2025hoeffding}:
    \begin{equation*}
        \sqrt{N} (U_N - p^\emptyset) \xrightarrow[N \rightarrow \infty]{\mathcal{D}} \mathcal{N}(0, \sigma^2),
    \end{equation*}
    where $\sigma^2 = \rho^{-1} V^{(1,0)} + (1 - \rho)^{-1} V^{(0,1)}$, with Proposition~\ref{prop:degenerate:variance_ustat_ep} giving $V^{(1,0)} = p^2 \mathbb{V}[\mathbb{E}[h(Y_{[ p ], [ q ]}) \mid \xi_1]]$ and $V^{(0,1)} = q^2 \mathbb{V}[\mathbb{E}[h(Y_{[ p ], [ q ]}) \mid \eta_1]]$. \\
    This also characterizes the degeneracy of $U_N$. $U_N$ is degenerate if and only if $V = 0$, which means both $\mathbb{E}[h(Y_{[ p ], [ q ]}) \mid \xi_1] = 0$ and $\mathbb{E}[h(Y_{[ p ], [ q ]}) \mid \eta_1] = 0$. This also only happens when neither $K_{1,0}$ nor $K_{0,1}$ are principal support graphs, i.e. when the order of degeneracy of $U_N$ is larger than $1$.
    \\
    We deduce that there is no hope to obtain a faster rate of convergence than $\sqrt{N}$ in non-degenerate cases and that it is always greater in degenerate cases. This is in accordance with the discussion of~\cite{leminh2023ustatistics}, but it shows how the principal support graphs and the order of degeneracy of $U_N$ characterize the degeneracy of $U_N$. 
    \\
    This characterization also shows that the coarser decomposition proposed by \cite{leminh2023ustatistics} is inadequate for identifying the precise distribution of degenerate $U$-statistics. The terms in the decomposition of \cite{leminh2023ustatistics} are indexed by pairs $(r,c)$, where $r \in [ p ]$ and $c \in [ q ]$. Each term $(r,c)$ of that decomposition aggregates the contributions of the finer decomposition presented in this paper, specifically summing over all bipartite graphs with $r$ row nodes and $c$ column nodes. However, the topology of these bipartite graphs determines the contribution of each term to the asymptotic distribution. In the decomposition of \cite{leminh2023ustatistics}, terms associated with graphs of distinct topologies are indistinguishable. Because of that, a result such as Lemma~\ref{lem:degenerate:convergence_gaussian_projection} cannot be obtained for any term of that decomposition, except for those of the first-order (because the graphs with $(r,c) = (0,1)$ or $(1,0)$ are trivial), which only solves the non-degenerate case.
\end{remark}

\subsection{Applications}
\label{sec:degenerate:running_examples}

In this section, we develop further and conclude with the three running examples of Section~\ref{sec:degenerate:principal_part}, identifying the distributional limit for the various statistics of interest. In particular, statistical tests can be derived from Running examples~\ref{ex:2} and~\ref{ex:3}. We provide some simulation results, most of which are deferred to Appendix~\ref{app:simulations} for concision.

\paragraph{Running example~\ref{ex:1}: Same-row co-engagement ($K_{1,2}$).}
    In this running example, $Y$ is a random matrix such that $Y_{ij} \overset{i.i.d.}{\sim} \mathcal{N}(0,1)$. Consider $h_1$ the kernel function defined by $h_1(Y_{\{1\}, \{1,2\}}) = Y_{11} Y_{12}$ and $U^{h_1}_N$ the $U$-statistic associated to this kernel. In Section~\ref{sec:degenerate:principal_part}, we have seen that $U^{h_1}_N$ is degenerate of order $2$ and the family of principal support graphs of $U^{h_1}_N$ is $(K_{\ib, \jb})_{\ib \in \mathcal{P}_1([ m_N ]), \jb \in \mathcal{P}_2([ n_N ])}$, which are all connected. \\
    Therefore, Theorem~\ref{th:degenerate:gaussian_theorem} implies
    \begin{equation*}
        N^{3/2} U^{h_1}_N \xrightarrow[N \rightarrow \infty]{\mathcal{D}} \mathcal{N}(0, \sigma_1^2),
    \end{equation*}
    where $\sigma_1^2 = V^{(1,2)} = \frac{4}{\rho (1-\rho)^2} \lvert \Aut(K_{1,2}) \rvert^{-1} \mathbb{E}[(p^{K_{1,2}})^2] = \frac{4}{\rho (1-\rho)^2} \frac{1}{2} \mathbb{E}[Y_{11}^2 Y_{12}^2] = \frac{2}{\rho (1-\rho)^2}$. Figure~\ref{fig:degeneracy_basic} illustrates this convergence result.

    \begin{figure}[!tb]
\centering
\includegraphics[width=0.8\linewidth]{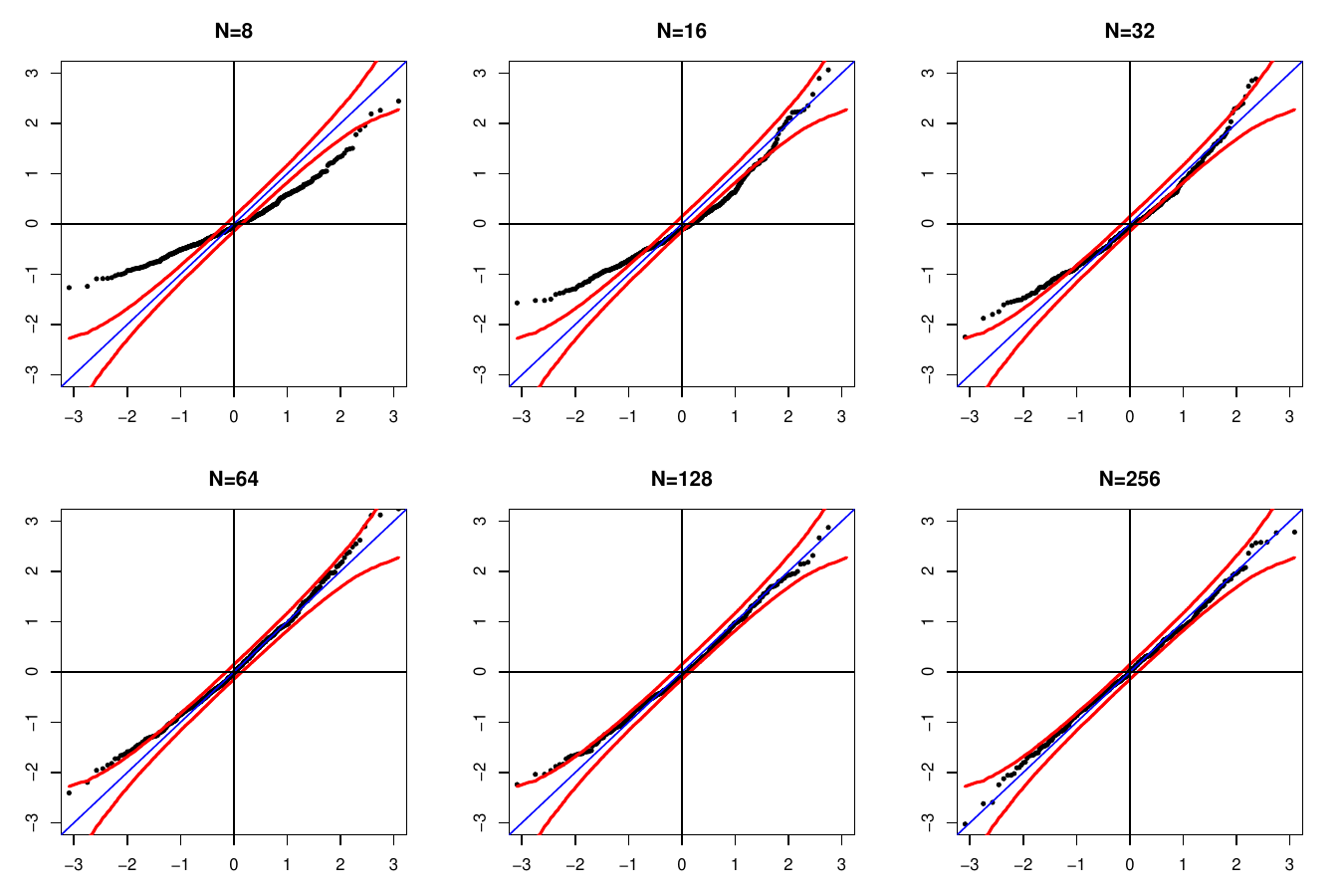}
\caption[Q-Q plots for Running example~\ref{ex:1}]{These Q-Q plots show convergence of $Z^A_N := N^{3/2} \sigma_1^{-1} U^{h_1}_N$ to a standard normal distribution ($500$ simulated networks for each network size). The red lines represent the $99\%$-confidence envelope. }
\label{fig:degeneracy_basic}
\end{figure}

\paragraph{Running example~\ref{ex:2}: Row heterogeneity.}
    We have previously seen that under the Poisson-BEDD model with $f \equiv 1$, the principal support graphs of $U^{h_3}_N = U^{h_1}_N - U^{h_2}_N$ are the graphs $(K_{\ib, \jb})_{\ib \in \mathcal{P}_1([ m_N ]), \jb \in \mathcal{P}_2([ n_N ])}$, which are connected graphs. Therefore, we can apply Theorem~\ref{th:degenerate:gaussian_theorem}, implying that 
    \begin{equation*}
    N^{3/2} U^{h_3}_N \xrightarrow[N \rightarrow \infty]{\mathcal{D}} \mathcal{N}(0, \sigma_3^2),
\end{equation*}
where $\sigma_3^2 =  V^{(1,2)} = \frac{16}{\rho (1-\rho)^2} \lvert \Aut(K_{1,2}) \rvert^{-1} \mathbb{E}[(p^{K_{1,2}})^2] = \frac{2 \lambda^2}{\rho (1-\rho)^2}$, applying Lemma~\ref{lem:degenerate:varexp} with $F_2 = F_3 = F_4 = 1$ under $\mathcal{H}_0 : f \equiv 1$. Thus, the statistic $Z^B_N := N^{3/2} \sigma_3^{-1} U^{h_3}_N$ has a known asymptotic distribution (standard Gaussian) and can be used to build a test for $\mathcal{H}_0: f \equiv 1$. Figure~\ref{fig:degenerate:power_tests} (left) shows the empirical power of such a test for different network sizes. Simulations illustrating the asymptotic behavior of the statistic can be found in Appendix~\ref{app:simulations}. 

\paragraph{Running example~\ref{ex:3}: Overdispersion beyond row-column effects.}
    Under the Overdispersed-Poisson-BEDD model with $\alpha = 0$, the principal support graphs of $U^{h_6}_N = U^{h_4}_N - U^{h_5}_N$ are the graphs $(K_{\{i\}, \{j\}})_{1 \le i \le n, 1 \le j \le m}$, which are connected graphs. According to Theorem~\ref{th:degenerate:gaussian_theorem}, we have 
    \begin{equation*}
    N U^{h_6}_N \xrightarrow[N \rightarrow \infty]{\mathcal{D}} \mathcal{N}(0, \sigma_6^2),
\end{equation*}
where $\sigma_6^2 = V^{(1,1)} = \frac{16}{\rho (1-\rho)} \lvert \Aut(K_{1,1}) \rvert^{-1} \mathbb{E}[(p^{K_{1,1}})^2] = \frac{\lambda^4}{\rho (1-\rho)} \left[ \lambda (F_3 - F_2^2)(G_3 - G_2^2) + 2 F_2 G_2 \right]$ (derivation in Lemma~\ref{lem:degenerate:var_ex3}). Thus, the statistic $Z^C_N := N \sigma_6^{-1} U^{h_6}_N$ has a known asymptotic distribution (standard Gaussian) and can be used to build a test for $\mathcal{H}_0 : \alpha = 0$. Figure~\ref{fig:degenerate:power_tests} (right) shows the empirical power of such a test for different network sizes. Simulations illustrating the asymptotic behavior of the statistic can be found in Appendix~\ref{app:simulations}. 

\begin{figure}[!t]
\centering
\includegraphics[width=0.48\linewidth]{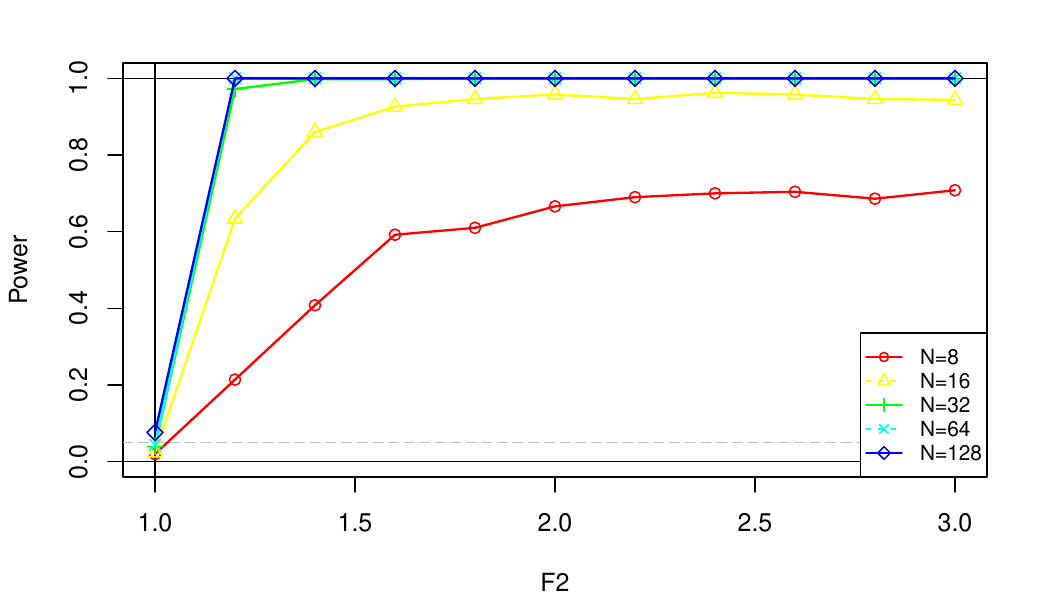}
\includegraphics[width=0.48\linewidth]{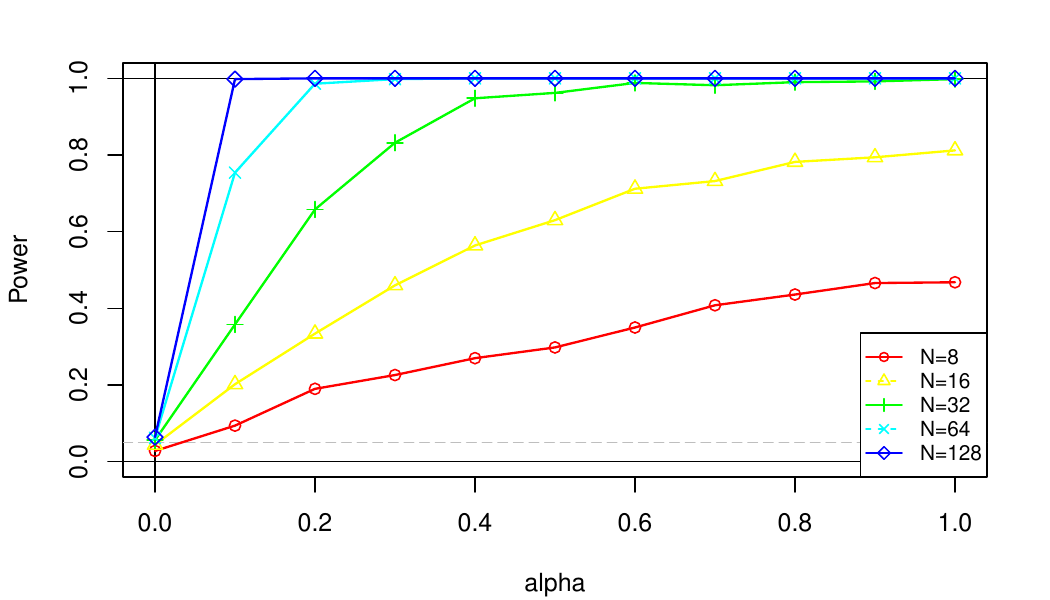}
\caption{Power of the tests on row heterogeneity (left) and overdispersion (right). For each network size and deviation from $\mathcal{H}_0$ considered ($F_2 > 1$ in the Poisson-BEDD for row heterogeneity, $\alpha > 0$ in the Overdispersed-Poisson-BEDD), we simulated $500$ networks and calculated the rejection rate of $\mathcal{H}_0$ using the statistic $Z^B_N$ (row heterogeneity) or $Z^C_N$ (overdispersion), and $95\%$ asymptotic confidence intervals. More details on the models used in these simulations are given in Appendix~\ref{app:simulations}.
}
\label{fig:degenerate:power_tests}
\end{figure}

\section{Extensions}

\subsection{Unbalanced asymptotic frameworks}
\label{sec:degenerate:other_asymptotic_frameworks}

In previous sections, we have assumed that $m_N + n_N = N$ and $m_N/N \rightarrow \rho \in ]0,1[$. It is in fact possible to extend all our results to any asymptotic behavior. In this section, let us only assume that $m_N \xrightarrow[N \rightarrow \infty]{} \infty$ and $n_N \xrightarrow[N \rightarrow \infty]{} \infty$ and see how it affects the limit distribution of $U_N$. 

The principal part of $U_N$ should be the dominant part of the variance. Remember that Proposition~\ref{prop:degenerate:variance_ustat_ep} states that
\begin{equation*}
    \mathbb{V}[U_{N}] = \sum_{(0,0) < (r,c) \le (p,q)} \frac{(m_N-r)!}{m_N!} \frac{(n_N-c)!}{n_N!} V^{(r,c)}.
\end{equation*}
We see that $\mathbb{V}[U_N]$ is the sum of the $p \times q$ terms of the form $\frac{(m_N-r)!}{m_N!} \frac{(n_N-c)!}{n_N!} V^{(r,c)}$. Each term behaves like $\frac{(m_N-r)!}{m_N!} \frac{(n_N-c)!}{n_N!} V^{(r,c)} \asymp m_N^{-r} n_N^{-c}$. The dominant part of $\mathbb{V}[U_N]$ consists of the terms $m_N^{-r} n_N^{-c}$ decreasing the slowest such that $V^{(r,c)} \neq 0$.

There is no equivalent to the previously defined order of degeneracy, but we can redefine principal degrees. Let the family of pairs $((r_\ell,c_\ell))_{1 \le \ell \le L}$ be such that $m_N^{r_1} n_N^{c_1} \asymp ... \asymp m_N^{r_L} n_N^{c_L}$ and $\mathbb{V}[U_N] \asymp \sum_{\ell = 1}^L \frac{V^{(r_\ell, c_\ell)}}{m_N^{r_\ell} n_N^{c_\ell}}$. We can call these pairs the \textit{principal degrees} of $U_N$, by analogy with the previous case. The quantity $\sum_{\ell = 1}^L P^{r_\ell,c_\ell}_{N}$ is called the \textit{principal part} of $U_N$. We call the \textit{principal support graphs} of $U_N$ the graphs $G$ such that 
\begin{itemize}
    \item $(v_1(G), v_2(G)) \in \{(r_\ell,c_\ell) : 1 \le \ell \le L \}$,
    \item $p^G \neq 0$.
\end{itemize}

\begin{example}
    Suppose $(m_N, n_N) = (N, \sqrt{N})$ and $V^{(0,1)} = 0$ but $V^{(0,2)} \neq 0$ and $V^{(1,0)} \neq 0$, then the principal degrees are $(1,0)$ and $(0,2)$ because $m_N = n_N^2 = N$ and $\mathbb{V}[U_N] = N^{-1}(V^{(1,0)} + V^{(0,2)})$. In this case, one valid choice of $\gamma(N)$ is $\gamma(N) = N$.
\end{example}

\begin{example}
    Suppose again that $(m_N, n_N) = (N, \sqrt{N})$, but this time $V^{(0,1)} = V^{(0,2)} = V^{(1,0)} = 0$. If $V^{(1,1)} \neq 0$ and $V^{(0,3)} \neq 0$, then the principal degrees are $(1,1)$ and $(0,3)$ because $m_N n_N = n_N^3 = N^{3/2}$. In this case, one valid choice of $\gamma(N)$ is $\gamma(N) = N^{3/2}$.
\end{example}

In this asymptotic framework, there is no reason that $N^{d/2}$ is the right normalization for the weak convergence of $U$-statistics. If the elements of $((r_\ell,c_\ell))_{1 \le \ell \le L}$ are the principal degrees of $U_N$, then there is a function $\gamma$ such that $m_N^{-r_\ell} n_N^{-c_\ell} \gamma(N) \xrightarrow[N \rightarrow \infty]{} \alpha_\ell$, where $\alpha_\ell > 0$ for all $1 \le \ell \le L$ and $\gamma(N) \mathbb{V}[U_N] = \sum_{1 \le \ell \le L} \alpha_\ell V^{(r_\ell, c_\ell)} + o(1)$. Next, we state the equivalent result to Theorem~\ref{th:degenerate:principal_part} in the new framework. The proof for this theorem is given in~\ref{app:degenerate:proof_principal_part_general}.

\begin{theorem}
    There is a random variable $W$ such that $\sqrt{\gamma(N)} \sum_{\ell = 1}^L P^{r_\ell,c_\ell}_{N} \xrightarrow[]{\mathcal{D}} W$ if and only if $\sqrt{\gamma(N)} (U_N - p^\emptyset) \xrightarrow[]{\mathcal{D}} W$.
    \label{th:degenerate:principal_part_general}
\end{theorem}

This theorem says that the limit distribution of $U_N - p^\emptyset$ renormalized by $\sqrt{\gamma(N)}$ is the same as that of its principal part $\sum_{\ell = 1}^L P^{r_\ell,c_\ell}_{N}$, renormalized by the same quantity. Therefore, similar to the initial framework, we shall investigate the asymptotic behavior of $U_N$ by studying its principal part.

In practice, one has to identify the principal part by finding the principal degrees of $U_N$. The principal degrees depend both on the kernel $h$ and the asymptotic behavior of $(m_N, n_N)$. After finding the principal degrees, then a function $\gamma(N)$ can be found. With $\gamma(N)$ and the principal degrees, the coefficients $\alpha_\ell$ can be calculated to yield an expression for the variance.

Now, we derive the equivalent to Theorem~\ref{th:degenerate:gaussian_theorem}, i.e. the convergence result when the principal support graphs of $U_N$ are connected. The proof of this theorem is given in Appendix~\ref{app:degenerate:proof_gaussian_theorem_general}.

\begin{theorem}
    If all principal support graphs of $U_N$ are connected, then
    \begin{equation*}
        \sqrt{\gamma(N)} (U_N - p^\emptyset) \xrightarrow[N \rightarrow \infty]{\mathcal{D}} \mathcal{N}(0, \sigma^2),
    \end{equation*}
    where 
    \begin{equation*}
        \sigma^2 = \sum_{\ell = 1}^L \alpha_\ell V^{(r_\ell, c_\ell)}.
    \end{equation*}
    \label{th:degenerate:gaussian_theorem_general}
\end{theorem}

Unsurprisingly, this theorem states that the limit distribution for $\sqrt{\gamma(N)} (U_N - p^\emptyset)$ is still a Gaussian like in Theorem~\ref{th:degenerate:gaussian_theorem}, but with a different expression for the variance. The new variance consists of terms associated of the principal degrees of $U_N$, depending on the behavior of $m_N$ and $n_N$.

\subsection{Asymmetric kernels}\label{sec:asymmetric}

Until now, we used symmetric kernels and defined $U$-statistics as sums over unordered tuples (eq.~\eqref{eq:ustat}). $U$-statistics can also be defined with arbitrary (possibly asymmetric) kernels via the ordered-tuples average: for any measurable $h : \mathcal{M}_{p,q}(\mathbb{R})\to\mathbb{R}$,
\[
\widetilde{U}_{m,n}(h):=\frac{1}{(m)_p (n)_q}
\sum_{(i_1,\dots,i_p)}\sum_{(j_1,\dots,j_q)}
h\!\big(Y_{(i_1,\dots,i_p);(j_1,\dots,j_q)}\big),
\]
where sums range over $p$- and $q$-tuples without replacement and $(x)_k=x(x-1)\cdots(x-k+1)$.

For any $h$, define its $S_p\times S_q$ symmetrization
\[
h^s(y):=\frac{1}{p!\,q!}\sum_{(\sigma_1, \sigma_2) \in \mathbb{S}_p \times \mathbb{S}_q} h\!\big(y_{\sigma_1,\sigma_2}\big).
\]
The associated (unordered) $U$-statistic used throughout the paper is
\[
U_{m,n}(h^s)
=\binom{m}{p}^{-1}\binom{n}{q}^{-1}\!
\sum_{\substack{\ib \in \mathcal{P}_p([ m ])\\ \jb \in \mathcal{P}_q([ n ])}} h^s\!\big(Y_{\ib,\jb}\big),
\]

\begin{lemma}[Ordered–unordered identity]\label{lem:ord-unord}
For all $m,n$ and all measurable $h$,
\[
\widetilde{U}_{m,n}(h)=\widetilde{U}_{m,n}(h^s)=U_{m,n}(h^s).
\]
\end{lemma}
\begin{proof}
    Group ordered tuples by their underlying unordered sets. Since $(m)_p\,(n)_q=\binom{m}{p}p!\,\binom{n}{q}q!$, each unordered pair $(\ib,\jb)$ contributes $\sum_{\sigma_1,\sigma_2} h\big(Y_{\sigma_1 \ib, \sigma_2 \jb}\big)=p!\,q!\,h^s(Y_{\ib,\jb})$, which yields the identity.
\end{proof} 

This convention matches exchangeability: our projections live on the latent $\sigma$-algebras $\sigma(H(G))$ indexed by bipartite graphs $G$ and depend only on which row/column/edge latents appear, not on their order. We therefore define principal support graphs of $h$ to be those of $h^s$, and compute principal degree, rates, and variances from $h^s$.

\begin{theorem}[CLT for arbitrary kernels]\label{thm:asym}
Let $h$ be arbitrary (not necessarily symmetric) and define principal support graphs via the projections of $h^s$ onto $L_2^*(G)$. If all principal support graphs are connected, then with the normalization $\gamma(N)$ dictated by the principal degree,
\[
\gamma(N)^{1/2}\Big(\widetilde{U}_{m_N,n_N}(h) - p^{\emptyset}\Big)\;\xrightarrow[N\to\infty]{}\; \mathcal{N}(0,\sigma^2),
\]
with the same variance expression as Theorem~\ref{th:degenerate:gaussian_theorem} using $h^s$, and $p^{\emptyset} = \mathbb{E}[h^s(Y_{\ib,\jb})] = \mathbb{E}[h(Y_{\ib,\jb})]$.
\end{theorem}

\begin{proof}[Proof]
By Lemma~\ref{lem:ord-unord}, $\widetilde{U}_{m,n}(h)=U_{m,n}(h^s)$. Apply Theorem~\ref{th:degenerate:gaussian_theorem} to $h^s$. Symmetrization does not change $\sigma(H(G))$, projections, principal graphs, or the variance computation. Finally, conclude with $\mathbb{E}[h^s(Y_{\ib,\jb})] = \mathbb{E}[h(Y_{\ib,\jb})]$.
\end{proof}

For context, \citet{janson1991asymptotic} study an order-sensitive definition that evaluates an unsymmetric kernel once per increasing tuple (no symmetrization). If one adopts that definition here, the ``connectedness implies Gaussian'' conclusion and the normalization from the principal degree are unchanged: only the variance is reweighted by fixed positional factors tied to the chosen ordering convention (see \citealp{janson1991asymptotic}). As these weights are not exchangeability-invariant, and add no substantive insight for our purposes, we do not pursue them further.

\section{Conclusion}

In this paper, we have introduced a new orthogonal decomposition for $U$-statistics on dissociated RCE matrices, providing a framework to characterize their asymptotic behavior. This decomposition relies on partitioning the probability space into orthogonal subspaces generated by specific sets, termed graph sets, of AHK variables. The asymptotic behavior of a $U$-statistic is determined by its principal part, which is composed of the leading non-zero terms of the decomposition. The graphs corresponding to these terms are referred to as the principal support graphs.

Principal support graphs play a central role in determining the asymptotic behavior of $U$-statistics. We have shown that all principal support graphs of a given $U$-statistic have the same number of nodes, which defines the principal degree. The principal degree corresponds to the order of degeneracy in traditional $U$-statistics of i.i.d. variables, determining the rate of convergence to the limit distribution. For that reason, degeneracy seems to be a desirable property of $U$-statistics for statistical applications, where a faster rate of convergence improves the precision of estimators, leading to tighter confidence intervals and more powerful hypothesis tests.

Nevertheless, identifying limit distributions in degenerate cases remains a challenging task, even for $U$-statistics of i.i.d. variables. For RCE matrices, we have shown that a simple assumption on the topology of principal support graphs, namely connectedness, ensures that the limit distribution is Gaussian. When this assumption holds, we obtain a simple limit distribution, but also, in degenerate cases, a rate of convergence exceeding $\sqrt{N}$. While a similar result has been discussed in \cite{janson1991asymptotic}, this highlights a significant difference with $U$-statistics of i.i.d. variables, for which degenerate cases yield more complex limit distributions, expressed as higher-degree polynomials of Gaussians with no straightforward expressions \citep{rubin1980asymptotic, lee1990u}. Future research could focus on characterizing the limit distributions of $U$-statistics on RCE matrices under different assumptions. In particular, the case where principal support graphs exhibit multiple connected components presents an intriguing avenue for exploration, though we anticipate that the resulting limit distributions would be more complex.

\section*{Acknowledgements}

The author thanks Sophie Donnet, Fran\c{c}ois Massol and St\'ephane Robin for many fruitful discussions and insights. This work was funded by the grant ANR-18-CE02-0010-01 of the French National Research Agency ANR (project EcoNet) and a grant from R\'egion \^Ile-de-France.

\newpage

\begin{appendix}

\newpage

\section{Cheat sheet} 
\label{app:cheat_sheet}

Essential notions and notation in the paper.
\vspace{1em}

\noindent\textbf{Main assumptions on data.} \\
Row-column exchangeable matrix: for any pair of permutations $(\sigma_1, \sigma_2)$, $(Y_{\sigma_1(i)\sigma_2(j)})_{i,j} \overset{\mathcal{D}}{=} Y$. \\
Dissociation: for any $(m,n)$, $(Y_{ij})_{i \le m, j \le n}$ is independent of $(Y_{ij})_{i > m, j > n}$. \\
Asymptotic regime: Same row and column growth. 
\vspace{1em}

\noindent\textbf{Main decomposition.}\\
Probability space decomposition: $L_2(G) \bigoplus^\perp_{F \subseteq G} L_2^*(F)$.\\
$U$-statistic decomposition: $U_{m,n} = \sum_{(0,0) \le (r,c) \le (p,q)} P^{r,c}_{m,n} = \sum_{G \in \Gamma^{-}_{r,c}} \left[(p-r)! (q-c)! \lvert \Aut(G) \rvert \right]^{-1} \widetilde{P}^{G}_{m,n}$.

\vspace{1em}

\noindent\textbf{Key concepts.} \\
Principal support graphs: Smallest graphs $G$ with nonzero contribution $\widetilde p^{G}_{\ib,\jb}$. \\
Principal degree $d$: Common value of $|V(G)|$ over principal support graphs. \\
Connectedness: All the principal support graphs are connected.
\vspace{1em}

\noindent\textbf{Rules to characterize the asymptotic behavior of $U$-statistics.} \\
Rate and degeneracy: Normalize by $N^{d/2}$. Nondegenerate if $d=1$. Degenerate if $d>1$. \\
Gaussianity: Under connectedness (always satisfied when nondegenerate), the centered and scaled statistic is asymptotically normal. 

\newpage

\noindent\textbf{Main notations (from Tables~\ref{tab:notations_intro},~\ref{tab:notations_1} and~\ref{tab:notations_2}).}

\small
\begin{tabular}{@{}ll@{}}

$Y=(Y_{ij})$ & Infinite dissociated row-column exchangeable (RCE) matrix \\
$m,n$ & Numbers of rows and columns used to compute $U_{m,n}$ \\
$p,q$ & Kernel dimensions, size of the submatrix given to $h$ \\
$r,c$ & Graph dimensions, numbers of row and column nodes \\
$\ib, \jb$ & Sets of row and column node indices \\
$h:\mathcal M_{p,q}\to\mathbb R$ & Symmetric kernel in $p\times q$ entries \\
$U_{m,n}$ & $U$-statistic on the leading $m \times n$ submatrix of $Y$ \\
$U_{m,n}^h$ & Same as above, emphasizing the kernel $h$ \\
$N$ & Total size $m+n$ (used for rates) \\
$(m_N,n_N)$ & Sequence of matrix sizes in the asymptotic framework \\
$U_N^h$ & Shorthand for $U_{m_N,n_N}^h$ along the sequence $(m_N,n_N)$ \\
$\rho$ & $\lim_{N\to\infty} m_N/N \in (0,1)$ with $N=m_N+n_N$ \\

AHK variables & Latents variables $(\xi_i),(\eta_j),(\zeta_{ij})_{i,j}$ in the Aldous--Hoover--Kallenberg representation \\
$K_{\ib,\jb}$ & Complete bipartite graph with row set $\ib$ and column set $\jb$, edge set $\ib\times\jb$ \\
$K_{r,c}$ & Shorthand for $K_{[r],[c]}$ with $[r]=\{1,\dots,r\}$ and $[c]=\{1,\dots,c\}$ \\
Minimal set & Only one representative per isomorphism class in the set \\
$\Gamma_{r,c}$ & Minimal set of subgraphs of $K_{r,c}$ with $r$ row and $c$ column nodes \\
$\Gamma_{p,q}^{-}$ & Minimal set of subgraphs of $K_{p,q}$ with at most $p$ row and $q$ column nodes \\
$|V(G)|$ & Total nodes of $G$, equals $|V_1(G)|+|V_2(G)|$ ($r+c$ for $G\in\Gamma_{r,c}$) \\
$G_1 \sim G_2$ & $G_1$ and $G_2$ are isomorphic \\
$\mathrm{Aut}(G)$ & Automorphism group of $G$ \\
$H(G)$ & AHK variables attached to graph $G$: \\
& \qquad $H(G) = ((\xi_i)_{i\in V_1(G)},(\eta_j)_{j\in V_2(G)},(\zeta_{ij})_{(i,j)\in E(G)})$ \\

$L_2(G)$ & Space of square-integrable variables measurable with respect to $\sigma(H(G))$ \\
$L_2^\ast(G)$ & Subspace of $L_2(G)$ with $\mathbb E[X\mid H(F)]=0$ for all $F\subset G$ \\

$p^F$ & Orthogonal projection on $L_2^\ast(F)$: for $X\in L_2(G)$, $F \subseteq G$, \\
& \qquad $p^\varnothing(X)=\mathbb E[X]$ and
$p^F(X)=\mathbb E[X\mid H(F)]-\sum_{F'\subset F} p^{F'}(X)$ \\

$\widetilde p^{G}_{\ib,\jb}$ & Symmetrized, shape-indexed increment for $G\in\Gamma_{r,c}$:
$\sum_{F\subseteq K_{\ib,\jb},\, F\sim G} p^{F}$ \\
$\widetilde{P}^{G}_{m,n}$ & Contribution of the isomorphism class of $G$: \\
& \qquad $\widetilde{P}^{G}_{m,n} = \binom{m}{p}^{-1} \binom{n}{q}^{-1} \sum_{\ib\in\mathcal P_p([m])} \sum_{\jb\in\mathcal P_q([n])}  \widetilde{p}^{G}_{\ib, \jb}$ \\
$P^{r,c}_{m,n}$ & Contribution of size $(r,c)$ in $U_{m,n}$: \\
& \qquad $P^{r,c}_{m,n} = \sum_{G \in \Gamma_{r,c}} \left[(p-r)! (q-c)! \lvert \Aut(G) \rvert \right]^{-1} \widetilde{P}^{G}_{m,n}$ 
\end{tabular}

\normalsize

\newpage

\section{Proofs for Section~\ref{sec:degenerate:decomposition_probability_space}}
\label{app:degenerate:proof_decomposition_probability_space}

\begin{proof}[Proof of Proposition~\ref{prop:degenerate:l2_direct_sum}]
    We show that for all $F$ and $F'$ such that $F' \subset F$, we have that $\mathbb{E}[p^F(X) \mid H(F')] = 0$ by induction on $F$. First, notice that $p^\emptyset(X) = \mathbb{E}[X] \in L_2^*(\emptyset)$ being the space of constant variables. Next, fix $F$ and suppose that the induction hypothesis is true for all $\overline{F} \subset F$, i.e. for all $\overline{F}$ and $F'$ such that $F' \subseteq \overline{F} \subset F$, we have that $\mathbb{E}[p^{\overline{F}}(X) \mid H(F')] = 0$. Now we can calculate for all $F' \subset F$,
\begin{equation*}
\begin{split}
    \mathbb{E}[p^F(X) \mid H(F')] &= \mathbb{E}[\mathbb{E}[X \mid H(F) ]\mid H(F')] - \sum_{\overline{F} \subset F} \mathbb{E}[p^{\overline{F}}(X) \mid H(F')] \\
    &= \mathbb{E}[ X \mid H(F')] - p^{F'}(X) - \sum_{\substack{\overline{F} \subset F \\ \overline{F} \neq F'}} \mathbb{E}[p^{\overline{F}}(X) \mid H(F')] \\
    &= \sum_{\overline{F} \subset F'} \mathbb{E}[p^{\overline{F}}(X) \mid H(F')] - \sum_{\substack{\overline{F} \subset F \\ \overline{F} \neq F'}} \mathbb{E}[p^{\overline{F}}(X) \mid H(F')] \\
    &= - \sum_{\substack{\overline{F} \subset F \\ \overline{F} \not\subset F'}} \mathbb{E}[p^{\overline{F}}(X) \mid H(F')] \\
    &= - \sum_{\substack{\overline{F} \subset F \\ \overline{F} \not\subset F'}} \mathbb{E}[p^{\overline{F}}(X) \mid H(F' \cap \overline{F})].
\end{split}
\end{equation*}     
By the induction hypothesis, all the terms of this sum are equal to 0, which concludes the proof by induction.
\end{proof}

\section{Proofs for Section~\ref{sec:degenerate:variance}}
\label{app:degenerate:proof_variance}

\begin{proof}[Proof of Proposition~\ref{prop:degenerate:variance_ustat_ep}]
\begin{equation*}
\begin{split}
    \mathbb{V}[U_{m,n}] &= \sum_{(0,0) < (r,c) \le (p,q)} \mathbb{V}[P^{r,c}_{m,n}] \\
    &= \sum_{(0,0) < (r,c) \le (p,q)}  \binom{m}{p}^{-2} \binom{n}{q}^{-2} \sum_{\substack{\ib,\ib' \in \mathcal{P}_p([ m ])\\ \jb,\jb' \in \mathcal{P}_q([ n ])}} \sum_{\substack{G \subseteq K_{\ib, \jb}, G' \subseteq K_{\ib', \jb'} \\ (v_1(G), v_2(G)) = (r,c) \\ (v_1(G'), v_2(G')) = (r,c)}}   \Cov(p^{G}, p^{G'}) \\
    &=  \sum_{(0,0) < (r,c) \le (p,q)}  \binom{m}{p}^{-1} \binom{n}{q}^{-1}  \binom{m-r}{p-r} \binom{n-c}{q-c} r! \binom{p}{r} c! \binom{q}{c} \sum_{G \in \Gamma_{r,c}} \lvert \Aut(G) \rvert^{-1} \mathbb{V}[p^{G}] \\
    &=  \sum_{(0,0) < (r,c) \le (p,q)}  \binom{m}{r}^{-1} \binom{n}{c}^{-1} r! \binom{p}{r}^2 c! \binom{q}{c}^2 \sum_{G \in \Gamma_{r,c}} \lvert \Aut(G) \rvert^{-1} \mathbb{E}[(p^G)^2] \\
    &= \sum_{(0,0) < (r,c) \le (p,q)} \frac{(m-r)!}{m!} \frac{(n-c)!}{n!} V^{(r,c)} \\
\end{split}
\end{equation*}
\end{proof}

\begin{proof}[Proof of Lemma~\ref{lem:degenerate:var_pg}]
    Let $G \in \Gamma_{r,c}$.
    \begin{equation*}
    \begin{split}
        \mathbb{V}[\widetilde{P}^{G}_{m,n}] &= \binom{m}{p}^{-2} \binom{n}{q}^{-2} \sum_{\substack{\ib,\ib' \in \mathcal{P}_p([ m ])\\ \jb,\jb' \in \mathcal{P}_q([ n ])}} \Cov(\widetilde{p}^{G}_{\ib, \jb}, \widetilde{p}^{G}_{\ib', \jb'}) \\
        &= \binom{m}{p}^{-2} \binom{n}{q}^{-2} \sum_{\substack{\ib,\ib' \in \mathcal{P}_p([ m ])\\ \jb,\jb' \in \mathcal{P}_q([ n ])}} \sum_{\Phi, \Phi' \in \mathbb{S}_p \times \mathbb{S}_q} \Cov(p^{\Phi G_{\ib,\jb}}, p^{\Phi' G_{\ib',\jb'}}) \\
    \end{split}
    \end{equation*}
where for all $(\ib, \jb) \in \mathcal{P}_p([ m ]) \times \mathcal{P}_q([ n ])$, $G_{\ib, \jb}$ is any graph of $K_{\ib, \jb}$ which is isomorphic to $G$. 
    
    Now see that if $\Phi G_{\ib,\jb} \neq \Phi' G_{\ib',\jb'}$, then $\Cov(p^{\Phi G_{\ib,\jb}}, p^{\Phi' G_{\ib',\jb'}}) = 0$. Otherwise $\Phi G_{\ib,\jb} = \Phi' G_{\ib',\jb'}$, then $\Cov(p^{\Phi G_{\ib,\jb}}, p^{\Phi' G_{\ib',\jb'}}) = \mathbb{V}[p^{G}] = \mathbb{E}[(p^{G})^2]$. So, it follows that
\begin{equation*}
        \mathbb{V}[\widetilde{P}^{G}_{m,n}] = \binom{m}{p}^{-2} \binom{n}{q}^{-2} \sum_{\substack{\ib,\ib' \in \mathcal{P}_p([ m ])\\ \jb,\jb' \in \mathcal{P}_q([ n ])}} \sum_{\Phi, \Phi' \in \mathbb{S}_p \times \mathbb{S}_q} \mathds{1}(\Phi G_{\ib,\jb} = \Phi' G_{\ib',\jb'}) \mathbb{E}[(p^{G})^2].
\end{equation*}
Finally, applying Lemma~\ref{lem:degenerate:count_equal_graphs}, we have
\begin{equation*}
\begin{split}
        \mathbb{V}[\widetilde{P}^{G}_{m,n}] &= \binom{m}{p}^{-2} \binom{n}{q}^{-2} \frac{m! (m-r)!}{(m-p)!^2} \frac{n! (n-c)!}{(n-q)!^2} \lvert \Aut(G) \rvert \mathbb{E}[(p^{G})^2] \\
        &= \frac{(m-r)!}{m!} \frac{(n-c)!}{n!} p!^2 q!^2 \rvert \Aut(G) \lvert \mathbb{E}[(p^{G})^2]. 
\end{split}
\end{equation*}
\end{proof}

\begin{proof}[Proof of Lemma~\ref{lem:degenerate:count_equal_graphs}]
    First, fix $\ib_1, \jb_1, \Phi_1$. Write $G^1 := \Phi_1 G^1_{\ib_1, \jb_1}$. We count the number of picks for $\ib_2, \jb_2, \Phi_2$ such that $\Phi_2 G^2_{\ib_2, \jb_2} = G^1$. 
    
    $\ib_2$ and $\jb_2$ must contain the $r$ row nodes and the $c$ column nodes of $G^1$ and $\Phi_2$ must place these nodes in the same order than in $G^1$, or belong to its automorphism group. This happens for $\binom{m-r}{p-r} \binom{n-c}{q-c}$ picks for $(\ib_2, \jb_2)$ and for each, there are $(p-r)! (q-c)! \lvert \Aut(G) \rvert$ valid picks for $\Phi_2$. 
    
    This happens for all $\binom{m}{p} \binom{n}{q}$ picks of $(\ib_1, \jb_1)$ and $p! q!$ picks of $\Phi_1$. Therefore,  
\begin{equation*}
\begin{split}
    \sum_{\substack{\ib_1, \ib_2 \in \mathcal{P}_p([ m ]) \\ \jb_1, \jb_2 \in \mathcal{P}_q([ n ])}} \sum_{\Phi_1, \Phi_2 \in \mathbb{S}_p \times \mathbb{S}_q} &\mathds{1}(\Phi_1 G^1_{\ib_1, \jb_1} = \Phi_2 G^2_{\ib_2, \jb_2}) \\
    &= \binom{m}{p} \binom{n}{q} \binom{m-r}{p-r} \binom{n-c}{q-c} p! q! (p-r)! (q-c)! \lvert \Aut(G) \rvert,
\end{split}
\end{equation*}    
which develops into the form given by this lemma.
\end{proof}

\section{Proofs for Section~\ref{sec:degenerate:limit_distribution}}
\label{app:degenerate:proof_principal_part}

\begin{proof}[Proof of Theorem~\ref{th:degenerate:principal_part}]
    Since, $d-1$ is the order of degeneracy, we have $P^{r,c}_{N} = 0$ for all $(r,c)$ such that $r+c < d$. Therefore, we have $U_N - p^\emptyset - \sum_{\substack{(0,0) \le (r,c) \le (p,q) \\ r+c = d}} P^{r,c}_{N} = \sum_{\substack{(0,0) \le (r,c) \le (p,q) \\ r+c > d}} P^{r,c}_{N}$. So
\begin{equation*}
\begin{split}
    \mathbb{V}\left[N^{d/2} \left(U_N - p^\emptyset - \sum_{\substack{(0,0) \le (r,c) \le (p,q) \\ r+c = d}} P^{r,c}_{N}\right)\right] &= N^{d} \sum_{\substack{(0,0) \le (r,c) \le (p,q) \\ r+c > d}} \mathbb{V}[P^{r,c}_{N}]  \\
    &= N^{d} \sum_{\substack{(0,0) \le (r,c) \le (p,q) \\ r+c > d}} \frac{(m_N - r)!}{m_N!} \frac{(n_N - c)!}{n_N!} V^{(r,c)}. \\
\end{split}    
\end{equation*}  
But for all $(r,c)$, we have $\frac{(m_N - r)!}{m_N!} \frac{(n_N - c)!}{n_N!} = O(N^{-r-c})$, therefore
\begin{equation*}
\begin{split}
    \mathbb{V}\left[N^{d/2} \left(U_N - p^\emptyset - \sum_{\substack{(0,0) \le (r,c) \le (p,q) \\ r+c = d}} P^{r,c}_{N}\right)\right] &= N^d \times O\left(\sum_{\substack{(0,0) \le (r,c) \le (p,q) \\ r+c > d}} N^{-r-c}\right) \\
    &= N^d \times o(N^{-d}) \\
    &= o(1).
\end{split}    
\end{equation*}     

Finally, this implies that $N^{d/2} (U_N - p^\emptyset) = N^{d/2} \sum_{\substack{(0,0) \le (r,c) \le (p,q) \\ r+c = d}} P^{r,c}_{N} + o_P(1)$, which proves the theorem.
\end{proof}

\section{Proofs for Section~\ref{sec:degenerate:gaussian_case}}
\label{app:degenerate:proof_gaussian_case}

\begin{proof}[Proof of Lemma~\ref{lem:degenerate:graph_pairs}]
    For some $\ell \in [ K ]$, denote $G_{1:k}^{(-\ell)} = \cup_{\substack{i=1 \\ i \neq \ell}}^k G_i$. We have
\begin{equation*}
\begin{split}
    \mathbb{E}[\prod_{k = 1}^K p^{G_k}]  &= \mathbb{E}[\mathbb{E}[\prod_{k = 1}^K p^{G_k} \mid H(G_{1:K}^{(-\ell)})]] \\ 
    &= \prod_{\substack{k = 1 \\ k \neq \ell}}^K p^{G_k} \mathbb{E}[\mathbb{E}[p^{G_\ell}\mid H(G_{1:K}^{(-\ell)})]] \\
    &= \prod_{\substack{k=1 \\ k \neq \ell}}^K p^{G_k} \mathbb{E}[\mathbb{E}[p^{G_\ell}\mid H(G_\ell \cap G_{1:K}^{(-\ell)})]].
\end{split}
\end{equation*}
Suppose there is a vertex or edge of a $G_\ell$ that does not belong to any other $G_k$, $k \neq \ell$. In this case, $G_\ell \cap G_{1:K}^{(-\ell)} \subset G_\ell$, so $\mathbb{E}[p^{G_\ell}\mid H(G_\ell \cap G_{1:K}^{(-\ell)})] = 0$, which proves the first result.

From that result, if $\mathbb{E}[\prod_{k = 1}^K p^{G_k}] \neq 0$ and no vertex belongs to more than two of $G_1, ..., G_K$, then each vertex and edge belongs to exactly two of them. This also means that every connected component must belong to exactly two of them. Therefore, if all graphs are connected, then these graphs coincide in pairs.
\end{proof}

\begin{proof}[Proof of Lemma~\ref{lem:degenerate:convergence_gaussian_projection}]
    Let $a_k$ be nonnegative integers. For all $(\ib, \jb) \in \mathcal{P}_p([ m_N ]) \times \mathcal{P}_q([ n_N ])$, let $G_{k, \ib, \jb}$ be a graph of $K_{\ib, \jb}$ which is isomorphic to $G_k$. Then
    
    \begin{equation}
    \begin{split}
        \mathbb{E}\Big[\prod_{k=1}^{K} (&m_N^{r_k/2} n_N^{c_k/2} \widetilde{P}^{G_k}_{N})^{a_k}\Big] \\
        &= m_N^{\sum_{k=1}^K a_k r_k/2} \binom{m_N}{p}^{-\sum_{k=1}^K a_k} n_N^{\sum_{k=1}^K a_k c_k/2} \binom{n_N}{q}^{-\sum_{k=1}^K a_k} \mathbb{E}\left[\prod_{k=1}^{K} \left(\sum_{\substack{\ib_k \in \mathcal{P}_p([ m_N ])\\ \jb_k \in \mathcal{P}_q([ n_N ])}} \widetilde{p}^{G_k}_{\ib_k, \jb_k}\right)^{a_k}\right] ,
        \label{eq:degenerate:develop_convergence_projection}
    \end{split}
    \end{equation}
    where we can develop    
    \begin{equation*}
    \begin{split}
        \mathbb{E}\left[\prod_{k=1}^{K} \left(\sum_{\substack{\ib_k \in \mathcal{P}_p([ m_N ])\\ \jb_k \in \mathcal{P}_q([ n_N ])}} \widetilde{p}^{G_k}_{\ib_k, \jb_k}\right)^{a_k}\right] &= \sum_{\substack{\ib^\ell_k \in \mathcal{P}_p([ m_N ])\\ \jb^\ell_k \in \mathcal{P}_q([ n_N ])}} \mathbb{E}\left[\prod_{k=1}^{K} \prod_{\ell=1}^{a_k} \widetilde{p}^{G_k}_{\ib^\ell_k, \jb^\ell_k}\right] \\
        &= \sum_{\substack{\ib^\ell_k \in \mathcal{P}_p([ m_N ])\\ \jb^\ell_k \in \mathcal{P}_q([ n_N ])}} \sum_{\Phi^k_\ell \in \mathbb{S}_p \times \mathbb{S}_q} \mathbb{E}\left[\prod_{k=1}^{K} \prod_{\ell=1}^{a_k} p^{\Phi^\ell_k G_{k, \ib^\ell_k, \jb^\ell_k}}\right]. \\
    \end{split}
    \end{equation*}
    
    Lemma~\ref{lem:degenerate:graph_pairs} states that $\mathbb{E}[\prod_{k=1}^{K} \prod_{\ell=1}^{a_k} p^{\Phi^\ell_k G_{k, \ib^\ell_k, \jb^\ell_k}}] \neq 0$ if and only if either all the $\Phi^\ell_k G_{k, \ib^\ell_k, \jb^\ell_k}$ coincide in pairs (and only in pairs), or no vertex appears in exactly one of these graphs and at least one vertex appears in at least three. 
    
    In the second case, assume without loss of generality that a row node appears in three graphs. Then $G^*_{(\ib^\ell_k), (\jb^\ell_k)} := \cup_{k=1}^K \cup_{j = 1}^{a_k} \Phi^\ell_k G_{k, \ib^\ell_k, \jb^\ell_k}$ has $v_1(G^*_{(\ib^\ell_k), (\jb^\ell_k)})$ row nodes and $v_2(G^*_{(\ib^\ell_k), (\jb^\ell_k)})$ column nodes, where $\max r_k \le v_1(G^*_{(\ib^\ell_k), (\jb^\ell_k)}) \le \sum_{k=1}^K a_k r_k /2 - 1$ and $\max c_k \le v_2(G^*_{(\ib^\ell_k), (\jb^\ell_k)}) \le \sum_{k=1}^K a_k c_k /2 - 1$ (we have $\max r_k \le \sum_{k=1}^K a_k r_k /2 - 1$ and $\max c_k \le \sum_{k=1}^K a_k c_k /2 - 1$, else $\mathbb{E}[\prod_{k=1}^{K} \prod_{\ell=1}^{a_k} p^{\Phi^\ell_k G_{k, \ib^\ell_k, \jb^\ell_k}}] = 0$).
    
    Let $(\max r_k, \max c_k) \le (r*, c*) \le (p,q)$. Let us count the number of terms of the sum such that $v_1(G^*_{(\ib^\ell_k), (\jb^\ell_k)}) = r^*$ and $v_2(G^*_{(\ib^\ell_k), (\jb^\ell_k)}) = c^*$. There are exactly $\binom{m_N}{r^*} \binom{n_N}{c^*}$ ways to pick $r^*$ row nodes and $c^*$ nodes for $G^*_{(\ib^\ell_k), (\jb^\ell_k)}$. Now, for a specific set of $r^*$ row nodes and $c^*$ column nodes, for each $1 \le k \le K$, $1 \le \ell \le a_k$, there are $\binom{r^*}{r_k} \binom{c^*}{c_k} \binom{m_N-r^*}{p - r_k} \binom{n_N-c^*}{q - c_k}$ ways to pick $(\ib^\ell_k, \jb^\ell_k)$ such that the nodes of $G_{k, \ib^\ell_k, \jb^\ell_k}$ are contained in the $r^*$ specific row nodes and $c^*$ specific column nodes. Therefore, there are at most $p!q!\binom{r^*}{r_k} \binom{c^*}{c_k} \binom{m_N-r^*}{p - r_k} \binom{n_N-c^*}{q - c_k}$ picks for $(\ib^\ell_k, \jb^\ell_k)$ and $\Phi_k^\ell$. Finally, the number of terms is smaller than
\begin{equation*}
\begin{split}
    B^{r^*, c^*}_N &:= \binom{m_N}{r^*} \binom{n_N}{c^*} \prod_{k = 1}^K \prod_{\ell = 1}^{a_k} p! q! \binom{r^*}{r_k} \binom{m_N-r^*}{p-r_k} \binom{n_N-c^*}{q-c_k} \\
    &= \binom{m_N}{r^*} \binom{n_N}{c^*} \prod_{k = 1}^K \left[ p! q! \binom{r^*}{r_k} \binom{m_N-r^*}{p-r_k} \binom{n_N-c^*}{q-c_k}\right]^{a_k} \\
    &= O\left(m_N^{r^*} n_N^{c^*} \prod_{k = 1}^K \left[ m_N^{p-r_k} n_N^{q-c_k}\right]^{a_k} \right) \\
    &= O\left(m_N^{r^* + \sum_{k=1}^K a_k(p-r_k)} n_N^{c^* + \sum_{k=1}^K a_k(q-c_k)} \right). \\
\end{split}
\end{equation*}
The total number of these terms is 
\begin{equation*}
\begin{split}
    B_N &\le \sum_{(\max r_k, \max c_k) \le (r^*,c^*) \le (\sum_{k=1}^K a_k r_k/ 2 - 1, \sum_{k=1}^K a_k c_k/ 2)} B^{r^*, c^*}_N \\
    &= O(B^{\sum_{k=1}^K a_k r_k/ 2 - 1, \sum_{k=1}^K a_k c_k/ 2}_N) \\
    &= O\left(m_N^{\sum_{k=1}^K a_k(p-r_k/2) - 1} n_N^{\sum_{k=1}^K a_k(q-c_k/2)} \right) \\
    &= o\left(m_N^{\sum_{k=1}^K a_k(p-r_k/2)} n_N^{\sum_{k=1}^K a_k(q-c_k/2)} \right).
\end{split}
\end{equation*}
We notice that the contribution of these terms is $o(1)$ in equation~\eqref{eq:degenerate:develop_convergence_projection}.

Now, there remains the terms of the first case, where the $\Phi^\ell_k G_{k, \ib^\ell_k, \jb^\ell_k}$ coincide in pairs. Note that since the $G_k$ are non-isomorphic, only graphs arising for the permutations of the same graph $G_k$ can coincide. Therefore, the $a_k$ are necessarily even. Furthermore, for each $k$, there are $a_k/2$ different pairs of coinciding graphs $\Phi^\ell_k G_{k, \ib^\ell_k, \jb^\ell_k}$. There are $\frac{a_k !}{2^{a_k/2} (a_k/2)!}$ ways to partition a set of $a_k$ graphs into $a_k/2$ pairs. 

Fix $k, \ell_1, \ell_2$. The number of picks for $\ib^{\ell_1}_k, \jb^{\ell_1}_k, \ib^{\ell_2}_k, \jb^{\ell_2}_k, \Phi^{\ell_1}, \Phi^{\ell_2}$ such that $\Phi^{\ell_1}_k G_{k, \ib^{\ell_1}_k, \jb^{\ell_1}_k} = \Phi^{\ell_2}_k G_{k, \ib^{\ell_2}_k, \jb^{\ell_2}_k}$ is given by Lemma~\ref{lem:degenerate:count_equal_graphs}. Accounting for all $a_k/2$ pairs of the type $(\ell_1, \ell_2)$, there are   
\begin{equation*}
    \frac{m_N! (m_N-r_k)!}{(m_N-p)!^2} \frac{n_N! (n_N-c_k)!}{(n_N-q)!^2} \lvert \Aut(G) \rvert.
\end{equation*}
    
Therefore, taking into account the number of possible pairings and the picks for all $1 \le k \le K$, $1 \le \ell \le a_k$, there are 
\begin{equation*}
\begin{split}
    A_N &= \prod_{k=1}^K \frac{a_k !}{2^{a_k/2} (a_k/2)!} \left( \frac{m_N! (m_N-r_k)!}{(m_N-p)!^2} \frac{n_N! (n_N-c_k)!}{(n_N-q)!^2} \lvert \Aut(G_k) \rvert \right)^{a_k/2} \\
    &= m_N^{\sum_{k=1}^K a_k(r_k/2-p)} n_N^{\sum_{k=1}^K a_k(c_k/2-q)} \prod_{k=1}^K \frac{a_k !}{2^{a_k/2} (a_k/2)!} \lvert \Aut(G_k) \rvert^{a_k/2} \\
    &\quad + o\left( m_N^{\sum_{k=1}^K a_k(r_k/2-p)} n_N^{\sum_{k=1}^K a_k(c_k/2-q)} \right). \\
\end{split}
\end{equation*}
Each of these $A_N$ terms is equal to $\mathbb{E}\left[\prod_{k=1}^{K} \prod_{\ell=1}^{a_k} p^{\Phi^\ell_k G_{k, \ib^\ell_k, \jb^\ell_k}}\right] = \prod_{k=1}^{K} \mathbb{E}[(p^{G_k})^2]^{a_k/2}$.

In conclusion, if all the $a_k$ are even, then
\begin{equation*}
\begin{split}
    \mathbb{E}\left[\prod_{k=1}^{K} (m_N^{r_k/2} n_N^{c_k/2} \widetilde{P}^{G_k}_{N})^{a_k}\right] &= m_N^{\sum_{k=1}^K a_k r_k/2} \binom{m_N}{p}^{-\sum_{k=1}^K a_k} n_N^{\sum_{k=1}^K a_k c_k/2} \binom{n_N}{q}^{-\sum_{k=1}^K a_k} \\
    &\quad \times A_N \prod_{k=1}^{K} \mathbb{E}[(p^{G_k})^2]^{a_k/2} \\
    &= (p! q!)^{\sum_{k=1}^K a_k} \prod_{k=1}^K \frac{a_k !}{2^{a_k/2} (a_k/2)!} \lvert \Aut(G_k) \rvert^{a_k/2} \mathbb{E}[(p^{G_k})^2]^{a_k/2}  \\
    &= \prod_{k=1}^K \frac{a_k !}{2^{a_k/2} (a_k/2)!} \left(p!^2 q!^2 \lvert \Aut(G_k) \rvert \mathbb{E}[(p^{G_k})^2]\right)^{a_k/2},
\end{split}
\end{equation*}
and in the general case, 
\begin{equation}
\begin{split}
    \mathbb{E}\Big[\prod_{k=1}^{K} (&m_N^{r_k/2} n_N^{c_k/2} \widetilde{P}^{G_k}_{N})^{a_k}\Big] \\
    &= \begin{cases} \prod_{k=1}^K \frac{a_k !}{2^{a_k/2} (a_k/2)!} \left(p!^2 q!^2 \lvert \Aut(G_k) \rvert \mathbb{E}[(p^{G_k})^2]\right)^{a_k/2} & \text{if all $a_k$ are even,} \\ 0 & \text{if at least one $a_k$ is odd.} \end{cases}
\end{split}
    \label{eq:degenerate:lemma_moments}
\end{equation}
Else, if there is at least one odd $a_k$, we have $\mathbb{E}[\prod_{k=1}^{K} (m_N^{r_k/2} n_N^{c_k/2} \widetilde{P}^{G_k}_{N})^{a_k}] = 0$.
    
    We remind that the moment of order $a$ of a Gaussian variable $X$ with mean $0$ and variance $\sigma^2$ is
    \begin{equation*}
        \mathbb{E}[X^a] = \begin{cases} \frac{a!}{2^{a/2} (a/2)!} \sigma^{a} & \text{if $a$ is even,} \\
        0 & \text{if $a$ is odd.} \\
        \end{cases}
    \end{equation*}
So the application of the methods of moments to equation~\eqref{eq:degenerate:lemma_moments} concludes the proof of this lemma.
\end{proof}

\section{Proofs for Section~\ref{sec:degenerate:other_asymptotic_frameworks}}
\label{app:degenerate:proof_other_asymptotic_frameworks}

\subsection{Proof of Theorem~\ref{th:degenerate:principal_part_general}}
\label{app:degenerate:proof_principal_part_general}

In order to prove Theorem~\ref{th:degenerate:principal_part_general}, define $\mathcal{S} = \{(r_\ell,c_\ell) : 1 \le \ell \le L \}$ the set of principal degrees of $h$. We may define $\mathcal{S}_0$ the set of pairs $(0,0) < (r_0,c_0) \le (p,q)$ such that $\gamma(N)^{-1} = o(m_N^{-r_0} n_N^{-c_0})$, for any $(r,c) \in \mathcal{S}$. We may also define $\mathcal{S}_+$, the set of pairs $(0,0) < (r_+,c_+) \le (p,q)$ such that $m_N^{-r_+} n_N^{-c_+} = o(\gamma(N)^{-1})$, for any $(r,c) \in \mathcal{S}$. We need the following lemma.

\begin{lemma}
    For all $(r,c) \in \mathcal{S}_0$, for all graphs $G$ such that $(v_1(G), v_2(G)) = (r,c)$, we have $p^G = 0$.
    \label{lem:degenerate:null_projections}
\end{lemma}

\begin{proof}
    We have
\begin{equation*}
\begin{split}
        \mathbb{V}[U_{N}] &= \sum_{(0,0) < (r,c) \le (p,q)} \frac{(m_N-r)!}{m_N!} \frac{(n_N-c)!}{n_N!} V^{(r,c)} \\
        &= \sum_{(r,c) \in \mathcal{S}_0} \frac{(m_N-r)!}{m_N!} \frac{(n_N-c)!}{n_N!} V^{(r,c)} + \sum_{(r,c) \in \mathcal{S}} \frac{(m_N-r)!}{m_N!} \frac{(n_N-c)!}{n_N!} V^{(r,c)} \\
        &\quad + \sum_{(r,c) \in \mathcal{S}_+} \frac{(m_N-r)!}{m_N!} \frac{(n_N-c)!}{n_N!} V^{(r,c)}.
\end{split}
\end{equation*}

By definition, $(r,c) \in \mathcal{S}_+$, $m_N^{-r} n_N^{-c} = o(\gamma(N)^{-1})$ and $$\mathbb{V}[U_N] = \gamma(N)^{-1} \sum_{1 \le \ell \le L} \alpha_\ell V^{(r_\ell, c_\ell)} + o(\gamma(N)^{-1}).$$ 
Therefore, 
\begin{equation*}
\begin{split}
        \sum_{(r,c) \in \mathcal{S}_0} \frac{(m_N-r)!}{m_N!} \frac{(n_N-c)!}{n_N!} V^{(r,c)} &= \sum_{l = 1}^{L} \left(\frac{\alpha_\ell}{\gamma(N)} - \frac{(m_N-r_\ell)!}{m_N!} \frac{(n_N-c_\ell)!}{n_N!}\right) V^{(r_\ell,c_\ell)} + o(\gamma(N)^{-1}).
\end{split}
\end{equation*}

Again, by definition, we have for all $1 \le \ell \le L$, $\gamma(N) \frac{(m_N-r_\ell)!}{m_N!} \frac{(n_N-c_\ell)!}{n_N!} \xrightarrow[N \rightarrow \infty]{} \alpha_\ell$. Therefore, the previous equation yields
\begin{equation*}
\begin{split}
        \gamma(N) \sum_{(r,c) \in \mathcal{S}_0} \frac{(m_N-r)!}{m_N!} \frac{(n_N-c)!}{n_N!} V^{(r,c)} = o(1).
\end{split}
\end{equation*}

But for all $(r,c) \in \mathcal{S}_0$, $\gamma(N) \frac{(m_N-r)!}{m_N!} \frac{(n_N-c)!}{n_N!} \xrightarrow[N \rightarrow \infty]{} \infty$. Since $V^{(r,c)} \ge 0$ for all $(0,0) \le (r,c) \le (p,q)$, this means that for all $(r,c) \in \mathcal{S}_0$, we have $V^{(r,c)} = 0$. Thus, $$V^{(r,c)} = \frac{p!}{(p-r)!} \frac{q!}{(q-r)!} \sum_{G \in \Gamma_{r,c}} \lvert \Aut(G) \rvert^{-1} \mathbb{V}[p^{G}],$$ 
this means $\mathbb{V}[p^{G}] = 0$ for all $G \in \Gamma_{r,c}$.

Finally, let $G$ be any graph such that $(v_1(G), v_2(G)) = (r,c)$. Then there exists a graph $G^* \in \Gamma_{r,c}$ such that $\mathbb{V}[p^G] = \mathbb{V}[p^{G^*}]$. We have already shown that $\mathbb{V}[p^{G^*}]=0$ for all $(r,c) \in \mathcal{S}_0$, so adding the fact that $\mathbb{E}[p^G] = 0$ for all graphs $G \neq \emptyset$, it means that $p^G = 0$, for all graphs $G$ such that $(v_1(G), v_2(G)) = (r,c) \in \mathcal{S}_0$.

\end{proof}

\begin{proof}[Proof of Theorem~\ref{th:degenerate:principal_part_general}]
\begin{equation*}
    \sqrt{\gamma(N)} \left[U_N - p^\emptyset - \sum_{\ell = 1}^L P^{r_\ell,c_\ell}_{N}\right] = \sqrt{\gamma(N)} \left[ \sum_{(r,c) \in \mathcal{S}_0} P^{r,c}_{N} + \sum_{(r,c) \in \mathcal{S}_+} P^{r,c}_{N} \right].
\end{equation*}
By Lemma~\ref{lem:degenerate:null_projections}, $P^{r,c}_{N} = 0$ for all $(r,c) \in \mathcal{S}_0$. 

\begin{equation*}
\begin{split}
    \mathbb{V}\left[\sqrt{\gamma(N)} \sum_{(r,c) \in \mathcal{S}_+} P^{r,c}_{N}\right] &= \gamma(N) \sum_{(r,c) \in \mathcal{S}_+} \frac{(m-r)!}{m!} \frac{(n-c)!}{n!} V^{(r,c)} \\
    &= o(1).
\end{split}
\end{equation*}

That means $\sqrt{\gamma(N)} (U_N - p^\emptyset) = \sqrt{\gamma(N)} \sum_{\ell = 1}^L P^{r_\ell,c_\ell}_{N} + o_P(1)$, which concludes the proof.
\end{proof}

\subsection{Proof of Theorem~\ref{th:degenerate:gaussian_theorem_general}}
\label{app:degenerate:proof_gaussian_theorem_general}

\begin{proof}
Theorem~\ref{th:degenerate:principal_part_general} states that $\sqrt{\gamma(N)} (U_N - p^\emptyset)$ has the same limit as $\sqrt{\gamma(N)} \sum_{\ell = 1}^L P^{r_\ell,c_\ell}_{N}$.
    
For all $(0,0) < (r,c) \le (p,q)$,
\begin{equation*}
     P^{r,c}_{N} = \sum_{G \in \Gamma_{r,c}} \frac{1}{(p-r)! (q-c)! \lvert \Aut(G) \rvert} \widetilde{P}^{G}_{N}.
\end{equation*}

So 
\begin{equation*}
     \sqrt{\gamma(N)} \sum_{\ell = 1}^L P^{r_\ell,c_\ell}_{N} = \sum_{\ell = 1}^L \sqrt{\gamma(N)} m_N^{-r_\ell/2} n_N^{-c_\ell/2} \sum_{G \in \Gamma_{r_\ell,c_\ell}}  \frac{m_N^{r_\ell/2} n_N^{c_\ell/2} \widetilde{P}^{G}_{N}}{(p-r_\ell)! (q-c_\ell)! \lvert \Aut(G) \rvert}.
\end{equation*}

By definition, $\gamma(N) m_N^{-r_\ell} n_N^{-c_\ell} \xrightarrow[N \rightarrow \infty]{} \alpha_\ell$. Therefore, by Lemma~\ref{lem:degenerate:convergence_gaussian_projection}, $\sqrt{\gamma(N)} \sum_{\ell = 1}^L P^{r_\ell,c_\ell}_{N}$ converges in distribution to $Z = \sum_{\ell = 1}^L \sqrt{\alpha_\ell} \sum_{G \in \Gamma_{r_\ell,c_\ell}} W_G$, where all $W_G$ are independent Gaussian variables with mean $0$ and variance $\frac{(p!)^2 (q!)^2}{((p-r_\ell)!)^2 ((q-c_\ell)!)^2 \lvert \Aut(G) \rvert} \mathbb{V}[p^{G}]$.

Finally, it follows that $Z$ is a Gaussian variable with mean $0$ and variance $\sum_{\ell = 1}^L \sqrt{\alpha_\ell} V^{(r_\ell, c_\ell)}$ where 
\begin{equation*}
    V^{(r_\ell, c_\ell)} = \sum_{G \in \Gamma_{r_\ell, c_\ell}} \frac{(p!)^2 (q!)^2}{((p-r_\ell)!)^2 ((q-c_\ell)!)^2 \lvert \Aut(G) \rvert} \mathbb{V}[p^{G}]
\end{equation*}

\end{proof}

\section{Derivation of the variances of Running example~\ref{ex:2}}
\label{app:degenerate:proof_application_ex2}

In this section, we calculate the conditional expectations and the variances of Running example~\ref{ex:2}, investigated in Sections~\ref{sec:degenerate:principal_part} and~\ref{sec:degenerate:gaussian_case}. Let the distribution of $Y$ be defined by 
\begin{align*}
        &\xi_i \overset{i.i.d.}{\sim} \mathcal{U}[0,1], &\forall 1 \le i \le m, \\
        &\eta_j \overset{i.i.d.}{\sim} \mathcal{U}[0,1], &\forall 1 \le j \le n, \\
        &Y_{ij}~\mid~\xi_i, \eta_j \sim \mathcal{P}(\lambda f(\xi_i) g(\eta_j)), &\forall 1 \le i \le m, 1 \le j \le n.
\end{align*}
Let $U_N$ be the $U$-statistic with kernel $h_3 = h_1 - h_2$ where
\begin{equation*}
    h_1(Y_{\{i_1,i_2\},\{j_1,j_2\}}) = \frac{1}{2}(Y_{i_1j_1}Y_{i_1j_2} + Y_{i_2j_1}Y_{i_2j_2}),
\end{equation*}
and
\begin{equation*}
    h_2(Y_{\{i_1,i_2\},\{j_1,j_2\}}) = \frac{1}{2}(Y_{i_1j_1}Y_{i_2j_2} + Y_{i_2j_1}Y_{i_1j_2}).
\end{equation*}

\begin{lemma}
    We have $\mathbb{E}[h_3(Y_{\{1,2\},\{1,2\}}) \mid \xi_1, \xi_2] = \frac{\lambda^2}{2}  (f(\xi_1) - f(\xi_2))^2$.
    \label{lem:degenerate:condexp_1}
\end{lemma}

\begin{proof}
    We have
    \begin{equation*}
        \begin{split}
            \mathbb{E}[h_1(Y_{\{1,2\},\{1,2\}}) \mid \xi_1, \xi_2] &= \frac{1}{2} \mathbb{E}[Y_{11} Y_{12} + Y_{21} Y_{22} \mid \xi_1, \xi_2] \\
            &= \frac{1}{2} \mathbb{E}[\mathbb{E}[Y_{11} Y_{12} + Y_{21} Y_{22} \mid \boldsymbol{\xi}, \boldsymbol{\eta}] \mid \xi_1, \xi_2] \\
            &= \frac{1}{2} \mathbb{E}[\lambda^2 f(\xi_1)^2 g(\eta_1) g(\eta_2) + \lambda^2 f(\xi_2)^2 g(\eta_1) g(\eta_2) \mid \xi_1, \xi_2] \\
            &= \frac{\lambda^2}{2} (f(\xi_1)^2 + f(\xi_2)^2),
        \end{split}
    \end{equation*}
    and
    \begin{equation*}
        \begin{split}
            \mathbb{E}[h_2(Y_{\{1,2\},\{1,2\}}) \mid \xi_1, \xi_2] &= \frac{1}{2} \mathbb{E}[Y_{11} Y_{22} + Y_{12} Y_{21} \mid \xi_1, \xi_2] \\
            &= \frac{1}{2} \mathbb{E}[\mathbb{E}[Y_{11} Y_{22} + Y_{12} Y_{21} \mid \boldsymbol{\xi}, \boldsymbol{\eta}] \mid \xi_1, \xi_2] \\
            &= \frac{1}{2} \mathbb{E}[ 2 \lambda^2 f(\xi_1) f(\xi_2) g(\eta_1) g(\eta_2) \mid \xi_1, \xi_2] \\
            &= \lambda^2 f(\xi_1) f(\xi_2).
        \end{split}
    \end{equation*}
    This proves the result.
\end{proof}

\begin{lemma}
    We have $\mathbb{E}[h_3(Y_{\{1,2\},\{1,2\}}) \mid \eta_1, \eta_2] = \lambda^2 (F_2 - 1) g(\eta_1) g(\eta_2)$.
\end{lemma}

\begin{proof}
    We have
    \begin{equation*}
        \begin{split}
            \mathbb{E}[h_1(Y_{\{1,2\},\{1,2\}}) \mid \eta_1, \eta_2] &= \frac{1}{2} \mathbb{E}[Y_{11} Y_{12} + Y_{21} Y_{22} \mid \eta_1, \eta_2] \\
            &= \frac{1}{2} \mathbb{E}[\mathbb{E}[Y_{11} Y_{12} + Y_{21} Y_{22} \mid \boldsymbol{\xi}, \boldsymbol{\eta}] \mid \eta_1, \eta_2] \\
            &= \frac{1}{2} \mathbb{E}[\lambda^2 f(\xi_1)^2 g(\eta_1) g(\eta_2) + \lambda^2 f(\xi_2)^2 g(\eta_1) g(\eta_2) \mid \eta_1, \eta_2] \\
            &= \lambda^2 F_2 g(\eta_1) g(\eta_2),
        \end{split}
    \end{equation*}
    and
    \begin{equation*}
        \begin{split}
            \mathbb{E}[h_2(Y_{\{1,2\},\{1,2\}}) \mid \eta_1, \eta_2] &= \frac{1}{2} \mathbb{E}[Y_{11} Y_{22} + Y_{12} Y_{21} \mid \eta_1, \eta_2] \\
            &= \frac{1}{2} \mathbb{E}[\mathbb{E}[Y_{11} Y_{22} + Y_{12} Y_{21} \mid \boldsymbol{\xi}, \boldsymbol{\eta}] \mid \eta_1, \eta_2] \\
            &= \frac{1}{2} \mathbb{E}[ 2 \lambda^2 f(\xi_1) f(\xi_2) g(\eta_1) g(\eta_2) \mid \eta_1, \eta_2] \\
            &= \lambda^2 g(\eta_1) g(\eta_2).
        \end{split}
    \end{equation*}
    This proves the result.
\end{proof}

\begin{lemma}
    We have $\mathbb{E}[h_3(Y_{\{1,2\},\{1,2\}}) \mid \xi_1, \eta_1] = \frac{\lambda^2}{2}  (f(\xi_1)^2 - 2 f(\xi_1) + F_2) g(\eta_1)$.
\end{lemma}

\begin{proof}
    We have
    \begin{equation*}
        \begin{split}
            \mathbb{E}[h_1(Y_{\{1,2\},\{1,2\}}) \mid \xi_1, \eta_1] &= \frac{1}{2} \mathbb{E}[Y_{11} Y_{12} + Y_{21} Y_{22} \mid \xi_1, \eta_1] \\
            &= \frac{1}{2} \mathbb{E}[\mathbb{E}[Y_{11} Y_{12} + Y_{21} Y_{22} \mid \boldsymbol{\xi}, \boldsymbol{\eta}] \mid \xi_1, \eta_1] \\
            &= \frac{1}{2} \mathbb{E}[\lambda^2 f(\xi_1)^2 g(\eta_1) g(\eta_2) + \lambda^2 f(\xi_2)^2 g(\eta_1) g(\eta_2) \mid \xi_1, \eta_1] \\
            &= \frac{\lambda^2}{2} (f(\xi_1)^2 + F_2) g(\eta_1),
        \end{split}
    \end{equation*}
    and
    \begin{equation*}
        \begin{split}
            \mathbb{E}[h_2(Y_{\{1,2\},\{1,2\}}) \mid \xi_1, \eta_1] &= \frac{1}{2} \mathbb{E}[Y_{11} Y_{22} + Y_{12} Y_{21} \mid \xi_1, \eta_1] \\
            &= \frac{1}{2} \mathbb{E}[\mathbb{E}[Y_{11} Y_{22} + Y_{12} Y_{21} \mid \boldsymbol{\xi}, \boldsymbol{\eta}] \mid \xi_1, \eta_1] \\
            &= \frac{1}{2} \mathbb{E}[ 2 \lambda^2 f(\xi_1) f(\xi_2) g(\eta_1) g(\eta_2) \mid \xi_1, \eta_1] \\
            &= \lambda^2 f(\xi_1) g(\eta_1).
        \end{split}
    \end{equation*}
    This proves the result.
\end{proof}

\begin{lemma}
    We have $\mathbb{E}[h_3(Y_{\{1,2\},\{1,2\}}) \mid \xi_1, \eta_1, \zeta_{11}] = \frac{\lambda}{2} (f(\xi_1) - 1) Y_{11} + \frac{\lambda^2}{2} (F_2 - f(\xi_1)) g(\eta_1)$.
    \label{lem:degenerate:condexp_2}
\end{lemma}

\begin{proof}
    We have
    \begin{equation*}
        \begin{split}
            \mathbb{E}[h_1(Y_{\{1,2\},\{1,2\}}) \mid \xi_1, \eta_1, \zeta_{11}] &= \frac{1}{2} \mathbb{E}[Y_{11} Y_{12} + Y_{21} Y_{22} \mid \xi_1, \eta_1, \zeta_{11}] \\
            &= \frac{1}{2} \mathbb{E}[\mathbb{E}[Y_{11} Y_{12} + Y_{21} Y_{22} \mid \boldsymbol{\xi}, \boldsymbol{\eta}, Y_{11}] \mid \xi_1, \eta_1, \zeta_{11}] \\
            &= \frac{1}{2} \mathbb{E}[\lambda f(\xi_1) g(\eta_2) Y_{11} + \lambda^2 f(\xi_2)^2 g(\eta_1) g(\eta_2) \mid \xi_1, \eta_1, \zeta_{11}] \\
            &= \frac{\lambda}{2} f(\xi_1) Y_{11}  + \frac{\lambda^2}{2}  F_2 g(\eta_1),
        \end{split}
    \end{equation*}
    and
    \begin{equation*}
        \begin{split}
            \mathbb{E}[h_2(Y_{\{1,2\},\{1,2\}}) \mid \xi_1, \xi_2, \zeta_{11}] &= \frac{1}{2} \mathbb{E}[Y_{11} Y_{22} + Y_{12} Y_{21} \mid \xi_1, \eta_1, \zeta_{11}] \\
            &= \frac{1}{2} \mathbb{E}[\mathbb{E}[Y_{11} Y_{22} + Y_{12} Y_{21} \mid \boldsymbol{\xi}, \boldsymbol{\eta}, Y_{11}] \mid \xi_1, \eta_1, \zeta_{11}] \\
            &= \frac{1}{2} \mathbb{E}[ \lambda Y_{11} f(\xi_2) g(\eta_2) + \lambda^2 f(\xi_1) f(\xi_2) g(\eta_1) g(\eta_2) \mid \xi_1, \eta_1, \zeta_{11}] \\
            &= \frac{\lambda}{2} Y_{11} + \frac{\lambda^2}{2} f(\xi_1) g(\eta_1).
        \end{split}
    \end{equation*}
    This proves the result.
\end{proof}

\begin{lemma}
    We have $\mathbb{E}[h_3(Y_{\{1,2\},\{1,2\}}) \mid \xi_1, \xi_2, \eta_1] = \frac{\lambda^2}{2} (f(\xi_1) - f(\xi_2))^2 g(\eta_1)$.
    \label{lem:degenerate:condexp_3}
\end{lemma}

\begin{proof}
    We have
    \begin{equation*}
        \begin{split}
            \mathbb{E}[h_1(Y_{\{1,2\},\{1,2\}}) \mid \xi_1, \xi_2, \eta_1] &= \frac{1}{2} \mathbb{E}[Y_{11} Y_{12} + Y_{21} Y_{22} \mid \xi_1, \xi_2, \eta_1] \\
            &= \frac{1}{2} \mathbb{E}[\mathbb{E}[Y_{11} Y_{12} + Y_{21} Y_{22} \mid \boldsymbol{\xi}, \boldsymbol{\eta}] \mid \xi_1, \xi_2, \eta_1] \\
            &= \frac{1}{2} \mathbb{E}[\lambda^2 f(\xi_1)^2 g(\eta_1) g(\eta_2) + \lambda^2 f(\xi_2)^2 g(\eta_1) g(\eta_2) \mid \xi_1, \xi_2, \eta_1] \\
            &= \frac{\lambda^2}{2} (f(\xi_1)^2 + f(\xi_2)^2) g(\eta_1),
        \end{split}
    \end{equation*}
    and
    \begin{equation*}
        \begin{split}
            \mathbb{E}[h_2(Y_{\{1,2\},\{1,2\}}) \mid \xi_1, \xi_2, \eta_1] &= \frac{1}{2} \mathbb{E}[Y_{11} Y_{22} + Y_{12} Y_{21} \mid \xi_1, \xi_2, \eta_1] \\
            &= \frac{1}{2} \mathbb{E}[\mathbb{E}[Y_{11} Y_{22} + Y_{12} Y_{21} \mid \boldsymbol{\xi}, \boldsymbol{\eta}] \mid \xi_1, \xi_2, \eta_1] \\
            &= \frac{1}{2} \mathbb{E}[ 2 \lambda^2 f(\xi_1) f(\xi_2) g(\eta_1) g(\eta_2) \mid \xi_1, \xi_2, \eta_1] \\
            &= \lambda^2 f(\xi_1) f(\xi_2) g(\eta_1).
        \end{split}
    \end{equation*}
    This proves the result.
\end{proof}

\begin{lemma}
    We have $\mathbb{E}[h_3(Y_{\{1,2\},\{1,2\}}) \mid \xi_1, \xi_2, \eta_1, \zeta_{11}] = \frac{\lambda}{2} (f(\xi_1) - f(\xi_2)) Y_{11} + \frac{\lambda^2}{2} (f(\xi_2) - f(\xi_1)) f(\xi_2) g(\eta_1)$.
\end{lemma}

\begin{proof}
    We have
    \begin{equation*}
        \begin{split}
            \mathbb{E}[h_1(Y_{\{1,2\},\{1,2\}}) \mid \xi_1, \xi_2, \eta_1, \zeta_{11}] &= \frac{1}{2} \mathbb{E}[Y_{11} Y_{12} + Y_{21} Y_{22} \mid \xi_1, \xi_2, \eta_1, \zeta_{11}] \\
            &= \frac{1}{2} \mathbb{E}[\mathbb{E}[Y_{11} Y_{12} + Y_{21} Y_{22} \mid \boldsymbol{\xi}, \boldsymbol{\eta}, Y_{11}] \mid \xi_1, \xi_2, \eta_1, \zeta_{11}] \\
            &= \frac{1}{2} \mathbb{E}[\lambda f(\xi_1) g(\eta_2) Y_{11} + \lambda^2 f(\xi_2)^2 g(\eta_1) g(\eta_2) \mid \xi_1, \xi_2, \eta_1, \zeta_{11}] \\
            &= \frac{\lambda}{2} f(\xi_1) Y_{11} + \frac{\lambda^2}{2} f(\xi_2)^2 g(\eta_1),
        \end{split}
    \end{equation*}
    and
    \begin{equation*}
        \begin{split}
            \mathbb{E}[h_2(Y_{\{1,2\},\{1,2\}}) \mid \xi_1, \xi_2, \eta_1, \zeta_{11}] &= \frac{1}{2} \mathbb{E}[Y_{11} Y_{22} + Y_{12} Y_{21} \mid \xi_1, \xi_2, \eta_1, \zeta_{11}] \\
            &= \frac{1}{2} \mathbb{E}[\mathbb{E}[Y_{11} Y_{22} + Y_{12} Y_{21} \mid \boldsymbol{\xi}, \boldsymbol{\eta}, Y_{11}] \mid \xi_1, \xi_2, \eta_1, \zeta_{11}] \\
            &= \frac{1}{2} \mathbb{E}[ \lambda f(\xi_2) g(\eta_2) Y_{11} + \lambda^2 f(\xi_1) f(\xi_2) g(\eta_1) g(\eta_2) \mid \xi_1, \xi_2, \eta_1, \zeta_{11}] \\
            &= \frac{\lambda}{2} f(\xi_2) Y_{11} + \frac{\lambda^2}{2} f(\xi_1) f(\xi_2) g(\eta_1).
        \end{split}
    \end{equation*}
    This proves the result.
\end{proof}

\begin{lemma}
    We have $\mathbb{E}[h_3(Y_{\{1,2\},\{1,2\}}) \mid \xi_1, \xi_2, \eta_1, \zeta_{11}, \zeta_{21}] = \frac{\lambda}{2} (f(\xi_1) - f(\xi_2)) (Y_{11} - Y_{21})$.
\end{lemma}

\begin{proof}
    We have
    \begin{equation*}
        \begin{split}
            \mathbb{E}[h_1(Y_{\{1,2\},\{1,2\}}) \mid \xi_1, \xi_2, \eta_1, \zeta_{11}, \zeta_{21}] &= \frac{1}{2} \mathbb{E}[Y_{11} Y_{12} + Y_{21} Y_{22} \mid \xi_1, \xi_2, \eta_1, \zeta_{11}, \zeta_{21}] \\
            &= \frac{1}{2} \mathbb{E}[\mathbb{E}[Y_{11} Y_{12} + Y_{21} Y_{22} \mid \boldsymbol{\xi}, \boldsymbol{\eta}, Y_{11}, Y_{21}] \mid \xi_1, \xi_2, \eta_1, \zeta_{11}, \zeta_{21}] \\
            &= \frac{1}{2} \mathbb{E}[\lambda f(\xi_1) g(\eta_2) Y_{11} + \lambda^2 f(\xi_2) g(\eta_2) Y_{21} \mid \xi_1, \xi_2, \eta_1, \zeta_{11}, \zeta_{21}] \\
            &= \frac{\lambda}{2} (f(\xi_1) Y_{11} +  f(\xi_2) Y_{21}),
        \end{split}
    \end{equation*}
    and
    \begin{equation*}
        \begin{split}
            \mathbb{E}[h_2(Y_{\{1,2\},\{1,2\}}) \mid \xi_1, \xi_2, \eta_1, \zeta_{11}, \zeta_{21}] &= \frac{1}{2} \mathbb{E}[Y_{11} Y_{22} + Y_{12} Y_{21} \mid \xi_1, \xi_2, \eta_1, \zeta_{11}, \zeta_{21}] \\
            &= \frac{1}{2} \mathbb{E}[\mathbb{E}[Y_{11} Y_{22} + Y_{12} Y_{21} \mid \boldsymbol{\xi}, \boldsymbol{\eta}, Y_{11}, Y_{21}] \mid \xi_1, \xi_2, \eta_1, \zeta_{11}, \zeta_{21}] \\
            &= \frac{1}{2} \mathbb{E}[ \lambda f(\xi_2) g(\eta_2) Y_{11} + \lambda f(\xi_1) g(\eta_2) Y_{21} \mid \xi_1, \xi_2, \eta_1, \zeta_{11}, \zeta_{21}] \\
            &= \frac{\lambda}{2} (f(\xi_2)  Y_{11} + f(\xi_1) Y_{21}).
        \end{split}
    \end{equation*}
    This proves the result.
\end{proof}

\begin{lemma}
    We have $\mathbb{E}[h_3(Y_{\{1,2\},\{1,2\}}) \mid \xi_1, \eta_1, \eta_2] = \frac{\lambda^2}{2}  (f(\xi_1)^2 - 2 f(\xi_1) + F_2) g(\eta_1) g(\eta_2)$.
\end{lemma}

\begin{proof}
    We have
    \begin{equation*}
        \begin{split}
            \mathbb{E}[h_1(Y_{\{1,2\},\{1,2\}}) \mid \xi_1, \eta_1, \eta_2] &= \frac{1}{2} \mathbb{E}[Y_{11} Y_{12} + Y_{21} Y_{22} \mid \xi_1, \eta_1, \eta_2] \\
            &= \frac{1}{2} \mathbb{E}[\mathbb{E}[Y_{11} Y_{12} + Y_{21} Y_{22} \mid \boldsymbol{\xi}, \boldsymbol{\eta}] \mid \xi_1, \eta_1, \eta_2] \\
            &= \frac{1}{2} \mathbb{E}[\lambda^2 f(\xi_1)^2 g(\eta_1) g(\eta_2) + \lambda^2 f(\xi_2)^2 g(\eta_1) g(\eta_2) \mid \xi_1, \eta_1, \eta_2] \\
            &= \frac{\lambda^2}{2} (f(\xi_1)^2 + F_2) g(\eta_1) g(\eta_2),
        \end{split}
    \end{equation*}
    and
    \begin{equation*}
        \begin{split}
            \mathbb{E}[h_2(Y_{\{1,2\},\{1,2\}}) \mid \xi_1, \eta_1, \eta_2] &= \frac{1}{2} \mathbb{E}[Y_{11} Y_{22} + Y_{12} Y_{21} \mid \xi_1, \eta_1, \eta_2] \\
            &= \frac{1}{2} \mathbb{E}[\mathbb{E}[Y_{11} Y_{22} + Y_{12} Y_{21} \mid \boldsymbol{\xi}, \boldsymbol{\eta}] \mid \xi_1, \eta_1, \eta_2] \\
            &= \frac{1}{2} \mathbb{E}[ 2 \lambda^2 f(\xi_1) f(\xi_2) g(\eta_1) g(\eta_2) \mid \xi_1, \eta_1, \eta_2] \\
            &= \lambda^2 f(\xi_1) g(\eta_1) g(\eta_2).
        \end{split}
    \end{equation*}
    This proves the result.
\end{proof}

\begin{lemma}
    We have $\mathbb{E}[h_3(Y_{\{1,2\},\{1,2\}}) \mid \xi_1, \eta_1, \eta_2, \zeta_{11}] = \frac{\lambda}{2} (f(\xi_1) - 1) g(\eta_2) Y_{11} + \frac{\lambda^2}{2} (F_2 - f(\xi_1)) g(\eta_1) g(\eta_2)$.
\end{lemma}

\begin{proof}
    We have
    \begin{equation*}
        \begin{split}
            \mathbb{E}[h_1(Y_{\{1,2\},\{1,2\}}) \mid \xi_1, \eta_1, \eta_2, \zeta_{11}] &= \frac{1}{2} \mathbb{E}[Y_{11} Y_{12} + Y_{21} Y_{22} \mid \xi_1, \eta_1, \eta_2, \zeta_{11}] \\
            &= \frac{1}{2} \mathbb{E}[\mathbb{E}[Y_{11} Y_{12} + Y_{21} Y_{22} \mid \boldsymbol{\xi}, \boldsymbol{\eta}, Y_{11}] \mid \xi_1, \eta_1, \eta_2, \zeta_{11}] \\
            &= \frac{1}{2} \mathbb{E}[\lambda f(\xi_1) g(\eta_2) Y_{11} + \lambda^2 f(\xi_2)^2 g(\eta_1) g(\eta_2) \mid \xi_1, \eta_1, \eta_2, \zeta_{11}] \\
            &= \frac{\lambda}{2} f(\xi_1) g(\eta_2) Y_{11} + \frac{\lambda^2}{2} F_2 g(\eta_1) g(\eta_2),
        \end{split}
    \end{equation*}
    and
    \begin{equation*}
        \begin{split}
            \mathbb{E}[h_2(Y_{\{1,2\},\{1,2\}}) \mid \xi_1, \eta_1, \eta_2, \zeta_{11}] &= \frac{1}{2} \mathbb{E}[Y_{11} Y_{22} + Y_{12} Y_{21} \mid \xi_1, \eta_1, \eta_2, \zeta_{11}] \\
            &= \frac{1}{2} \mathbb{E}[\mathbb{E}[Y_{11} Y_{22} + Y_{12} Y_{21} \mid \boldsymbol{\xi}, \boldsymbol{\eta}, Y_{11}] \mid \xi_1, \eta_1, \eta_2, \zeta_{11}] \\
            &= \frac{1}{2} \mathbb{E}[ \lambda f(\xi_2) g(\eta_2) Y_{11} + \lambda^2 f(\xi_1) f(\xi_2) g(\eta_1) g(\eta_2) \mid \xi_1, \eta_1, \eta_2, \zeta_{11}] \\
            &= \frac{\lambda}{2} g(\eta_2) Y_{11} + \frac{\lambda^2}{2} f(\xi_1) g(\eta_1) g(\eta_2).
        \end{split}
    \end{equation*}
    This proves the result.
\end{proof}

\begin{lemma}
    We have $\mathbb{E}[h_3(Y_{\{1,2\},\{1,2\}}) \mid \xi_1, \eta_1, \eta_2, \zeta_{11}, \zeta_{12}] = \frac{1}{2} Y_{11} Y_{12} - \frac{\lambda}{2} (g(\eta_2) Y_{11} + g(\eta_1) Y_{12}) + \frac{\lambda^2}{2} F_2 g(\eta_1) g(\eta_2)$.
    \label{lem:degenerate:condexp_4}
\end{lemma}

\begin{proof}
    We have
    \begin{equation*}
        \begin{split}
            \mathbb{E}[h_1(Y_{\{1,2\},\{1,2\}}) \mid \xi_1, \eta_1, \eta_2, \zeta_{11}, \zeta_{12}] &= \frac{1}{2} \mathbb{E}[Y_{11} Y_{12} + Y_{21} Y_{22} \mid \xi_1, \eta_1, \eta_2, \zeta_{11}, \zeta_{12}] \\
            &= \frac{1}{2} \mathbb{E}[\mathbb{E}[Y_{11} Y_{12} + Y_{21} Y_{22} \mid \boldsymbol{\xi}, \boldsymbol{\eta}, Y_{11}, Y_{12}] \mid \xi_1, \eta_1, \eta_2, \zeta_{11}, \zeta_{12}] \\
            &= \frac{1}{2} \mathbb{E}[Y_{11} Y_{12} + \lambda^2 f(\xi_2)^2 g(\eta_1) g(\eta_2) \mid \xi_1, \eta_1, \eta_2, \zeta_{11}, \zeta_{12}] \\
            &= \frac{1}{2} Y_{11} Y_{12} + \frac{\lambda^2}{2} F_2 g(\eta_1) g(\eta_2),
        \end{split}
    \end{equation*}
    and
    \begin{equation*}
        \begin{split}
            \mathbb{E}[h_2(Y_{\{1,2\},\{1,2\}}) \mid \xi_1, \eta_1, \eta_2, \zeta_{11}, \zeta_{12}] &= \frac{1}{2} \mathbb{E}[Y_{11} Y_{22} + Y_{12} Y_{21} \mid \xi_1, \eta_1, \eta_2, \zeta_{11}, \zeta_{12}] \\
            &= \frac{1}{2} \mathbb{E}[\mathbb{E}[Y_{11} Y_{22} + Y_{12} Y_{21} \mid \boldsymbol{\xi}, \boldsymbol{\eta}, Y_{11}, Y_{12}] \mid \xi_1, \eta_1, \eta_2, \zeta_{11}, \zeta_{12}] \\
            &= \frac{1}{2} \mathbb{E}[ \lambda f(\xi_2) g(\eta_2) Y_{11} + \lambda f(\xi_2) g(\eta_1) Y_{12} \mid \xi_1, \eta_1, \eta_2, \zeta_{11}, \zeta_{12}] \\
            &= \frac{\lambda}{2} (g(\eta_2) Y_{11} + g(\eta_1) Y_{12}).
        \end{split}
    \end{equation*}
    This proves the result.
\end{proof}

\begin{lemma}
    We have $\mathbb{E}[\mathbb{E}[h_3(Y_{\{1,2\},\{1,2\}}) \mid \xi_1, \eta_1, \eta_2, \zeta_{11}, \zeta_{12}]^2] = \frac{\lambda^2}{4} F_2 + \frac{\lambda^3}{2} (F_3 - 2 F_2 + 1) G_2 + \frac{\lambda^4}{4} (F_4 - 4 F_3 + 3 F_2^2) G_2^2$.
    \label{lem:degenerate:varexp}
\end{lemma}

\begin{proof}
    We have
    \begin{equation*}
        \begin{split}
            \mathbb{E}[h_3(Y_{\{1,2\},\{1,2\}})& \mid \xi_1, \eta_1, \eta_2, \zeta_{11}, \zeta_{12}]^2 \\
            &= \left(\frac{1}{2} Y_{11} Y_{12} - \frac{\lambda}{2} (g(\eta_2) Y_{11} + g(\eta_1) Y_{12}) + \frac{\lambda^2}{2} F_2 g(\eta_1) g(\eta_2)\right)^2 \\
            &= \frac{1}{4} Y_{11}^2 Y_{12}^2 + \frac{\lambda^2}{4} g(\eta_2)^2 Y_{11}^2 + \frac{\lambda^2}{4} g(\eta_1)^2 Y_{12}^2 + \frac{\lambda^2}{2} g(\eta_1) g(\eta_2) Y_{11} Y_{12} \\
            &\quad + \frac{\lambda^4}{4} F_2^2 g(\eta_1)^2 g(\eta_2)^2 - \frac{\lambda}{2} g(\eta_2) Y_{11}^2 Y_{12} - \frac{\lambda}{2} g(\eta_1) Y_{11} Y_{12}^2 \\
            &\quad + \frac{\lambda^2}{2} F_2 g(\eta_1) g(\eta_2) Y_{11} Y_{12} - \frac{\lambda^3}{2} F_2 g(\eta_1) g(\eta_2)^2 Y_{11} - \frac{\lambda^3}{2} F_2 g(\eta_1)^2 g(\eta_2) Y_{12}.
        \end{split}
    \end{equation*}
    Taking the expectation of this random variable and using the row-column exchangeability of $Y$, it becomes
    \begin{equation*}
        \begin{split}
            \mathbb{E}[\mathbb{E}[h_3(Y_{\{1,2\},\{1,2\}})& \mid \xi_1, \eta_1, \eta_2, \zeta_{11}, \zeta_{12}]^2] \\
            &= \frac{1}{4} \mathbb{E}[Y_{11}^2 Y_{12}^2] + \frac{\lambda^2}{2} \mathbb{E}[g(\eta_2)^2 Y_{11}^2] + \frac{\lambda^2}{2} (F_2 + 1) \mathbb{E}[ g(\eta_1) g(\eta_2) Y_{11} Y_{12} ] \\
            &\quad + \frac{\lambda^4}{4} F_2^2 \mathbb{E}[g(\eta_1)^2 g(\eta_2)^2] - \lambda \mathbb{E}[g(\eta_2) Y_{11}^2 Y_{12}] - \lambda^3 F_2 \mathbb{E}[g(\eta_1) g(\eta_2)^2 Y_{11}].
        \end{split}
    \end{equation*}
    We calculate each term of this expression separately, obtaining
    \begin{equation*}
        \begin{split}
            \frac{1}{4} \mathbb{E}[Y_{11}^2 Y_{12}^2] &= \mathbb{E}[\mathbb{E}[Y_{11}^2 Y_{12}^2 \mid \boldsymbol{\xi}, \boldsymbol{\eta}]] \\
            &= \frac{1}{4} \mathbb{E}[(\lambda f(\xi_1) g(\eta_1) + \lambda^2 f(\xi_1)^2 g(\eta_1)^2) \\
            &\quad \times (\lambda f(\xi_1) g(\eta_2) + \lambda^2 f(\xi_1)^2 g(\eta_2)^2) ] \\
            &= \frac{\lambda^2}{4} \mathbb{E}[ f(\xi_1)^2 g(\eta_1) g(\eta_2) ] + \frac{\lambda^3}{2} \mathbb{E}[ f(\xi_1)^3 g(\eta_1)^2 g(\eta_2)] \\
            &\quad + \frac{\lambda^4}{4} \mathbb{E}[f(\xi_1)^4 g(\eta_1)^2 g(\eta_2)^2 ] \\
            &= \frac{\lambda^2}{4} F_2 + \frac{\lambda^3}{2} F_3 G_2 + \frac{\lambda^4}{4} F_4 G_2^2, \\
            \frac{\lambda^2}{2} \mathbb{E}[g(\eta_2)^2 Y_{11}^2] &= \frac{\lambda^2}{2} \mathbb{E}[\mathbb{E}[g(\eta_2)^2 Y_{11}^2 \mid \boldsymbol{\xi}, \boldsymbol{\eta}]] \\
            &= \frac{\lambda^2}{2} \mathbb{E}[g(\eta_2)^2 (\lambda f(\xi_1) g(\eta_1) + \lambda^2 f(\xi_1)^2 g(\eta_1)^2)] \\
            &= \frac{\lambda^3}{2} G_2 + \frac{\lambda^4}{2} F_2 G_2^2, \\
            \frac{\lambda^2}{2} (F_2 + 1) \mathbb{E}[ g(\eta_1) g(\eta_2) Y_{11} Y_{12} ] &= \frac{\lambda^2}{2} (F_2 + 1) \mathbb{E}[\mathbb{E}[ g(\eta_1) g(\eta_2) Y_{11} Y_{12} \mid \boldsymbol{\xi}, \boldsymbol{\eta}]] \\
            &= \frac{\lambda^2}{2} (F_2 + 1) \mathbb{E}[ \lambda^2 f(\xi_1)^2 g(\eta_1)^2 g(\eta_2)^2] \\
            &= \frac{\lambda^4}{2} (F_2 + 1) F_2 G_2^2, \\
            \lambda \mathbb{E}[g(\eta_2) Y_{11}^2 Y_{12}] &= \lambda \mathbb{E}[\mathbb{E}[g(\eta_2) Y_{11}^2 Y_{12}\mid \boldsymbol{\xi}, \boldsymbol{\eta}]] \\
            &= \lambda \mathbb{E}[g(\eta_2) (\lambda f(\xi_1) g(\eta_1) + \lambda^2 f(\xi_1)^2 g(\eta_1)^2) \lambda f(\xi_1) g(\eta_2)] \\
            &= \lambda^3 F_2 G_2 + \lambda^4 F_3 G_2^2, \\
            \lambda^3 F_2 \mathbb{E}[g(\eta_1) g(\eta_2)^2 Y_{11}] &= \lambda^3 F_2 \mathbb{E}[ \mathbb{E}[g(\eta_1) g(\eta_2)^2 Y_{11} \mid \boldsymbol{\xi}, \boldsymbol{\eta}]] \\
            &= \lambda^3 F_2 \mathbb{E}[\lambda f(\xi_1) g(\eta_1)^2 g(\eta_2)^2] \\
            &= \lambda^4 F_2 G_2^2.
        \end{split}
    \end{equation*}
    Therefore, 
    \begin{equation*}
        \begin{split}
            \mathbb{E}[\mathbb{E}[h_3(Y_{\{1,2\},\{1,2\}})& \mid \xi_1, \eta_1, \eta_2, \zeta_{11}, \zeta_{12}]^2] \\
            &= \frac{\lambda^2}{4} F_2 + \frac{\lambda^3}{2} F_3 G_2 + \frac{\lambda^4}{4} F_4 G_2^2 + \frac{\lambda^3}{2} G_2 + \frac{\lambda^4}{2} F_2 G_2^2 \\
            &\quad + \frac{\lambda^4}{2} (F_2 + 1) F_2 G_2^2 + \frac{\lambda^4}{4} F_2^2 G_2^2 - \lambda^3 F_2 G_2 - \lambda^4 F_3 G_2^2 - \lambda^4 F_2 G_2^2 \\
            &= \frac{\lambda^2}{4} F_2 + \frac{\lambda^3}{2} (F_3 - 2 F_2 + 1) G_2 + \frac{\lambda^4}{4} (F_4 - 4 F_3 + 3 F_2^2) G_2^2,
        \end{split}
    \end{equation*}
    which is the expression given by the lemma.
\end{proof}

\section{Derivation of the variances of Running example~\ref{ex:3}}
\label{app:degenerate:proof_application_ex3}

In this section, we calculate the conditional expectations and the variances of Running example~\ref{ex:3}, investigated in Sections~\ref{sec:degenerate:principal_part} and~\ref{sec:degenerate:gaussian_case}. Let the distribution of $Y$ be defined by 
\begin{align*}
        &\xi_i \overset{i.i.d.}{\sim} \mathcal{U}[0,1], &\forall 1 \le i \le m, \\
        &\eta_j \overset{i.i.d.}{\sim} \mathcal{U}[0,1], &\forall 1 \le j \le n, \\
        &W_{ij} \overset{i.i.d.}{\sim} \mathcal{L}(\alpha), &\forall 1 \le i \le m, 1 \le j \le n, \\
        &Y_{ij}~\mid~\xi_i, \eta_j, W_{ij} \sim \mathcal{P}(\lambda f(\xi_i) g(\eta_j) W_{ij}), &\forall 1 \le i \le m, 1 \le j \le n.
\end{align*}
Let $U_N$ be the $U$-statistic with kernel $h_6 = h_4 - h_5$ where
    \begin{equation*}
        h_4(Y_{\{i_1,i_2\}, \{j_1,j_2\}}) = \frac{1}{4}\left( Y_{i_1j_1} (Y_{i_1j_1} - 1) Y_{i_2j_2} + Y_{i_1j_2} (Y_{i_1j_2} - 1) Y_{i_2j_1} + Y_{i_2j_1} (Y_{i_2j_1} - 1) Y_{i_1j_2} + Y_{i_2j_2} (Y_{i_2j_2} - 1) Y_{i_1j_1}\right)
    \end{equation*}
    and
    \begin{equation*}
        h_5(Y_{\{i_1,j_2\}, \{i_1,j_2\}}) = \frac{1}{4}\left( Y_{i_1j_1} Y_{i_1j_2} Y_{i_2j_2} + Y_{i_1j_2} Y_{i_2j_2} Y_{i_2j_1} + Y_{i_2j_2} Y_{i_2j_1} Y_{i_1j_1} + Y_{i_2j_1} Y_{i_1j_1} Y_{i_1j_2} \right).
    \end{equation*}

\begin{lemma}
    We have $\mathbb{E}[h_6(Y_{\{1,2\},\{1,2\}})] = \lambda^3 F_2 G_2 \alpha$.
    \label{lem:degenerate:exp}
\end{lemma}

\begin{proof}
    We have 
    \begin{equation*}
        \begin{split}
            \mathbb{E}&[h_4(Y_{\{1,2\},\{1,2\}})] \\
            &= \frac{1}{4} \mathbb{E}[Y_{11} (Y_{11} - 1) Y_{22} + Y_{12} (Y_{12} - 1) Y_{21} + Y_{21} (Y_{21} - 1) Y_{12} + Y_{22} (Y_{22} - 1) Y_{11}] \\
            &= \frac{1}{4} \mathbb{E}[\mathbb{E}[Y_{11} (Y_{11} - 1) Y_{22} + Y_{12} (Y_{12} - 1) Y_{21} + Y_{21} (Y_{21} - 1) Y_{12} + Y_{22} (Y_{22} - 1) Y_{11} \mid \boldsymbol{\xi}, \boldsymbol{\eta}, \boldsymbol{W}]] \\
            &= \frac{1}{4} \mathbb{E}[\lambda^3 f(\xi_1)^2 f(\xi_2) g(\eta_1)^2 g(\eta_2) W_{11}^2 W_{22}  + \lambda^3 f(\xi_1)^2 f(\xi_2) g(\eta_2)^2 g(\eta_1) W_{12}^2 W_{21} \\
            &\quad\quad+ \lambda^3 f(\xi_2)^2 f(\xi_1) g(\eta_1)^2 g(\eta_2) W_{21}^2 W_{12} + \lambda^3 f(\xi_2)^2 f(\xi_1) g(\eta_2)^2 g(\eta_1) W_{22}^2 W_{11}] \\
            &= \lambda^3 F_2 G_2 (\alpha + 1),
        \end{split}
    \end{equation*}
    and 
    \begin{equation*}
        \begin{split}
            \mathbb{E}&[h_5(Y_{\{1,2\},\{1,2\}})] \\
            &= \frac{1}{4} \mathbb{E}[Y_{11} Y_{12} Y_{22} + Y_{12} Y_{22} Y_{21} + Y_{22} Y_{21} Y_{11} + Y_{21} Y_{11} Y_{12}] \\
            &= \frac{1}{4} \mathbb{E}[\mathbb{E}[Y_{11} Y_{12} Y_{22} + Y_{12} Y_{22} Y_{21} + Y_{22} Y_{21} Y_{11} + Y_{21} Y_{11} Y_{12} \mid \boldsymbol{\xi}, \boldsymbol{\eta}, \boldsymbol{W}]] \\
            &= \frac{1}{4} \mathbb{E}[\lambda^3 f(\xi_1)^2 f(\xi_2) g(\eta_2)^2 g(\eta_1) W_{11} W_{12} W_{22}  + \lambda^3 f(\xi_2)^2 f(\xi_1) g(\eta_2)^2 g(\eta_1) W_{12} W_{22} W_{21} \\
            &\quad\quad+ \lambda^3 f(\xi_2)^2 f(\xi_1) g(\eta_1)^2 g(\eta_2) W_{22} W_{21} W_{11} + \lambda^3 f(\xi_1)^2 f(\xi_2) g(\eta_1)^2 g(\eta_2) W_{21} W_{11} W_{12}] \\
            &= \lambda^3 F_2 G_2.
        \end{split}
    \end{equation*}
    This proves the result.
\end{proof}

\begin{lemma}
    We have $\mathbb{E}[h_6(Y_{\{1,2\},\{1,2\}}) \mid \xi_1] = \frac{\lambda^3}{2} f(\xi_1) (f(\xi_1) + F_2) G_2 \alpha$.
    \label{lem:degenerate:condexp_5}
\end{lemma}

\begin{proof}
    We have 
    \begin{equation*}
        \begin{split}
            \mathbb{E}&[h_4(Y_{\{1,2\},\{1,2\}}) \mid \xi_1] \\
            &= \frac{1}{4} \mathbb{E}[Y_{11} (Y_{11} - 1) Y_{22} + Y_{12} (Y_{12} - 1) Y_{21} + Y_{21} (Y_{21} - 1) Y_{12} + Y_{22} (Y_{22} - 1) Y_{11} \mid \xi_1] \\
            &= \frac{1}{4} \mathbb{E}[\mathbb{E}[Y_{11} (Y_{11} - 1) Y_{22} + Y_{12} (Y_{12} - 1) Y_{21} + Y_{21} (Y_{21} - 1) Y_{12} + Y_{22} (Y_{22} - 1) Y_{11} \mid \boldsymbol{\xi}, \boldsymbol{\eta}, \boldsymbol{W}] \mid \xi_1] \\
            &= \frac{1}{4} \mathbb{E}[\lambda^3 f(\xi_1)^2 f(\xi_2) g(\eta_1)^2 g(\eta_2) W_{11}^2 W_{22}  + \lambda^3 f(\xi_1)^2 f(\xi_2) g(\eta_2)^2 g(\eta_1) W_{12}^2 W_{21} \\
            &\quad\quad+ \lambda^3 f(\xi_2)^2 f(\xi_1) g(\eta_1)^2 g(\eta_2) W_{21}^2 W_{12} + \lambda^3 f(\xi_2)^2 f(\xi_1) g(\eta_2)^2 g(\eta_1) W_{22}^2 W_{11} \mid \xi_1] \\
            &= \frac{\lambda^3}{2} f(\xi_1) (f(\xi_1) + F_2) G_2 (\alpha + 1),
        \end{split}
    \end{equation*}
    and 
    \begin{equation*}
        \begin{split}
            \mathbb{E}&[h_5(Y_{\{1,2\},\{1,2\}}) \mid \xi_1] \\
            &= \frac{1}{4} \mathbb{E}[Y_{11} Y_{12} Y_{22} + Y_{12} Y_{22} Y_{21} + Y_{22} Y_{21} Y_{11} + Y_{21} Y_{11} Y_{12} \mid \xi_1] \\
            &= \frac{1}{4} \mathbb{E}[\mathbb{E}[Y_{11} Y_{12} Y_{22} + Y_{12} Y_{22} Y_{21} + Y_{22} Y_{21} Y_{11} + Y_{21} Y_{11} Y_{12} \mid \boldsymbol{\xi}, \boldsymbol{\eta}, \boldsymbol{W}] \mid \xi_1] \\
            &= \frac{1}{4} \mathbb{E}[\lambda^3 f(\xi_1)^2 f(\xi_2) g(\eta_2)^2 g(\eta_1) W_{11} W_{12} W_{22}  + \lambda^3 f(\xi_2)^2 f(\xi_1) g(\eta_2)^2 g(\eta_1) W_{12} W_{22} W_{21} \\
            &\quad\quad+ \lambda^3 f(\xi_2)^2 f(\xi_1) g(\eta_1)^2 g(\eta_2) W_{22} W_{21} W_{11} + \lambda^3 f(\xi_1)^2 f(\xi_2) g(\eta_1)^2 g(\eta_2) W_{21} W_{11} W_{12} \mid \xi_1] \\
            &= \frac{\lambda^3}{2} f(\xi_1) (f(\xi_1) + F_2) G_2.
        \end{split}
    \end{equation*}
    This proves the result.
\end{proof}

\begin{lemma}
    We have $\mathbb{E}[h_6(Y_{\{1,2\},\{1,2\}}) \mid \eta_1] = \frac{\lambda^3}{2} F_2 g(\eta_1) (g(\eta_1) + G_2) \alpha$.
    \label{lem:degenerate:condexp_6}
\end{lemma}

\begin{proof}
    We have 
    \begin{equation*}
        \begin{split}
            \mathbb{E}&[h_4(Y_{\{1,2\},\{1,2\}}) \mid \eta_1] \\
            &= \frac{1}{4} \mathbb{E}[Y_{11} (Y_{11} - 1) Y_{22} + Y_{12} (Y_{12} - 1) Y_{21} + Y_{21} (Y_{21} - 1) Y_{12} + Y_{22} (Y_{22} - 1) Y_{11} \mid \eta_1] \\
            &= \frac{1}{4} \mathbb{E}[\mathbb{E}[Y_{11} (Y_{11} - 1) Y_{22} + Y_{12} (Y_{12} - 1) Y_{21} + Y_{21} (Y_{21} - 1) Y_{12} + Y_{22} (Y_{22} - 1) Y_{11} \mid \boldsymbol{\xi}, \boldsymbol{\eta}, \boldsymbol{W}] \mid \eta_1] \\
            &= \frac{1}{4} \mathbb{E}[\lambda^3 f(\xi_1)^2 f(\xi_2) g(\eta_1)^2 g(\eta_2) W_{11}^2 W_{22}  + \lambda^3 f(\xi_1)^2 f(\xi_2) g(\eta_2)^2 g(\eta_1) W_{12}^2 W_{21} \\
            &\quad\quad+ \lambda^3 f(\xi_2)^2 f(\xi_1) g(\eta_1)^2 g(\eta_2) W_{21}^2 W_{12} + \lambda^3 f(\xi_2)^2 f(\xi_1) g(\eta_2)^2 g(\eta_1) W_{22}^2 W_{11} \mid \eta_1] \\
            &= \frac{\lambda^3}{2} F_2 g(\eta_1) (g(\eta_1) + G_2) (\alpha + 1),
        \end{split}
    \end{equation*}
    and 
    \begin{equation*}
        \begin{split}
            \mathbb{E}&[h_5(Y_{\{1,2\},\{1,2\}}) \mid \eta_1] \\
            &= \frac{1}{4} \mathbb{E}[Y_{11} Y_{12} Y_{22} + Y_{12} Y_{22} Y_{21} + Y_{22} Y_{21} Y_{11} + Y_{21} Y_{11} Y_{12} \mid \eta_1] \\
            &= \frac{1}{4} \mathbb{E}[\mathbb{E}[Y_{11} Y_{12} Y_{22} + Y_{12} Y_{22} Y_{21} + Y_{22} Y_{21} Y_{11} + Y_{21} Y_{11} Y_{12} \mid \boldsymbol{\xi}, \boldsymbol{\eta}, \boldsymbol{W}] \mid \eta_1] \\
            &= \frac{1}{4} \mathbb{E}[\lambda^3 f(\xi_1)^2 f(\xi_2) g(\eta_2)^2 g(\eta_1) W_{11} W_{12} W_{22}  + \lambda^3 f(\xi_2)^2 f(\xi_1) g(\eta_2)^2 g(\eta_1) W_{12} W_{22} W_{21} \\
            &\quad\quad+ \lambda^3 f(\xi_2)^2 f(\xi_1) g(\eta_1)^2 g(\eta_2) W_{22} W_{21} W_{11} + \lambda^3 f(\xi_1)^2 f(\xi_2) g(\eta_1)^2 g(\eta_2) W_{21} W_{11} W_{12} \mid \eta_1] \\
            &= \frac{\lambda^3}{2} F_2 g(\eta_1) (g(\eta_1) + G_2).
        \end{split}
    \end{equation*}
    This proves the result.
\end{proof}

\begin{lemma}
    We have $\mathbb{E}[h_6(Y_{\{1,2\},\{1,2\}}) \mid \xi_1, \xi_2] = \frac{\lambda^3}{2} (f(\xi_1)^2 f(\xi_2) + f(\xi_2)^2 f(\xi_1)) G_2 \alpha$.
    \label{lem:degenerate:condexp_7}
\end{lemma}

\begin{proof}
    We have 
    \begin{equation*}
        \begin{split}
            \mathbb{E}&[h_4(Y_{\{1,2\},\{1,2\}}) \mid \xi_1, \xi_2] \\
            &= \frac{1}{4} \mathbb{E}[Y_{11} (Y_{11} - 1) Y_{22} + Y_{12} (Y_{12} - 1) Y_{21} + Y_{21} (Y_{21} - 1) Y_{12} + Y_{22} (Y_{22} - 1) Y_{11} \mid \xi_1, \xi_2] \\
            &= \frac{1}{4} \mathbb{E}[\mathbb{E}[Y_{11} (Y_{11} - 1) Y_{22} + Y_{12} (Y_{12} - 1) Y_{21} + Y_{21} (Y_{21} - 1) Y_{12} + Y_{22} (Y_{22} - 1) Y_{11} \mid \boldsymbol{\xi}, \boldsymbol{\eta}, \boldsymbol{W}] \mid \xi_1, \xi_2] \\
            &= \frac{1}{4} \mathbb{E}[\lambda^3 f(\xi_1)^2 f(\xi_2) g(\eta_1)^2 g(\eta_2) W_{11}^2 W_{22}  + \lambda^3 f(\xi_1)^2 f(\xi_2) g(\eta_2)^2 g(\eta_1) W_{12}^2 W_{21} \\
            &\quad\quad+ \lambda^3 f(\xi_2)^2 f(\xi_1) g(\eta_1)^2 g(\eta_2) W_{21}^2 W_{12} + \lambda^3 f(\xi_2)^2 f(\xi_1) g(\eta_2)^2 g(\eta_1) W_{22}^2 W_{11} \mid \xi_1, \xi_2] \\
            &= \frac{\lambda^3}{2} (f(\xi_1)^2 f(\xi_2) + f(\xi_2)^2 f(\xi_1)) G_2 (\alpha + 1),
        \end{split}
    \end{equation*}
    and 
    \begin{equation*}
        \begin{split}
            \mathbb{E}&[h_5(Y_{\{1,2\},\{1,2\}}) \mid \xi_1, \xi_2] \\
            &= \frac{1}{4} \mathbb{E}[Y_{11} Y_{12} Y_{22} + Y_{12} Y_{22} Y_{21} + Y_{22} Y_{21} Y_{11} + Y_{21} Y_{11} Y_{12} \mid \xi_1, \xi_2] \\
            &= \frac{1}{4} \mathbb{E}[\mathbb{E}[Y_{11} Y_{12} Y_{22} + Y_{12} Y_{22} Y_{21} + Y_{22} Y_{21} Y_{11} + Y_{21} Y_{11} Y_{12} \mid \boldsymbol{\xi}, \boldsymbol{\eta}, \boldsymbol{W}] \mid \xi_1, \xi_2] \\
            &= \frac{1}{4} \mathbb{E}[\lambda^3 f(\xi_1)^2 f(\xi_2) g(\eta_2)^2 g(\eta_1) W_{11} W_{12} W_{22}  + \lambda^3 f(\xi_2)^2 f(\xi_1) g(\eta_2)^2 g(\eta_1) W_{12} W_{22} W_{21} \\
            &\quad\quad+ \lambda^3 f(\xi_2)^2 f(\xi_1) g(\eta_1)^2 g(\eta_2) W_{22} W_{21} W_{11} + \lambda^3 f(\xi_1)^2 f(\xi_2) g(\eta_1)^2 g(\eta_2) W_{21} W_{11} W_{12} \mid \xi_1, \xi_2] \\
            &= \frac{\lambda^3}{2} (f(\xi_1)^2 f(\xi_2) + f(\xi_2)^2 f(\xi_1)) G_2.
        \end{split}
    \end{equation*}
    This proves the result.
\end{proof}

\begin{lemma}
    We have $\mathbb{E}[h_6(Y_{\{1,2\},\{1,2\}}) \mid \eta_1, \eta_2] = \frac{\lambda^3}{2} F_2 (g(\eta_1)^2 g(\eta_2) + g(\eta_2)^2 g(\eta_1)) \alpha$.
\end{lemma}

\begin{proof}
    We have 
    \begin{equation*}
        \begin{split}
            \mathbb{E}&[h_4(Y_{\{1,2\},\{1,2\}}) \mid \eta_1, \eta_2] \\
            &= \frac{1}{4} \mathbb{E}[Y_{11} (Y_{11} - 1) Y_{22} + Y_{12} (Y_{12} - 1) Y_{21} + Y_{21} (Y_{21} - 1) Y_{12} + Y_{22} (Y_{22} - 1) Y_{11} \mid \eta_1, \eta_2] \\
            &= \frac{1}{4} \mathbb{E}[\mathbb{E}[Y_{11} (Y_{11} - 1) Y_{22} + Y_{12} (Y_{12} - 1) Y_{21} + Y_{21} (Y_{21} - 1) Y_{12} + Y_{22} (Y_{22} - 1) Y_{11} \mid \boldsymbol{\xi}, \boldsymbol{\eta}, \boldsymbol{W}] \mid \eta_1, \eta_2] \\
            &= \frac{1}{4} \mathbb{E}[\lambda^3 f(\xi_1)^2 f(\xi_2) g(\eta_1)^2 g(\eta_2) W_{11}^2 W_{22}  + \lambda^3 f(\xi_1)^2 f(\xi_2) g(\eta_2)^2 g(\eta_1) W_{12}^2 W_{21} \\
            &\quad\quad+ \lambda^3 f(\xi_2)^2 f(\xi_1) g(\eta_1)^2 g(\eta_2) W_{21}^2 W_{12} + \lambda^3 f(\xi_2)^2 f(\xi_1) g(\eta_2)^2 g(\eta_1) W_{22}^2 W_{11} \mid \eta_1, \eta_2] \\
            &= \frac{\lambda^3}{2} F_2 (g(\eta_1)^2 g(\eta_2) + g(\eta_2)^2 g(\eta_1)) (\alpha + 1),
        \end{split}
    \end{equation*}
    and 
    \begin{equation*}
        \begin{split}
            \mathbb{E}&[h_5(Y_{\{1,2\},\{1,2\}}) \mid \eta_1, \eta_2] \\
            &= \frac{1}{4} \mathbb{E}[Y_{11} Y_{12} Y_{22} + Y_{12} Y_{22} Y_{21} + Y_{22} Y_{21} Y_{11} + Y_{21} Y_{11} Y_{12} \mid \eta_1, \eta_2] \\
            &= \frac{1}{4} \mathbb{E}[\mathbb{E}[Y_{11} Y_{12} Y_{22} + Y_{12} Y_{22} Y_{21} + Y_{22} Y_{21} Y_{11} + Y_{21} Y_{11} Y_{12} \mid \boldsymbol{\xi}, \boldsymbol{\eta}, \boldsymbol{W}] \mid \eta_1, \eta_2] \\
            &= \frac{1}{4} \mathbb{E}[\lambda^3 f(\xi_1)^2 f(\xi_2) g(\eta_2)^2 g(\eta_1) W_{11} W_{12} W_{22}  + \lambda^3 f(\xi_2)^2 f(\xi_1) g(\eta_2)^2 g(\eta_1) W_{12} W_{22} W_{21} \\
            &\quad\quad+ \lambda^3 f(\xi_2)^2 f(\xi_1) g(\eta_1)^2 g(\eta_2) W_{22} W_{21} W_{11} + \lambda^3 f(\xi_1)^2 f(\xi_2) g(\eta_1)^2 g(\eta_2) W_{21} W_{11} W_{12} \mid \eta_1, \eta_2] \\
            &= \frac{\lambda^3}{2} F_2 (g(\eta_1)^2 g(\eta_2) + g(\eta_2)^2 g(\eta_1)).
        \end{split}
    \end{equation*}
    This proves the result.
\end{proof}

\begin{lemma}
    We have $\mathbb{E}[h_6(Y_{\{1,2\},\{1,2\}}) \mid \xi_1, \eta_1] = \frac{\lambda^3}{4} f(\xi_1) g(\eta_1) (f(\xi_1) g(\eta_1) + f(\xi_1) G_2 + F_2 g(\eta_1) + F_2 G_2) \alpha$.
\end{lemma}

\begin{proof}
    We have 
    \begin{equation*}
        \begin{split}
            \mathbb{E}&[h_4(Y_{\{1,2\},\{1,2\}}) \mid \xi_1, \eta_1] \\
            &= \frac{1}{4} \mathbb{E}[Y_{11} (Y_{11} - 1) Y_{22} + Y_{12} (Y_{12} - 1) Y_{21} + Y_{21} (Y_{21} - 1) Y_{12} + Y_{22} (Y_{22} - 1) Y_{11} \mid \xi_1, \eta_1] \\
            &= \frac{1}{4} \mathbb{E}[\mathbb{E}[Y_{11} (Y_{11} - 1) Y_{22} + Y_{12} (Y_{12} - 1) Y_{21} + Y_{21} (Y_{21} - 1) Y_{12} + Y_{22} (Y_{22} - 1) Y_{11} \mid \boldsymbol{\xi}, \boldsymbol{\eta}, \boldsymbol{W}] \mid \xi_1, \eta_1] \\
            &= \frac{1}{4} \mathbb{E}[\lambda^3 f(\xi_1)^2 f(\xi_2) g(\eta_1)^2 g(\eta_2) W_{11}^2 W_{22}  + \lambda^3 f(\xi_1)^2 f(\xi_2) g(\eta_2)^2 g(\eta_1) W_{12}^2 W_{21} \\
            &\quad\quad+ \lambda^3 f(\xi_2)^2 f(\xi_1) g(\eta_1)^2 g(\eta_2) W_{21}^2 W_{12} + \lambda^3 f(\xi_2)^2 f(\xi_1) g(\eta_2)^2 g(\eta_1) W_{22}^2 W_{11} \mid \xi_1, \eta_1] \\
            &= \frac{\lambda^3}{4} (f(\xi_1)^2 g(\eta_1)^2 + f(\xi_1)^2 G_2 g(\eta_1) + F_2 f(\xi_1) g(\eta_1)^2 + F_2 f(\xi_1) G_2 g(\eta_1)) (\alpha + 1),
        \end{split}
    \end{equation*}
    and 
    \begin{equation*}
        \begin{split}
            \mathbb{E}&[h_5(Y_{\{1,2\},\{1,2\}}) \mid \xi_1, \eta_1] \\
            &= \frac{1}{4} \mathbb{E}[Y_{11} Y_{12} Y_{22} + Y_{12} Y_{22} Y_{21} + Y_{22} Y_{21} Y_{11} + Y_{21} Y_{11} Y_{12} \mid \xi_1, \eta_1] \\
            &= \frac{1}{4} \mathbb{E}[\mathbb{E}[Y_{11} Y_{12} Y_{22} + Y_{12} Y_{22} Y_{21} + Y_{22} Y_{21} Y_{11} + Y_{21} Y_{11} Y_{12} \mid \boldsymbol{\xi}, \boldsymbol{\eta}, \boldsymbol{W}] \mid \xi_1, \eta_1] \\
            &= \frac{1}{4} \mathbb{E}[\lambda^3 f(\xi_1)^2 f(\xi_2) g(\eta_2)^2 g(\eta_1) W_{11} W_{12} W_{22}  + \lambda^3 f(\xi_2)^2 f(\xi_1) g(\eta_2)^2 g(\eta_1) W_{12} W_{22} W_{21} \\
            &\quad\quad+ \lambda^3 f(\xi_2)^2 f(\xi_1) g(\eta_1)^2 g(\eta_2) W_{22} W_{21} W_{11} + \lambda^3 f(\xi_1)^2 f(\xi_2) g(\eta_1)^2 g(\eta_2) W_{21} W_{11} W_{12} \mid \xi_1, \eta_1] \\
            &= \frac{\lambda^3}{4} (f(\xi_1)^2 G_2 g(\eta_1) + F_2 f(\xi_1) G_2 g(\eta_1) + F_2 f(\xi_1) g(\eta_1)^2 + f(\xi_1)^2 g(\eta_1)^2).
        \end{split}
    \end{equation*}
    This proves the result.
\end{proof}

\begin{lemma}
    We have 
    \begin{equation*}
    \begin{split}
        \mathbb{E}[h_6(Y_{\{1,2\},\{1,2\}}) \mid \xi_1, \eta_1, \zeta_{11}] &= \frac{\lambda^3}{4} f(\xi_1) g(\eta_1) (f(\xi_1) G_2 (\alpha + 1) + F_2 g(\eta_1) (\alpha + 1) - F_2 G_2) \\
        &\quad + \frac{\lambda^2}{4} Y_{11} (F_2 G_2 (\alpha + 1) - f(\xi_1) G_2 - F_2 g(\eta_1) - f(\xi_1) g(\eta_1)) \\
        &\quad + \frac{\lambda}{4} Y_{11} (Y_{11} - 1).
    \end{split}
    \end{equation*}
    \label{lem:degenerate:condexp_8}
\end{lemma}

\begin{proof}
    We have 
    \begin{equation*}
        \begin{split}
            \mathbb{E}&[h_4(Y_{\{1,2\},\{1,2\}}) \mid \xi_1, \eta_1, \zeta_{11}] \\
            &= \frac{1}{4} \mathbb{E}[Y_{11} (Y_{11} - 1) Y_{22} + Y_{12} (Y_{12} - 1) Y_{21} + Y_{21} (Y_{21} - 1) Y_{12} + Y_{22} (Y_{22} - 1) Y_{11} \mid \xi_1, \eta_1, \zeta_{11}] \\
            &= \frac{1}{4} \mathbb{E}[\mathbb{E}[Y_{11} (Y_{11} - 1) Y_{22} + Y_{12} (Y_{12} - 1) Y_{21} + Y_{21} (Y_{21} - 1) Y_{12} + Y_{22} (Y_{22} - 1) Y_{11} \mid \boldsymbol{\xi}, \boldsymbol{\eta}, \boldsymbol{W}, Y_{11}] \mid \xi_1, \eta_1, \zeta_{11}] \\
            &= \frac{1}{4} \mathbb{E}[\lambda Y_{11} (Y_{11} - 1) f(\xi_2) g(\eta_2) W_{22}  + \lambda^3 f(\xi_1)^2 f(\xi_2) g(\eta_2)^2 g(\eta_1) W_{12}^2 W_{21} \\
            &\quad\quad+ \lambda^3 f(\xi_2)^2 f(\xi_1) g(\eta_1)^2 g(\eta_2) W_{21}^2 W_{12} + \lambda^2 Y_{11} f(\xi_2)^2 g(\eta_2)^2 W_{22}^2 \mid \xi_1, \eta_1, \zeta_{11}] \\
            &= \frac{\lambda^3}{4} (f(\xi_1)^2 G_2 g(\eta_1) + F_2 f(\xi_1) g(\eta_1)^2 ) (\alpha + 1) + \frac{\lambda^2}{4} Y_{11} F_2 G_2 (\alpha + 1) + \frac{\lambda}{4} Y_{11} (Y_{11} - 1),
        \end{split}
    \end{equation*}
    and 
    \begin{equation*}
        \begin{split}
            \mathbb{E}&[h_5(Y_{\{1,2\},\{1,2\}}) \mid \xi_1, \eta_1, \zeta_{11}] \\
            &= \frac{1}{4} \mathbb{E}[Y_{11} Y_{12} Y_{22} + Y_{12} Y_{22} Y_{21} + Y_{22} Y_{21} Y_{11} + Y_{21} Y_{11} Y_{12} \mid \xi_1, \eta_1, \zeta_{11}] \\
            &= \frac{1}{4} \mathbb{E}[\mathbb{E}[Y_{11} Y_{12} Y_{22} + Y_{12} Y_{22} Y_{21} + Y_{22} Y_{21} Y_{11} + Y_{21} Y_{11} Y_{12} \mid \boldsymbol{\xi}, \boldsymbol{\eta}, \boldsymbol{W}, Y_{11}] \mid \xi_1, \eta_1, \zeta_{11}] \\
            &= \frac{1}{4} \mathbb{E}[\lambda^2 Y_{11} f(\xi_1) f(\xi_2) g(\eta_2)^2 W_{12} W_{22}  + \lambda^3 f(\xi_2)^2 f(\xi_1) g(\eta_2)^2 g(\eta_1) W_{12} W_{22} W_{21} \\
            &\quad\quad+ \lambda^2 Y_{11} f(\xi_2)^2 g(\eta_1) g(\eta_2) W_{22} W_{21} + \lambda^2 Y_{11} f(\xi_1) f(\xi_2) g(\eta_1) g(\eta_2) W_{21} W_{12} \mid \xi_1, \eta_1, \zeta_{11}] \\
            &= \frac{\lambda^3}{4} f(\xi_1) F_2 g(\eta_1) G_2 + \frac{\lambda^2}{4} Y_{11} (f(\xi_1) G_2 + F_2 g(\eta_1) + f(\xi_1) g(\eta_1)).
        \end{split}
    \end{equation*}
    This proves the result.
\end{proof}

\begin{lemma}
    In Running example~\ref{ex:3}, the asymptotic variance of $U^{h_6}_N$ under $\mathcal{H}_0 : \alpha = 0$ is given by 
    \begin{equation*}
        \sigma_6^2 = \frac{\lambda^4}{\rho (1-\rho)} \left[ \lambda (F_3 - F_2^2)(G_3 - G_2^2) + 2 F_2 G_2 \right].
    \end{equation*}
    \label{lem:degenerate:var_ex3}
\end{lemma}

\begin{proof}
    In Running example~\ref{ex:3}, the asymptotic variance of $U^{h_6}_N$ under $\mathcal{H}_0 : \alpha = 0$ is given by $(r,c) = (1,1)$ projection
    \begin{equation*}
        \sigma_6^2 = \frac{16}{\rho (1-\rho)} \lvert \Aut(K_{1,1}) \rvert^{-1} \mathbb{E}[(p^{K_{1,1}})^2] = \frac{16}{\rho (1-\rho)} \mathbb{E}[\mathbb{E}[h_6(Y_{\{1,2\}, \{1,2\}}) \mid \xi_1, \eta_1, \zeta_{11}]^2].
    \end{equation*}
    
    For brevity, write (only in this proof):
    \begin{equation*}
        (h, Y, \xi, \eta, \zeta, f, g) := \left(h_6(Y_{\{1,2\},\{1,2\}}), Y_{11}, \xi_1, \eta_1, \zeta_{11}, f(\xi_1), g(\eta_1)\right).
    \end{equation*}
    Under $\mathcal{H}_0 : \alpha = 0$ (Poisson-BEDD), Lemma~\ref{lem:degenerate:condexp_8} gives (set $\alpha = 0$) 
    \begin{equation*}
    \begin{split}
        H := \frac{4}{\lambda} \mathbb{E}[h \mid \xi, \eta, \zeta] = \lambda^2 f g (f G_2 + g F_2 - F_2 G_2) + \lambda Y (F_2 G_2 - f G_2 - g F_2 - f g) + Y (Y - 1).
    \end{split}
    \end{equation*}
    The desired quantity is therefore $\sigma_6^2 = \frac{\lambda^2}{\rho (1-\rho)} \mathbb{E}[H^2]$.
    
    Set $S := fg$, $T := f G_2 + g F_2 - F_2 G_2$, and $A := \lambda^2 ST$, $B:=- \lambda (S+T)$, so 
    \begin{equation*}
        H= A + B Y + Y (Y - 1).
    \end{equation*}
    Furthermore, we have $Y \mid \xi, \eta \sim \mathcal{P}(\lambda S)$ (Poisson-BEDD). Hence
    \begin{align*}
        &\mathbb{E}[Y \mid \xi, \eta] = \lambda S,  &  &\mathbb{E}[Y^2 \mid \xi, \eta] = \lambda^2 S^2 + \lambda S, \\
        &\mathbb{E}[Y^2 (Y-1) \mid \xi, \eta] = \lambda^3 S^3 + 2 \lambda^2 S^2,  &  &\mathbb{E}[Y^2 (Y-1)^2 \mid \xi, \eta] = \lambda^4 S^4 + 4 \lambda^3 S^3 + 2 \lambda^2 S^2.  
    \end{align*}
    Since $\mathbb{E}[A \mid \xi, \eta] = A$ and $\mathbb{E}[B \mid \xi, \eta] = B$, we develop
    \begin{align*}
        \mathbb{E}[H^2 \mid \xi, \eta] &= A^2 + 2 AB \mathbb{E}[Y  \mid \xi, \eta] + 2 A \mathbb{E}[Y(Y-1) \mid \xi, \eta] + B^2 \mathbb{E}[Y^2 \mid \xi, \eta] \\
        &\quad + 2 B \mathbb{E}[Y^2 (Y-1) \mid \xi, \eta] + \mathbb{E}[Y^2 (Y - 1)^2 \mid \xi, \eta] \\ 
        &= \lambda^4 S^2 T^2 - 2 \lambda^4 S^2 T (S + T) + 2 \lambda^4 S^3T + \lambda^4 S^2 (S + T)^2 + \lambda^3 S (S + T)^2 \\
        &\quad - 2 \lambda^4 S^3 (S + T) - 4 \lambda^3 S^2 (S + T) +  \lambda^4 S^4 + 4 \lambda^3 S^3 + 2 \lambda^2 S^2 \\
        &= \lambda^2 \left[ \lambda S (S - T)^2 + 2 S^2 \right].
    \end{align*}
    Using $S - T = (f-F_2)(g-G_2)$, we obtain
    \begin{align*}
        \mathbb{E}[H^2 \mid \xi, \eta] &= \lambda^2 \left[ \lambda fg( f - F_2)^2(g - G_2)^2 + 2 f^2 g^2 \right].
    \end{align*}
    Now, averaging over $(\xi, \eta)$:
    \begin{align*}
        \mathbb{E}[H^2] &= \mathbb{E}[\mathbb{E}[H^2 \mid \xi, \eta]] \\
        &= \lambda^2 \left( \lambda \mathbb{E}[ f (f - F_2)^2]\mathbb{E}[ g (g - G_2)^2] + 2 \mathbb{E}[ f^2]\mathbb{E}[ g^2] \right) \\
        &= \lambda^2 \left[ \lambda (F_3 - F_2^2)(G_3 - G_2^2) + 2 F_2 G_2 \right],
    \end{align*}
    and thus,
    \begin{align*}
        \sigma_6^2 = \frac{\lambda^4}{\rho (1-\rho)} \left[ \lambda (F_3 - F_2^2)(G_3 - G_2^2) + 2 F_2 G_2 \right].
    \end{align*}
\end{proof}

\section{Simulations}
\label{app:simulations}

In this section, we illustrate the asymptotic behavior of degenerate $U$-statistics through simulation results. Specifically, we consider the $U$-statistics used in Running examples~\ref{ex:2} and~\ref{ex:3}, which are relevant to network analysis by investigating row heterogeneity and Poisson overdispersion, respectively. We examine their empirical distribution as $N$ goes to $+\infty$.  

\paragraph{Running example~\ref{ex:2}: Row heterogeneity.} We consider the Poisson-BEDD network model, with $\lambda=1$, $f(u) = 1$ and $g(v) = (\alpha_g + 1) v^{\alpha_g}$, where $\alpha_g = 1+\sqrt{2}$. Under these settings, we have $F_2 = 1$ and $G_2 = 2$. 

The asymptotically normal $U$-statistic $U^{h_3}_N$ can be used to investigate the row heterogeneity of observed networks. Its asymptotic variance is $\sigma_3^2 = 2 \lambda^2/(\rho(1-\rho)^2)$. The statistic 
\begin{equation*}
    Z^B_N = N^{3/2} \sigma_3^{-1} U^{h_3}_N
\end{equation*}
converges to a standard normal distribution.

$\sigma_3^2$ can also easily be estimated using the $U$-statistic $U^{h_2}_N$, which correspond to the kernel $h_2$ defined in Running example~\ref{ex:1}. Notably, $U^{h_2}_N$ is an unbiased and consistent estimator for $\lambda^2$. According to Slutsky's theorem, the statistic 
\begin{equation*}
    Z^{B'}_N = N^{3/2} (U^{h_2}_{N})^{-1/2}U^{h_3}_N
\end{equation*}
converges to a standard normal distribution.

For $\rho = 1$, $N \in \{N^i : i \in \llbracket 3, 8 \rrbracket \}$, we simulated $K = 500$ adjacency matrices under the Poisson-BEDD model with the specified parameters, and $\rho = 1/2$. In particular, we can modify the value of $F_2$ by setting $\alpha_f$ accordingly. Figure~\ref{fig:degeneracy_f2} displays the Q-Q plots of $Z^B_N$ under $\mathcal{H}_0$, illustrating its quick convergence to a standard normal distribution. Figure~\ref{fig:degeneracy_f2_alt} gives its behavior under the alternative $F_2 = 1.2$. 

\begin{figure}[!tb]
\centering
\includegraphics[width=0.8\linewidth]{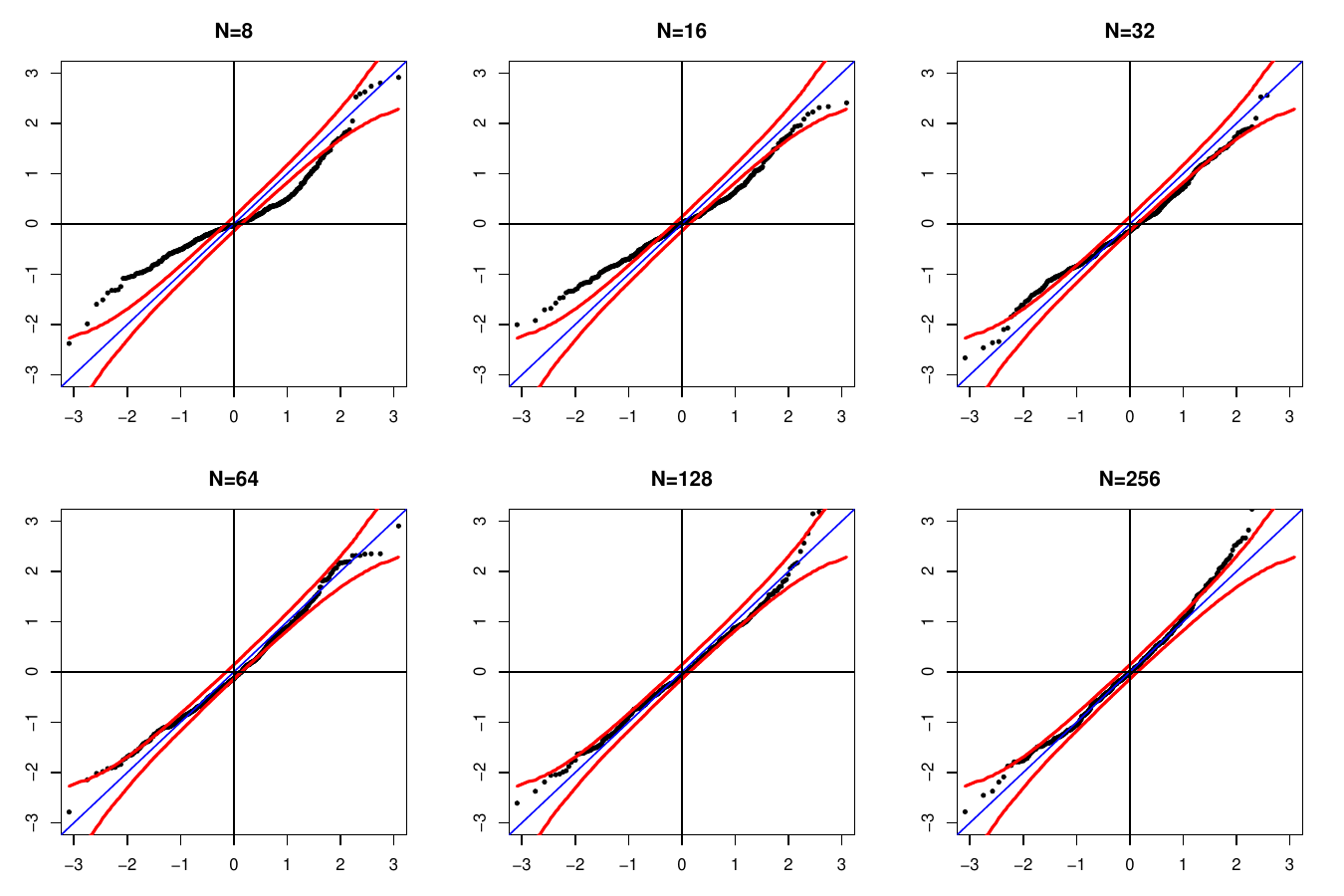}
\caption[Q-Q plots for Running example~\ref{ex:2}]{These Q-Q plots show convergence of $Z^B_N$ to a standard normal distribution under $\mathcal{H}_0$. The red lines represent the $99\%$-confidence envelope. }
\label{fig:degeneracy_f2}
\end{figure}
\begin{figure}[!tb]
\centering
\includegraphics[width=0.8\linewidth]{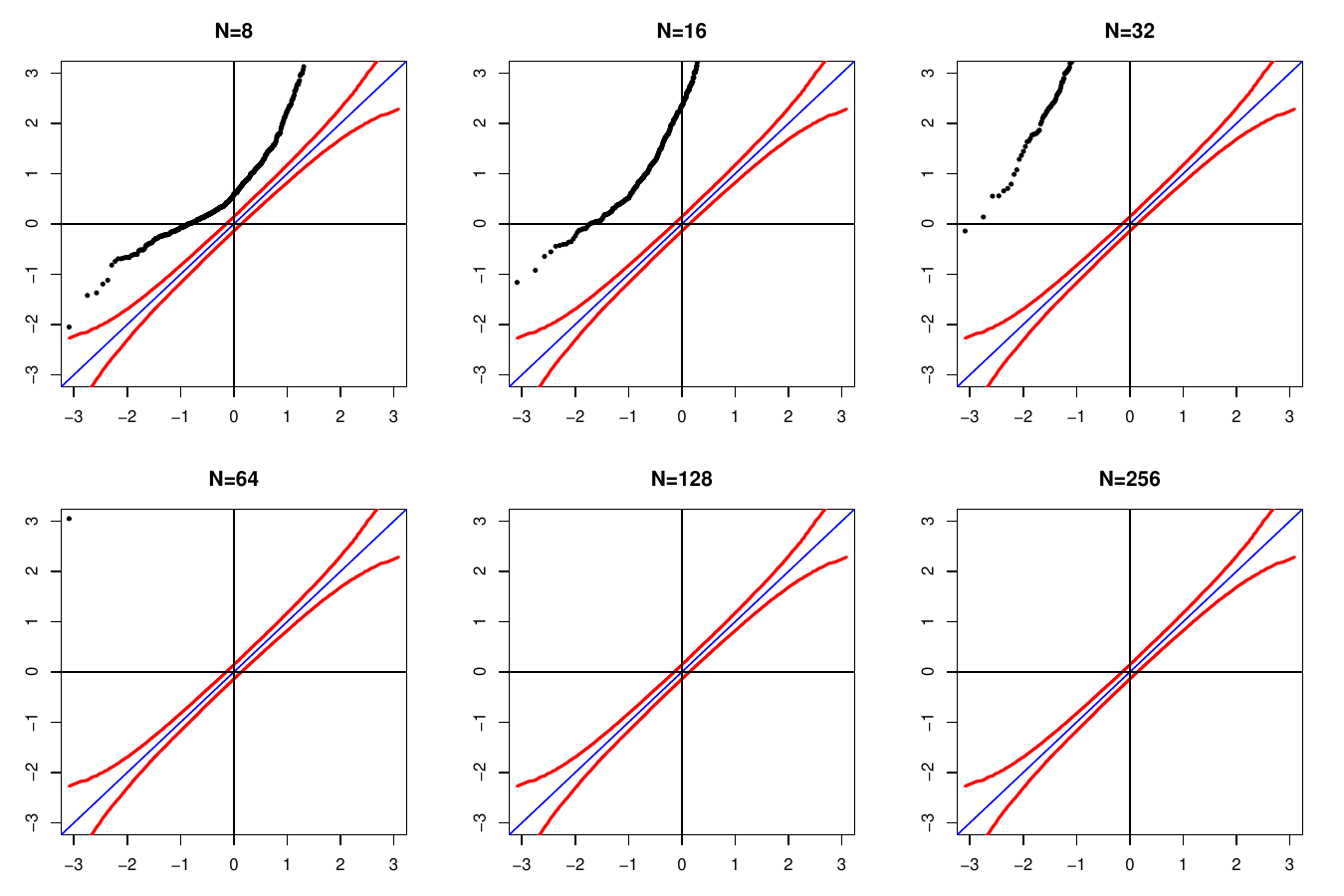}
\caption[Q-Q plots for Running example~\ref{ex:2}]{These Q-Q plots show deviation of $Z^B_N$ from a standard normal distribution under $\mathcal{H}_1$, here $F_2 = 1.2$. The red lines represent the $99\%$-confidence envelope. }
\label{fig:degeneracy_f2_alt}
\end{figure}

\paragraph{Running example~\ref{ex:3}: Overdispersion beyond row-column effects.} We consider the Overdispersed-Poisson-BEDD network model, with $\lambda=1$, $f(u) = 2u$, $g(v) = 2v$ and $\alpha = 0$. In this case, we have $F_2 = G_2 = 4/3$. 

The asymptotically normal $U$-statistic $U^{h_6}_N$ can be used to investigate the overdispersion of the count data represented by the networks. The asymptotic variance $\sigma_6^2$ is given by Lemma~\ref{lem:degenerate:var_ex3}. The statistic 
\begin{equation*}
    Z^C_N = N \sigma_6^{-1} U^{h_6}_N
\end{equation*}
converges to a standard normal distribution.

For $\rho = 1$, $N \in \{N^i : i \in \llbracket 3, 8 \rrbracket \}$, we simulated $K = 500$ adjacency matrices under the Overdispersed-Poisson-BEDD model with the specified parameters, and $\rho = 1/2$. Figure~\ref{fig:degeneracy_over} displays the Q-Q plots of $Z^C_N$, illustrating its quick convergence to a normal distribution with variance 1. 

\begin{figure}[!tb]
\centering
\includegraphics[width=0.8\linewidth]{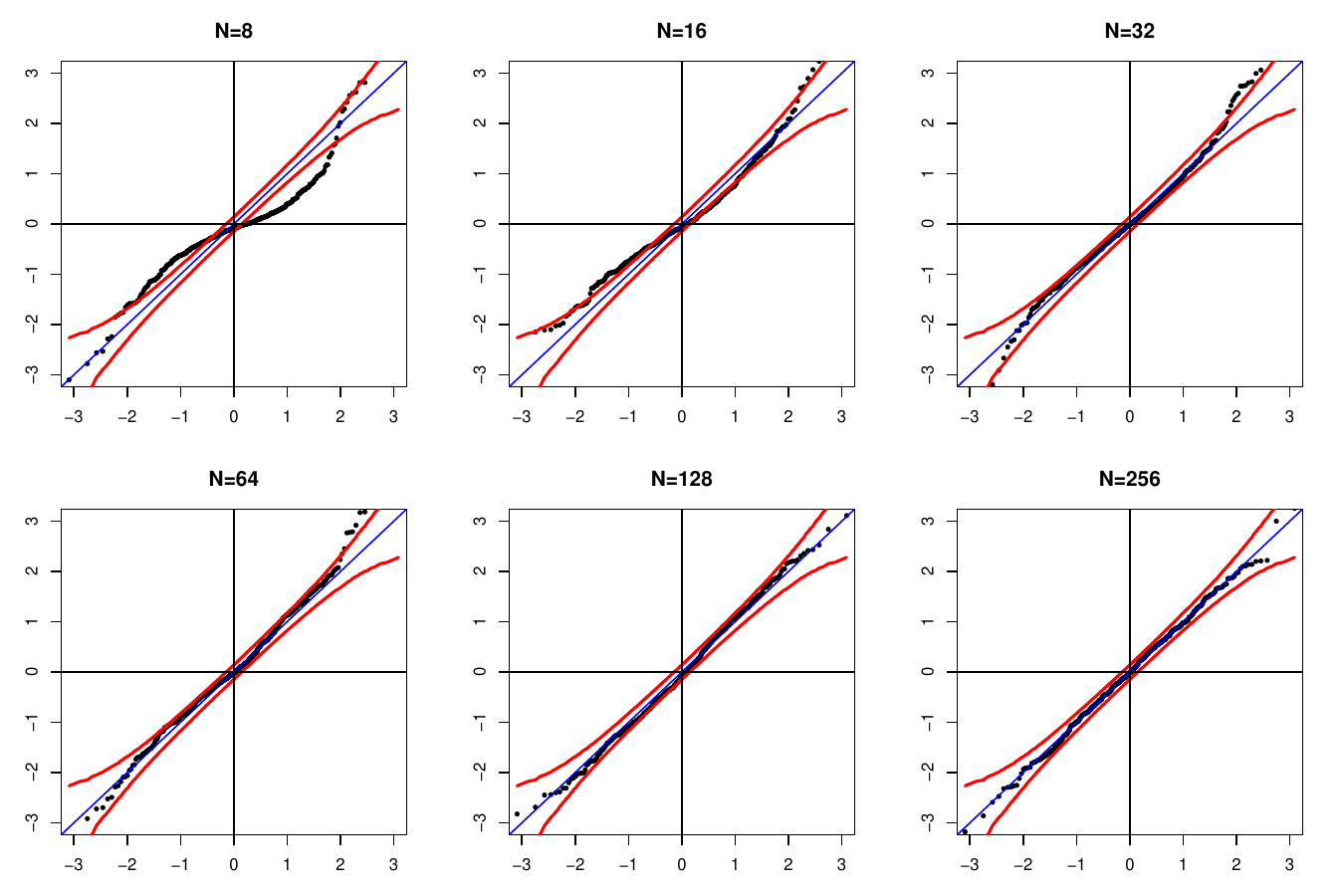}
\caption[Q-Q plots for Running example~\ref{ex:3}]{These Q-Q plots show convergence of $Z^C_N$ to a standard normal distribution under $\mathcal{H}_0$. The red lines represent the $99\%$-confidence envelope.}
\label{fig:degeneracy_over}
\end{figure}

\begin{figure}[!tb]
\centering
\includegraphics[width=0.8\linewidth]{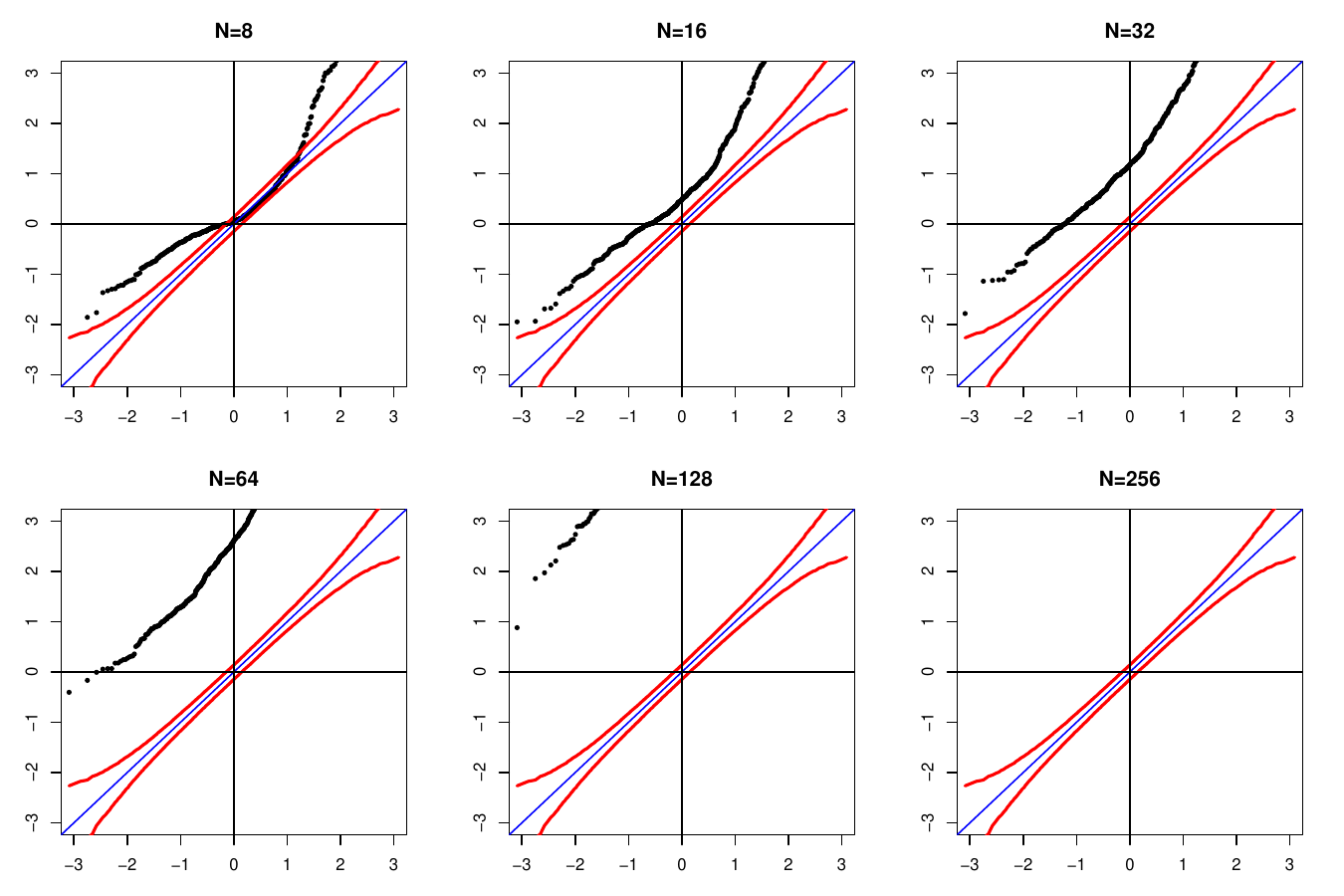}
\caption[Q-Q plots for Running example~\ref{ex:3}]{These Q-Q plots show deviation of $Z^C_N$ from a standard normal distribution under $\mathcal{H}_1$, here $\alpha = 0.1$. The red lines represent the $99\%$-confidence envelope.}
\label{fig:degeneracy_over}
\end{figure}

\end{appendix}

\newpage

\bibliographystyle{apalike}
\bibliography{biblio.bib}

\end{document}